\documentclass[10pt,paper=a4,oneside,notitlepage]{report}

 \usepackage{colortbl}
\usepackage{wrapfig}
\usepackage[T1]{fontenc}
\usepackage{titlesec, blindtext, color}
\usepackage{typearea}
\areaset[6mm]{140mm}{230mm}
\usepackage[paper=a4paper,left=32mm,right=32mm,top=40mm,bottom=40mm]{geometry}
\usepackage{caption}
\usepackage{subcaption}
\usepackage[hyphens]{url}
\usepackage{changepage}
\usepackage[utf8]{inputenc}	% Wir wollen UTF-8 (=keine Probleme mit Umlauten etc.)
\usepackage{ae,aecompl} %bessere Schrift
\usepackage{amsmath,nccmath}
\usepackage{amsthm}
\usepackage{esint}        % Mittelwert Integral
\usepackage{amssymb} 
\usepackage[ngerman,british]{babel}  	% Deutsches Wörterbuch usw.
\usepackage{graphicx}	% wird für Figur Umgebung benötigt
\usepackage{epstopdf}   % Wandelt .eps Dateien automatisch um
\usepackage{url}	% für URL mit \url{.....}
\usepackage{bbm}	% für einheitsmatrix \mathbbm{1}
\usepackage{listings}	% für code (gnuplot, C etc.)
\usepackage[font=small,labelfont=bf]{caption}  % Optionen für Bild- und Codeunterschriften
\usepackage[hidelinks]{hyperref}         % damit Links in der PDF anklickbar werden
\usepackage{booktabs}	% bessere Tabellen mit Abstand zur hline
\usepackage{longtable}
\usepackage{float}
\usepackage{bibgerm}	% Bibstyle deutsches Literaturverzeichnis
\usepackage[dvipsnames]{xcolor}
\usepackage{tikz}
\usetikzlibrary{shapes.misc,shadows} % Für Schatten unter der box
\usepackage{stmaryrd}
\usepackage{tcolorbox}
\usepackage{tikz}              %Diagramme
\usepackage{multirow,tabularx}
\usepackage{changepage}
\usepackage{nameref} % zitiert Theorem, Formel etc. mit Namen
\usepackage{pifont} %Herz
\usepackage{theoremref} 
\usepackage{etoolbox}    %Kapitel beginnnt auf derselben Seite
\usepackage{fixmath} % for \mathbold
\usepackage[misc]{ifsym}
\usepackage{mathrsfs}
\usepackage{mathtools}  % multiline x=lllllll
\usepackage{hhline}

\usepackage{chngcntr}
\counterwithout{footnote}{chapter}
\allowdisplaybreaks

\makeatletter
\patchcmd{\chapter}{\if@openright\cleardoublepage\else\clearpage\fi}{}{}{} %Kapitel beginnnt auf derselben Seite
\makeatother

\theoremstyle{break}

\newtheorem{theorem}{Theorem} [chapter]
\newtheorem{defi}[theorem]{Definition} 
\newtheorem{prop}[theorem]{Proposition}

\newtheorem{lemma}[theorem]{Lemma}
\newtheorem{remark}[theorem]{Remark}

\allowdisplaybreaks
\numberwithin{equation}{chapter}
\usepackage{tocloft}
\cftsetindents{chapter}{0em}{2em}
\cftsetindents{section}{0em}{3em}
\cftsetindents{subsection}{0em}{4em}
\cftsetindents{subsubsection}{0em}{5em}

\newcommand*{\toccontents}{\@starttoc{toc}}

\makeatletter
\def\@makechapterhead#1{%
  \vspace*{50\p@}%
  {\parindent \z@ \raggedright \normalfont
    \interlinepenalty\@M
  \center  \LARGE\bfseries  \thechapter. #1\par\nobreak %Größe, fett,...
    \vskip 20\p@                                       % Abstand Kapitelüberschrift zu nachfolgendem Text
  }}
 
  \titleformat{\section}
  {\normalfont\bfseries\Large}{\thesection}{1em}{} % Statt scshape-> \bfseries?

  \titleformat{\subsection}
  {\normalfont\bfseries\large}{\thesubsection}{1em}{}

\addto{\captionsbritish}{%
  
}

\addto{\captionsbritish}{%
  
}

\makeatother %Wörter verlinken

\makeatletter
\newcommand{\subalign}[1]{%
  \vcenter{%
    \Let@ \restore@math@cr \default@tag
    \baselineskip\fontdimen10 \scriptfont\tw@
    \advance\baselineskip\fontdimen12 \scriptfont\tw@
    \lineskip\thr@@\fontdimen8 \scriptfont\thr@@
    \lineskiplimit\lineskip
    \ialign{\hfil$\m@th\scriptstyle##$&$\m@th\scriptstyle{}##$\hfil\crcr
      #1\crcr
    }%
  }%
}
\makeatother

\binoppenalty=\maxdimen %verhindert Zeilenumbrüche in $ $-Formeln
\relpenalty=\maxdimen

\newenvironment{leftalign*}[1][\parindent]{\setlength\hangindent{#1}\start@align\tw@\st@rredtrue\m@ne}{\endalign}
\DeclareMathOperator{\divs}{div}

\DeclareMathOperator{\Id}{Id}
\DeclareMathOperator{\tr}{tr}
\DeclareMathOperator{\supp}{supp}

\newcommand{\axi}{\ensuremath{a_{(\xi)}} \xspace}

  \newcommand{\wpr}{\ensuremath{w_{q+1}^{(p)}} \xspace} 
  \newcommand{\wc}{\ensuremath{w_{q+1}^{(c)}} \xspace}
  \newcommand{\wt}{\ensuremath{w_{q+1}^{(t)}} \xspace}

\newcommand{\oRlin}{\ensuremath{\mathring{R}_\text{lin}} \xspace}  
\newcommand{\oRnonlin}{\ensuremath{\mathring{R}_\text{nonlin1}} \xspace} 
\newcommand{\oRnonlinn}{\ensuremath{\mathring{R}_\text{nonlin2}} \xspace} 
\newcommand{\oRcom}{\ensuremath{\mathring{R}_\text{com1}} \xspace} 
\newcommand{\oRcomm}{\ensuremath{\mathring{R}_\text{com2}} \xspace} 
  
\newcommand{\oRosc}{\ensuremath{\mathring{R}_\text{osc}} \xspace}  
\newcommand{\oRcor}{\ensuremath{\mathring{R}_\text{cor}} \xspace} 
\newcommand{\oRq}{\ensuremath{\mathring{R}_q} \xspace} 
\newcommand{\oRqq}{\ensuremath{\mathring{R}_{q+1}} \xspace}

 \newcommand{\phixi}{\ensuremath\phi_{(\xi)}} 
  \newcommand{\Phixi}{\ensuremath\Phi_{(\xi)}}
  \newcommand{\psixi}{\ensuremath\psi_{(\xi)}}
 
 \newcommand{\Wxi}{\ensuremath{W_{(\xi)}} \xspace}
 \newcommand{\Wcxi}{\ensuremath{W^{(c)}_{(\xi)}} \xspace}
 \newcommand{\Vxi}{\ensuremath{V_{(\xi)}} \xspace}
 
 \newcommand{\Il}{\ensuremath{{I_{\ell}}} \xspace}

\newcommand{\ootimes}{\ensuremath\mathring{\otimes}}
\newcommand{\invdiv}{\ensuremath\mathcal{R}}
\newcommand{\Pm}{\ensuremath\mathbb{P}_{\neq 0}}
\newcommand{\E}{\ensuremath\mathbf{E}}

\begin{document}

\begin{center}
\boldmath
\LARGE{\textbf{Existence and Non-Uniqueness of Ergodic Leray--Hopf Solutions to the Stochastic Power-Law Flows}}
\unboldmath
\end{center}
~\\
\begin{center}
\large{Stefanie Elisabeth Berkemeier}\\
\end{center}
\begin{center}
\today
\end{center}
\,\\
\begin{adjustwidth}{25pt}{25pt}
\textbf{Abstract.}
We study long time behavior of shear-thinning fluid flows in $d \geq 3$ dimensions, driven by additive stochastic forcing of trace class, with power-law indices ranging from $1$ to $  \frac{2d}{d+2}$. We particularly focus on Leray--Hopf solutions, i.e. on analytically weak solutions satisfying energy inequality.\\ Introducing a new kind of energy related functional into the technique of convex integration enables the construction of infinitely many such solutions that are probabilistically strong for a certain initial value. Furthermore, we provide global in time estimates which lead to the existence of infinitely many stationary and even ergodic Leray--Hopf solutions.\\
These results represent the first construction of Leray--Hopf solutions in the framework of stochastic shear-thinning fluids within this range of power-law indices.
\end{adjustwidth}
~
\tableofcontents

%1.Chapter on same page as table of contents
{\let\clearpage\relax

\chapter{Introduction}\label{Introduction}

\section{Motivation and Previous Works} 
The Navier--Stokes equations, a cornerstone of fluid dynamics, continue to captivate and challenge researchers with their intricate behavior and unresolved complexities. Despite their fundamental role in describing incompressible fluid flows, the existence and uniqueness of strong solutions to these equations remain one of the seven unsolved Millennium Prize Problems. Nevertheless, significant progress has been made in the study of weak solutions, yielding profound insights into fluid dynamics.\\
The most significant recent successes are undeniably the establishment of non-uniqueness of weak solutions to the incompressible Navier--Stokes equations by Buckmaster and Vicol \cite{BV19a, BV19b}, as well as non-uniqueness of Leray--Hopf solutions to the forced Navier--Stokes equations by Albritton, Brué and Colombo \cite{ABC22}.\\
 The proof of the former result relies on the method of convex integration, which is an iterative procedure for constructing solutions to various deterministic and stochastic PDEs, building upon the ideas of Nash and Kuiper \cite{Na54, Ku55a, Ku55b}. Subsequent developments by Müller and Šverák \cite{MS03} refined these insights, leading to new techniques and applications in fluid dynamics. Their work paved the way for the establishment of existence of infinitely many weak solutions to the incompressible Euler equations which dissipate the total kinetic energy and satisfy the global and local energy inequality \cite{DLS09, DLS10, DLS13}. In 2016, Isett achieved a major breakthrough by proving  the long-standing open Onsager's conjecture \cite{Is18}. Building on these advancements, Buckmaster, De Lellis, Székelyhidi, and Vicol made further strides by showing that these solutions dissipate kinetic energy \cite{BDLSV18}.\medskip\\
 In recent years, the technique of convex integration has been successfully adapted to the stochastic setting by Breit, Feireisl, Hofmanov\'a \cite{BFH20}, leading to a series of groundbreaking results in the study of fluid dynamics under randomness. Hofmanová, Zhu, and Zhu have been at the forefront of this development, pioneering the application of convex integration to stochastic partial differential equations. Among their significant contributions is the establishment of non-uniqueness in law for the $3$D stochastic Navier--Stokes equations \cite{HZZ24a}, a result with profound implications to the field of stochastic fluid dynamics.\\
Building on this foundation, they further advanced the field by establishing the existence of infinitely many dissipative martingale solutions to the $3$D stochastic Euler equations \cite{HZZ21}. This was followed by additional breakthroughs, including the global existence and non-uniqueness of solutions to the $3$D stochastic Navier--Stokes equations driven by space-time white noise and transport noise \cite{HZZ23b, HLP24}.\\
Moreover, the existence of infinitely many analytically weak yet probabilistically strong solutions to the $3$D Navier--Stokes equations with prescribed energy, under perturbations by additive and multiplicative noise was proven in \cite{HZZ23c, Be23}. Hofmanová, Zhu, and Zhu also provided rigorous proofs of non-uniqueness of ergodic solutions for both deterministic and stochastic Navier--Stokes equations as well as for Euler equations \cite{HZZ25}. Additionally, recent work has established non-uniqueness in law of $2$D Navier--Stokes equations with diffusion weaker than a full Laplacian \cite{Ya22} and of $3$D Navier--Stokes equations diffused via a fractional Laplacian with power less than one half \cite{Ya24}. \\
These developments have significantly advanced the understanding of the interplay between randomness and fluid dynamics, and they continue to shape the direction of current research in the field.\medskip \\
It was only a matter of time before the method of convex integration was extended to a wider range of PDEs in both deterministic and stochastic settings. Notable examples include its application to the primitive and Prandtl equations \cite{BMT24}, the stochastic hypodissipative Navier--Stokes equations \cite{RS23}, the stochastic transport equation with transport noise \cite{MS24} as well as the stochastic surface quasi-geostrophic equations with irregular spatial perturbations, generic additive noise, multiplicative noise and derivatives of space-time white noise \cite{HZZ23a,BLW24,Ya23,HLZZ24}.\medskip \\
These models, however, do not address non-Newtonian fluids, such as toothpaste, shampoo, or blood, whose viscosity varies with the shear rate. The power-law model, introduced by Waele, Ostwald, and Norton in 1928/1929, offers a framework for capturing such variable viscosity behavior \cite{No29, Wa23, Os29}.  While these considerations were initially of chemical interest, Ladyshenskaya and Lions rigorously investigated these phenomena from a mathematical perspective for the first time \cite{La68, La69, Li69}. Building on monotone operator theory and conventional compactness techniques, they established the existence of weak solutions to the power-law equations for power indices $\iota > \frac{3d+2}{d+2}$ that comply with an energy inequality. This result was subsequently improved by lowering the threshold through $L^\infty$-truncation to the case $\iota>\frac{2d+2}{d+2}$~\cite{Wo07}. Further advancements feature refinements through Lipschitz truncation by Diening, R\r{u}\v{z}i\v{c}ka, and Wolf, who extended the result for power indices greater than $\frac{2d}{d+2}$~\cite{DRW10}. While the aforementioned methods break down for indices below this range, convex integration, by contrast, offers an effective approach to overcome these difficulties. Employed by Burczak, Modena, and Székelyhidi, this technique enabled the establishment of weak solutions for power-law fluids with power indices $\iota < \frac{2d}{d+2}$~\cite{BMS21}. \medskip\\
From a physical perspective, the identification of weak solutions that satisfy such an energy inequality, commonly referred to as Leray--Hopf solutions, is a fundamental aspect of fluid dynamics. It ensures that energy dissipation is controlled over time, which prevents unbounded growth and guarantees the long-term stability of the system. Within the context of convex integration, this is particularly challenging for viscous equations and has been achieved only if the dissipation is not too strong. In addition to the power-law equations, this technique was also employed in the study of the hypodissipative Navier–Stokes equations, where Colombo, De Lellis, and De Rosa were able to demonstrate the ill-posedness of such solutions when the power of the Laplacian is less than $1/5$ \cite{CDLDR18}.
More recently, significant progress has been made in the study of the three-dimensional forced Navier--Stokes equations and the two-dimensional hypodissipative Navier--Stokes equations, with the non-uniqueness of Leray solutions rigorously demonstrated in both cases \cite{ABC22, AC22}. \medskip\\
The former work was further extended to the stochastic setting, where pathwise non-uniqueness and non-uniqueness in law were both established \cite{HZZ24b, BJYZ23}. First results concerning the existence of Leray--Hopf solutions within the context of stochastic power-law equations, were obtained by Breit \cite{Br15}. Extending the technique of $L^\infty$-truncation and harmonic pressure decomposition, he constructed such solutions even for indices $\iota > \frac{2d+2}{d+2}$. Similarly to the deterministic framework, Lü and Zhu followed the methodology of convex integration in the stochastic context to lower this threshold below~$\frac{3d+2}{d+2}$~\cite{LZ23}.
These solutions, however, do not satisfy the energy inequality, which is the problem addressed in our present study for power indices~$ \iota<\frac{2d}{d+2}$. The question of whether such solutions exist within the intermediate range $\left(\frac{2d}{d+2},\frac{2d+2}{d+2}\right)$ remains an open problem, with Lipschitz truncation offering a compelling and promising avenue for future exploration.

\pagebreak

\section{Main Result}
In the present work, we construct infinitely many ergodic Leray--Hopf solutions to the stochastic power-law system driven by additive noise. This constitutes the first construction of weak solutions to stochastic viscous fluid equations that satisfy the energy inequality within the range~$\iota <\frac{2d}{d+2}$.\\
Since conventional compactness techniques, the $L^\infty$-truncation, and Lipschitz truncation have proven inadequate in the deterministic setting for power indices $\iota<\frac{2d}{d+2}$, we follow the ideas of Burczak, Modena, and Székelyhidi, of Lü and Zhu as well as of Hofmanová, Zhu, and Zhu \cite{BMS21,LZ23,HZZ25}, and apply the method of convex integration to overcome these obstacles in the stochastic framework. Specifically, we iteratively construct a sequence $(u_q)_{q\in \mathbb{N}_0}$, solving the power-law equations at each level $q\in \mathbb{N}_0$ up to an error term~$\mathring{R}_q$. While the error term, referred to as the Reynolds stress, becomes infinitesimally small as $q \to \infty$, the sequence of velocities is expected to converge to the desired solution $u$.\medskip\\
One of the central challenges in this construction is avoiding the blow-up of iterative estimates, especially when trying to directly incorporate the energy into the scheme. To address this, we introduce an auxiliary energy functional, which is carefully designed based on a thorough analysis of its behavior during the convex integration process. The precise formulation of this auxiliary energy is crucial, as even minor adjustments could lead to divergence or a failure of the energy inequality. \\
Compared to the approach in \cite{BMS21}, we face difficulties related to the stochastic nature of the system as well as to the fact that we are aiming for the long time behavior. For the latter reason, we needed to introduce the auxiliary energy functional and could not simply prescribe the energy as it was done in \cite{BMS21}. Unlike the method in \cite{LZ23}, which only leads to solutions defined up to a stopping time, we follow the strategy employed in \cite{HZZ25}, incorporating expectations directly into the iterative estimates, allowing us to achieve global in time estimates. While the core calculations remain fundamentally similar, the presence of the non-Newtonian tensor in the power-law equations necessitates enhanced regularity of solutions to guarantee their ergodicity. This, however, requires controlling higher-order derivatives in the iterative estimates, distinguishing our approach from \cite{HZZ25}.\\
The introduction of the auxiliary energy functional not only enables the construction of ergodic solutions but also ensures that these solutions satisfy the energy inequality and are therefore Leray--Hopf, as desired.  \\
\bigskip\\
In the sequel, we present our result in more detail; the model of non-Newtonian flows perturbed by an additive noise in $d\geq 3$ dimensions reads rigorously as 
\begin{align}
\begin{split}
du+\divs(u \otimes u) \,dt-\divs(\mathcal{A}(Du))\,dt+\nabla P \,dt&= \,dB_t, \\
\divs u&=0, \label{PLF}
\end{split}
\end{align}
 on $\Omega \times [0,\infty)\times\mathbb{T}^d $, where $\mathbb{T}^d= [-\pi,\pi]^d $ denotes the $d$-dimensional torus.\\ 
The system governs the time evolution of the velocity $u\colon \Omega\times [0,\infty)\times \mathbb{T}^d   \to \mathbb{R}^d$ of a non-Newtonian fluid for some unknown pressure $P\colon \Omega \times  [0,\infty)\times \mathbb{T}^d \to \mathbb{R}$. If the power index $\iota\in(1,\infty)$ of the non-Newtonian tensor\footnote{See \thref{Lemma 2.4} for more information.}  
\begin{align}
\mathcal{A}(Q):=\Big(\nu_0+\nu_1\|Q\|_F\Big)^{\iota-2}Q \label{Non-Newtonian-Tensor}
\end{align}
with  $Q \in\mathbb{R}^{d\times d} $, is strictly smaller than $2$ for some $\nu_0,\, \nu_1 \geq 0$, it models the flow of shear-thinning fluids, whereas the flow of shear-thickening fluids are covered for $\iota>2$. In the special case $\iota=2$ the system reduces to the stochastic Navier--Stokes equations, describing the motion of incompressible Newtonian fluids. \\
The strain-rate tensor also known as symmetric gradient $D$ is for all $i,j=1,\ldots, d$ defined as
\begin{align*}
Du:= \frac{1}{2}\Big((\nabla u^T)^T+ \nabla u^T\Big)=\frac{1}{2}\Big(\partial_{x_j}u_i+\partial_{x_i}u_j \Big)_{ij},
\end{align*}
whereas $\left(B_t\right)_{t \geq 0}$ is a mean and divergence free $G$-Wiener process on an given probability space $\left( \Omega, \mathcal{F}, \mathcal{P}\right)$. $G$ is a positive, selfadjoint operator of trace-class on $L_{\rm{free}}^2\left(\mathbb{T}^d;\mathbb{R}^d\right)$\footnote{The space $L_{\rm{free}}^2\left(\mathbb{T}^d;\mathbb{R}^d\right)$ contains all functions in $L^2\left(\mathbb{T}^d;\mathbb{R}^d\right)$ that are mean and divergence free.}, satisfying $\tr\big((\Id-\Delta)^{2\sigma}G\big)<\infty$ for some $\sigma>0$. The required regularity of $G$ will be of important meaning in Section \ref{Factorization}. It will ensure well definedness of the decomposed PDE (cf. \thref{Proposition SHE}). Moreover, we denote by $\left(\mathcal{F}_t\right)_{t\geq 0}$ the normal filtration generated by~$\left(B_t\right)_{t \geq 0}$. \medskip \\
It is quite natural to consider an energy functional $\mathscr{E}$ on 
\begin{align*}
L^2\big(\Omega;L^\infty\big([0,\infty);L^2\big(\mathbb{T}^d;\mathbb{R}^d \big)\big)\big)\cap L^\iota\big(\Omega;
 L^\iota\big([0,\infty);W^{1,\iota}\big(\mathbb{T}^d;\mathbb{R}^d\big) \big)\big)
\end{align*}
that consists of the mean kinetic energy $\E\|u\|_{L_x^2}^2$, the mean internal energy\footnote{$\langle \cdot , \cdot  \rangle_{L^2}$ denotes the inner product in $L^2(\mathbb{T}^d; \mathbb{R}^d)$, i.e. $\langle f,g \rangle_{L^2}:= \int_{\mathbb{T}^d}f(x)g(x) \, dx$, for any $f,g \in L^2(\mathbb{T}^d; \mathbb{R}^d)$.} \linebreak $\E\int_0^\cdot\langle\divs\big(\mathcal{A}\big(Du(s)\big)\big),u(s)\rangle_{L^2(\mathbb{T}^d)}\,ds$\footnote{In view of \thref{Lemma 2.4} the expression $\E\int_0^\cdot\langle\divs\big(\mathcal{A}\big(Du(s)\big)\big),u(s)\rangle_{L^2(\mathbb{T}^d)}\,ds$ is well defined within the class of solutions obtained in \thref{Main Result}.} and a correction term. More precisely we define $\mathscr{E}$ as\footnote{$:$ denotes the Frobenius inner product, defined by $A : B=\sum_{i,j} A_{ij}B_{ij}$ for two matrices $A,\,B$.} 
 \begin{align}
\mathscr{E}\{u\}(t):=\E\|u(t)\|_{L_x^2}^2+2\E\bigg[\int_0^{t} \int_{\mathbb{T}^d} \mathcal{A}(Du(s,x)) : Du (s,x)\,dx \,ds\bigg]-\tr(G)t\label{Energy Functional}
\end{align}
 at every time $t\in [0,\infty)$.
 \bigskip

\noindent In this paper we are concerned with solutions in the following sense.

\begin{defi} \thlabel{Definition Weak Solution}
An $(\mathcal{F}_t)_{t \geq 0}$-adapted stochastic process $u$ on $\big((\Omega,\mathcal{F},\mathcal{P}),(\mathcal{F}_t)_{t \geq 0},B \big)$ is said to be an analytically weak solution to the power-law system \eqref{PLF} with the power-law index $\iota\in (1,\infty)$, if
\begin{itemize}
\item[i)] it belongs to $L^\infty\big([0,\infty);L^2\big(\mathbb{T}^d;\mathbb{R}^d \big)\big)\cap L^\iota\big([0,\infty);W^{1,\iota}\big(\mathbb{T}^d;\mathbb{R}^d\big)\big)\, \mathcal{P}$-a.s., 
\item[ii)] it is weakly divergence free, i.e. it obeys 
\begin{align*}
\int_{\mathbb{T}^d} \left( u( t,x)\cdot \nabla \right)\phi(x)\,dx=0 
\end{align*}
$\mathcal{P}$-a.s for all $t \geq 0$ and every test function $\phi\in C^\infty\left(\mathbb{T}^d;\mathbb{R}\right)$,
\item[iii)] it satisfies the weak formulation of the non-Newtonian flow equations 
\begin{align} \label{Weak Formulation}
\begin{split}
 \int_{\mathbb{T}^d}  \varphi(x) B_t(x) \,dx 
 &= \int_{\mathbb{T}^d}   \Big(u(t,x)-u(0,x)\Big)\cdot\varphi(x)\,dx\\
 &\hspace{0.7cm}-\int_0^t \int_{\mathbb{T}^d} (u \otimes u)( s,x): \nabla \varphi^T(x)\,dx \,ds\\
 &\hspace{.7cm}+ \int_0^t \int_{\mathbb{T}^d}  \mathcal{A}(Du(s,x)): \nabla \varphi^T(x)\,dx\,ds 
\end{split}
\end{align}
$\mathcal{P}$-a.s. for all divergence free $\varphi\in C^\infty\left(\mathbb{T}^d;\mathbb{R}^d\right)$ and any  $t \geq 0$.
\end{itemize} 
\end{defi}

\begin{defi} \thlabel{Definition Leray--Hopf}
An analytically weak solution $u$ to the power-law system on $\big((\Omega,\mathcal{F},\mathcal{P}),(\mathcal{F}_t)_{t \geq 0},B \big)$ is of Leray--Hopf class, whenever the energy functional $\mathscr{E}$ given by \eqref{Energy Functional} complies with \linebreak $\mathscr{E}\{u\}(t)\leq \mathscr{E}\{u\}(0)$ for all $t\in [0,\infty)$.
\end{defi}

\begin{remark} \thlabel{Remark Leray--Hopf}
Leray--Hopf solutions are often defined in a broader framework by requiring that the energy inequality $\mathscr{E}\{u\}(t) \leq \mathscr{E}\{u\}(s)$ is for all $t \in [0, \infty)$ and for almost every $s \leq t$ satisfied. However, the method of convex integration, especially applied in \thref{Theorem Weak Solution}, yields solutions that fail to satisfy the stronger form of the energy inequality (cf.~Section \ref{Energy Results}). 
\end{remark}
\noindent To introduce the notion of ergodic solutions we consider the right-shift operators, which are defined in the following.
\begin{defi}\thlabel{Definition Shift} 
 For any $s\geq 0$ we define the right-shift operators as
\begin{align*}
 \mathcal{S}^1_su:=u(\cdot+s) \qquad \text{and} \qquad
 \mathcal{S}_s^2B:=B_{\cdot+s}-B_s,
\end{align*}
on
\begin{align*}
\mathcal{T}_{1}&:=C\big([0,\infty);L_{\rm{free}}^2\big( \mathbb{T}^d;\mathbb{R}^d\big)\cap W^{1,\iota}\big( \mathbb{T}^d;\mathbb{R}^d\big)
 \big)
\intertext{and} 
  \mathcal{T}_2&:= C\big([0,\infty);L_{\rm{free}}^2\big( \mathbb{T}^d;\mathbb{R}^d\big)\big)
\end{align*}
respectively. Moreover, let
\begin{align*}
\mathcal{S}_s&:=\mathcal{S}_s^1\times \mathcal{S}_s^2
\end{align*}
be their product operator with domain $\mathcal{T}:=\mathcal{T}_1 \times \mathcal{T}_2$.
\end{defi}

\begin{defi} \thlabel{Definition Stationary}
An analytically weak solution $u$ to \eqref{PLF} on $\big((\Omega,\mathcal{F},\mathcal{P}),(\mathcal{F}_t)_{t \geq 0},B \big)$ is called stationary, whenever 
\begin{align*}
\mathcal{P}\big(\mathcal{S}_s(u,B)\in A \big)=\mathcal{P}\big((u,B)\in A \big)
\end{align*}
holds for any Borel set $A\subseteq \mathcal{T}$ and $s\geq 0$.
\end{defi}

\begin{defi} \thlabel{Definition Ergodic}
An analytically weak solution $u$ to \eqref{PLF} on $\big((\Omega,\mathcal{F},\mathcal{P}),(\mathcal{F}_t)_{t \geq 0},B \big)$ is called ergodic, provided it is stationary and either 
\begin{align*}
\mathcal{P}\big((u,B)\in A \big)=0 \qquad \text{or} \qquad \mathcal{P}\big((u,B)\in A\big)=1
\end{align*}
holds for every shift invariant Borel set $A\subseteq \mathcal{T}$.

\end{defi}

\medskip
 \noindent Now we  formulate the main result of this paper.

\begin{theorem} \thlabel{Main Result}
For a certain initial value, $\iota\in\Big(1,\frac{2d}{d+2}\Big)$ and sufficiently small $\gamma \in (0,1)$ there exist infinitely many ergodic Leray--Hopf solutions to \eqref{PLF}, which belong $\mathcal{P}$-a.s. to the class 
\begin{align} \label{Target Space Main Result}
 C\big([0,\infty);H^\gamma\big( \mathbb{T}^d;\mathbb{R}^d\big) \big)\cap C_{\rm{loc}}^{0,\gamma}\big([0,\infty);L^{2}\big( \mathbb{T}^d;\mathbb{R}^d\big)\cap W^{1+\gamma,\iota}\big( \mathbb{T}^d;\mathbb{R}^d\big)\big).
\end{align}
\end{theorem}

\bigskip 
\noindent The proof fundamentally relies on the method of convex integration and is elaborated across several sections. During the scheme the norm
\begin{align*}
\sup_{|I|=1}\|\cdot\|_{L_\Omega^jC^{N,\gamma}_I W_x^{s,p}},
\end{align*}
initially introduced in \cite{HZZ25} for all $j\in [1,\infty]$, $N\in \mathbb{N}_0\cup \{\infty\}$, $\gamma\in (0,1],\, s \in \mathbb{R}$ and $p\in(1, \infty)$, plays an essential role within the construction of solutions. It prevents blow-up over large times in spaces of the form
 \begin{align*}
L^j\big(\Omega;\,C_{\rm{loc}}^{N,\gamma}\big([0,\infty);\,W^{s,p}\big( \mathbb{T}^d;\mathbb{R}^d\big) \big)\big),
\end{align*}
enabling the derivation of global in time estimates. We will also use this norm for spaces of continuous functions rather than locally Hölder continuous functions, as well as Sobolev spaces in place of Bessel potential spaces.\\
A detailed overview of the underlying strategy is presented in the following section.\\
  \noindent 
\section{Strategy of the Proof} \label{Stragtegy of the Proof}
We initiate the proof by establishing the existence of analytically weak and probabilistically strong solutions to equation \eqref{PLF}, satisfying the energy inequality in \thref{Definition Leray--Hopf}. This result is stated in \thref{Theorem Weak Solution}, accomplished through the technique of convex integration. The core iterative process is presented in \thref{Proposition Main Iteration}. \\
Building on this, we demonstrate existence of stationary solutions, formulated in \thref{Theorem Stationary}, trough the classical Krylov--Bogoliubov procedure. These solutions are non-unique and satisfy the energy inequality as well. \\
Finally, to confirm the existence of infinitely many ergodic Leray--Hopf solutions, presented in \thref{Main Result}, we apply Krein--Milman's theorem.\\

\section{Organization of the Paper}
Section \ref{Preliminaries} introduces the foundational notations that are frequently utilized throughout this paper. In Section~\ref{Leray--Hopf Solutions}, we focus on the existence and non-uniqueness of Leray--Hopf solutions, presented in \thref{Theorem Weak Solution}. More precisely, Section \ref{Factorization} and Section \ref{Extension} incorporate crucial steps that are mandatory for the application of the convex integration scheme. A detailed explanation of this technique is postponed to Section \ref{Convex Integration Scheme}, compiled in the formulation of \thref{Proposition Main Iteration}. Leveraging this result we conclude the proof of \thref{Theorem Weak Solution} in Section \ref{End of the Proof}.\\
Section \ref{Existence and Non-Uniqueness of ergodic Leray--Hopf Solutions} is dedicated to proving the existence and non-uniqueness of ergodic Leray--Hopf solutions. Specifically, Section \ref{Existence and Non-Uniqueness of stationary Leray--Hopf Solutions} provides the proof of existence and non-uniqueness of stationary Leray--Hopf solutions, which is then used to derive even infinitely many ergodic Leray--Hopf solutions in Section \ref{Ergodicity of Solutions}.\\
While Appendix~\ref{Appendix A.1} provides a detailed examination of the parameter choices made in Section~\ref{Choice of Parameters}, Appendix~\ref{Appendix A.2} and Appendix~\ref{Appendix A.3} cover essential lemmata and additional calculations that appear in the preceding sections.

\addcontentsline{toc}{chapter}{Acknowledgments}
\section*{Acknowledgments}
The author wishes to express her sincere gratitude to Martina Hofmanová for her constant support and guidance throughout this work, which forms part of the author's doctoral dissertation.  \clearpage

\chapter{Preliminaries}\label{Preliminaries}
In the forthcoming section we will introduce various function spaces and operators, whose definition based on the Fourier transform on the $d$-dimensional torus. We define the Fourier transform of a function $u$ and the corresponding inversion on $\mathbb{T}^d$ as
\begin{align*}
(\mathcal{F}u)(n)&:=\hat{u}_n:=(2\pi)^{-d} \int_{\mathbb{T}^d} u(y)e^{-in\cdot y}\, dy \qquad
\text{and} \quad
 (\mathcal{F}^{-1}\hat{u}_n)(x)=u(x)=\sum_{n \in \mathbb{Z}^d}\hat{u}_n e^{in\cdot x}
\end{align*}
for all $n \in \mathbb{Z}^d$ and $x\in \mathbb{T}^d$.\\
We will also consider symmetric $d\times d$ - matrices  $A\in\mathbb{R}^{d\times d}_\text{symm}$, especially those of zero trace. To emphasize that they belong to the class of traceless matrices $\mathring{\mathbb{R}}^{d\times d}$ we will always write $\mathring{A}$ instead of $A$. \\
Moreover, for an interval $I:=[T,T+1]$ with $T\in \mathbb{R}$ we set 
\begin{align} \label{Interval}
I_\ell:=[T-\ell,T+1]
\end{align}
for some $\ell \in(0,1)$.

\section{Function Spaces} \label{Function Spaces}
During the convex integration scheme, it is crucial to restrict the velocity to an arbitrary but fixed interval $I$ of length $1$. Inspired by \cite{HZZ25} we therefore endow the space
\begin{align*}
 L^j\big(\Omega;\,C_{\rm{loc}}^{N,\gamma}\big([0,\infty);\,W^{s,p}\big( \mathbb{T}^d;\mathbb{R}^d\big) \big)\big)
\end{align*}
for all $j\in[1, \infty]$, $N\in \mathbb{N}_0\cup \{\infty\},\,\gamma\in (0,1],\, s \in \mathbb{R}$, $p \in (1,\infty)$ and $d\geq 3$ with the norm
\begin{align*}
\sup_{|I|=1}\|\cdot\|_{L_\Omega^jC^{N,\gamma}_I W_x^{s,p}}.
\end{align*}
We will also apply this norm to spaces of continuous functions instead of locally Hölder continuous ones, and use Sobolev spaces in place of Bessel potential spaces. In the sequel we introduce each part of this space in detail.
\subsection{Continuous Functions}
For an arbitrary $I \subseteq \mathbb{R}$, all $N\in \mathbb{N}_0\cup \{\infty\}$ and  $d\geq 3$ we denote the space of all $N$-times continuously differentiable functions on $I\times \mathbb{T}^d$ by $C^N_{I,x}\left(I\times \mathbb{T}^d; \mathbb{R}^d \right)$. A function on $I$ with values in an arbitrary Banach space $\left(X,\|\cdot\|_X\right)$ is said to belong to $C_I^{N,\gamma}\left(I; X\right)$, if it has continuous derivatives up to order $N$ and is additionally Hölder continuous of order $\gamma\in (0,1]$. Endowed with their corresponding norms
\begin{gather*}
\|u\|_{C_{I,x}^N}:= \sum_{\substack{0 \leq n+|\alpha| \leq N\\  \alpha \in \mathbb{N}_0^d}}\sup_{t\in I} \|\partial_t^n D^\alpha u(t) \|_{L_x^\infty} \intertext{and}
\|u\|_{C_{I}^{N,\gamma}}:=\sum_{0 \leq n \leq N}\sup_{t\in I}\| \partial_t^n u(t)\|_{X}+[ \partial_t^N u]_{C_{I}^{0,\gamma}}
\intertext{with}
[u]_{C_{I}^{0,\gamma}}:= \sup_{\substack{s,t \in I \\ s \neq t}} \frac{\|u(t)-u(s)\|_X}{|t-s|^\gamma},
\end{gather*}
these spaces are complete. The space $C_{\rm{loc}}^{N,\gamma}\left([0,\infty);X\right)$ contains all functions that belong to $C_I^{N,\gamma}\left(I; X\right)$ for every interval $I$. \\
If a $N$-times continuously differentiable function is bounded or compactly supported, we say that it lies in $C_b^N$ or $C_c^N$, respectively. \\
In general, we will omit writing the index $N$ whenever $N=0$.

\addtocontents{toc}{\par\noindent\rule{\textwidth}{0.4pt}
This project has received funding from the European Research Council (ERC) under the European Union's Horizon 2020 research and innovation programme (grant agreement No. 949981). \medskip\\
Faculty of mathematics, University of Bielefeld, D-33501 Bielefeld, Germany,\\
\Letter \, \href{mailto:sberkeme@math.uni-bielefeld.de}{sberkeme@math.uni-bielefeld.de~\\}
  }

\subsection{Bessel Potential Spaces}
For two Banach spaces $\left(X,\|\cdot\|_X\right),\, \left(Y, \|\cdot\|_Y \right)$ and  $1\leq p \leq \infty$ we will designate the space of Bochner-integrable functions from $X$ to $Y$ by $\left( L^p\left(X;Y\right), \|\cdot\|_{L_X^p}\right)$ and for  integer $k\in \mathbb{N}_0$ the usual Sobolev space by $\left( W^{k,p}\left(\mathbb{T}^d;\mathbb{R}^d\right), \|\cdot\|_{W_x^{k,p}}\right)$.\\
To extend the definition of Sobolev spaces for all real $k\in \mathbb{R}$ and $p,q\in (1,\infty)$ satisfying $\frac{1}{p}+\frac{1}{q}=1$, we define the Bessel potential space by using the Fourier transform as
\begin{align*}
W^{s,p}\left(\mathbb{T}^d;\mathbb{R}^d\right)&:=\bigg\{u \in L^p\left(\mathbb{T}^d;\mathbb{R}^d\right): \Big\|\mathcal{F}^{-1}\big[(1+|\cdot|^2)^{s/2}\mathcal{F}u\big] \Big\|_{L_x^p}<\infty\bigg\}
\intertext{for $s\geq 0$, and} 
W^{-s,p}\left(\mathbb{T}^d;\mathbb{R}^d\right)&:=\bigg\{u ^\prime\in \left(C^\infty_c\left(\mathbb{T}^d;\mathbb{R}^d\right)\right)^\prime: \Big\|\mathcal{F}^{-1}\big[(1+|\cdot|^2)^{-s/2}\mathcal{F}u^\prime\big]  \Big\|_{(L_x^q)^\prime}<\infty\bigg\} 
\end{align*}
whenever $s>0$.\\
It is natural to endow these spaces with 
\begin{align*}
\|u\|_{W_x^{s,p}}:= \Big\|\mathcal{F}^{-1}\big[(1+|\cdot|^2)^{s/2}\mathcal{F}u\big] \Big\|_{L_x^p}
 \quad \text{and} \quad
\|u^\prime\|_{W_x^{-s,p}}:= \Big\|\mathcal{F}^{-1}\big[(1+|\cdot|^2)^{-s/2}\mathcal{F}u^\prime\big] \Big\|_{(L_x^q)^\prime},
\end{align*}
respectively, so that they become complete. \\
Note that the Bessel potential space
$W^{-s,q}\left(\mathbb{T}^d;\mathbb{R}^d\right)$ is isomorphic to $\left(W_0^{s,p}\left(\mathbb{T}^d;\mathbb{R}^d\right) \right)^\prime$, with equivalent norms, where $W_0^{s,p}\left(\mathbb{T}^d;\mathbb{R}^d\right)$ indicates the closure of $C_c^\infty\left(\mathbb{T}^d;\mathbb{R}^d\right)$ in $W^{s,p}\left(\mathbb{T}^d;\mathbb{R}^d\right)$.\\
Another different approach to bridge the gap between the usual Sobolev space $W^{k,p}\left(\mathbb{T}^d;\mathbb{R}^d\right)$ and $W^{k+1,p}\left(\mathbb{T}^d;\mathbb{R}^d\right)$ with integer $k$ is the use of real interpolation spaces, known as Sobolev–\linebreak Slobodeckij spaces. In general they distinguish from the Bessel potential spaces, which are obtained via complex interpolation. They only coincide for $s\in (0,1)$ and $p=2$ (see Proposition 3.4 in \cite{DNPV12}). 
\medskip \\
In general $X^\prime$ represents the dual space of $X$.\\
If one of the above spaces $Z$ consists only of mean free functions, we will emphasize it by writing $Z_{\neq 0}$ and if they are additionally divergence free by writing $Z_{\rm{free}}$. However we will especially write $H^s$ instead of $W^{s,2}_{\rm{free}}$.

\section{Operators}
Throughout this paper we will also need different kinds of operators, presented in the following sections.

\subsection{Leray Projection and Fourier cut-offs}
In what follows we denote by $\mathbb{P}= \Id- \nabla \Delta^{-1} \divs$ the extended Leray projection on $W_{\neq 0}^{s,p}\left(\mathbb{T}^d;\mathbb{R}^d\right)$ with $s\geq 0$ and~$p\in (1,\infty)$. \\
Furthermore, the convex integration technique requires working with mean-free functions, as well as functions truncated in the sense that they are projected onto their Fourier modes with frequencies smaller or greater than some $\kappa > 0$. More precisely, for any  $u \in L^p\left(\mathbb{T}^d\right)$ with $1<p\leq \infty$ we define 
\begin{align*}
\left(\mathbb{P}_{\neq 0}u\right)(x)&:=\mathcal{F}^{-1}\Big[ \mathbbm{1}_{\{|\cdot|\neq 0 \}}\mathcal{F}u\Big](x)=u(x)-(\mathcal{F}u)(0)=u(x)-\fint_{ \mathbb{T}^d}u(x)\,dx,\\
\left(\mathbb{P}_{\leq \kappa}u\right)(x)&:= \mathcal{F}^{-1}\Big[ \chi_\kappa\, \mathcal{F}u\Big](x)=\sum_{n \in \mathbb{Z}^d} \chi_\kappa(n)\hat{u}_n e^{in\cdot x}\\
\intertext{and}
\left(\mathbb{P}_{\geq \kappa}u\right)(x)&:= \mathcal{F}^{-1}\Big[ \big(1-\chi_\kappa\big)\, \mathcal{F}u\Big](x)=\sum_{n \in \mathbb{Z}^d} \big(1-\chi_\kappa(n)\big)\hat{u}_ne^{in\cdot x},
\end{align*}
with $\chi_\kappa:= \chi \left( \frac{\cdot}{\kappa}\right)$ and $\chi \in C_c^\infty \left(\mathbb{R}^d;\mathbb{R}\right)$ given by
\begin{align*}
\chi(x):=\left\{\begin{array}{ll}
1,& |x|\leq \frac{1}{2}, \\ [1.3ex] 
\in (0,1),& \frac{1}{2}< |x|<1,\\[1.3ex]
0,& |x|\geq 1.\\  
\end{array}\right. 
\end{align*}
Here $\fint_{ \mathbb{T}^d}u(x)\,dx$ is nothing else than the averaged integral $(2\pi)^{-d}\int_{ \mathbb{T}^d}u(x)\,dx$. \\
The most important properties of these Fourier multipliers are listed below: 
\begin{subequations}
\begin{flalign}
& \bullet \, \mathbb{P}_{\neq 0}, \, \mathbb{P}_{\leq \kappa}, \, \mathbb{P}_{\geq \kappa} \text{ are bounded on $L^p\left(\mathbb{T}^d\right)$ for all $1<p\leq \infty$, where the}\label{OPa}&\\
&\hspace{.4cm} \text{implicit constant is independent of $\kappa$,}\notag& \\
& \bullet \,   \left(-\Delta\right)^{s/2},  \left(\Id-\Delta\right)^{s/2},\mathbb{P}_{\leq \kappa},  \mathbb{P}_{\geq \kappa} \text{ commute for all $s\in \mathbb{R}$ with each other,}\label{OPb} &\\
& \bullet \, \, \|\left(-\Delta\right)^{s/2} \mathbb{P}_{\leq \kappa}\|_{L_x^p\to L_x^p} \lesssim  \kappa^s \text{ hold for all $1<p<\infty$ and $s\geq 0$},\label{OPc}&\\
& \bullet \,  \|\left(-\Delta\right)^{-s/2} \mathbb{P}_{\geq \kappa}\|_{L_x^p\to L_x^p} \lesssim \frac{1}{\kappa^s} \text{ hold for all $1<p<\infty$ and $s\geq 0$}.\label{OPd}&\\
\intertext{If $u\in L^p(\mathbb{T}^d)$ is $\left(\frac{\mathbb{T}}{L}\right)^d$-periodic with $L \in \mathbb{N}$ and $1< p\leq \infty$ one has}
 &  \bullet \,\left(\mathbb{P}_{\leq \kappa}u\right)(x)= \sum_{n \in (L\mathbb{Z})^d} \chi_\kappa(n)\hat{u}_{n}e^{in\cdot x},  \label{OPe}& \\
& \bullet \,\left(\mathbb{P}_{\geq \kappa}u\right)(x)= \sum_{n \in (L\mathbb{Z})^d} \big(1-\chi_\kappa(n)\big)\hat{u}_{n}e^{in\cdot x}.  \label{OPf}&
\intertext{If, in particular, $L$ is strictly greater than $\kappa$, we have}
&\bullet \, \mathbb{P}_{\geq \kappa}u=\mathbb{P}_{\neq 0}u \label{OPg}.&
\end{flalign}
\end{subequations}  \vspace{0cm}

\noindent The operator $-\Delta$ in \eqref{OPc} might be replaced by $\Id-\Delta$, whenever $\kappa \geq 1$ and in \eqref{OPd} for any $\kappa> 0$.\\
These results follow basically from simple calculations and Mikhlin’s multiplier theorem.

\subsection{Right-Inverse of Divergence}
A right-inverse of the divergence operator will also be from important meaning. The following lemmata recalls the one from \cite{CL22} on p. 1048 and its bilinear version.

\begin{lemma} \thlabel{Lemma 2.1} 
The operator $\mathcal{R}$ defined by 
\begin{gather*}
\mathcal{R}\colon \left(C_{\neq 0}^\infty\left(\mathbb{T}^d; \mathbb{R}^d\right),\|\cdot\|_{L_x^p}\right) \to \left(L^p\left(\mathbb{T}^d; \mathring{\mathbb{R}}_{\text{sym}}^{d\times d}\right), \|\cdot\|_{L_x^p}\right),\\
(\mathcal{R}u)_{ij}=\partial_{x_i} \Delta^{-1}u_j+ \partial_{x_j}\Delta^{-1}u_i-\frac{1}{d-1} \Big((d-2)\partial_{ x_i} \partial_{x_j} \Delta^{-1} +\delta_{ij} \Big)\divs(\Delta^{-1}u)
\end{gather*}
is a right inverse of the divergence operator and is particularly bounded for $1\leq p \leq \infty$.
\end{lemma}

\begin{lemma}\thlabel{Lemma 2.2}
 The operator
\begin{gather*}
\mathcal{B}\colon C_{\neq 0}^\infty\left(\mathbb{T}^d; \mathbb{R}^{d\times d}\right)\times C^\infty\left(\mathbb{T}^d; \mathbb{R}^d\right) \to C^\infty\left(\mathbb{T}^d; \mathring{\mathbb{R}}_{\text{sym}}^{d\times d}\right),\\
\mathcal{B}(A,u):=\Pm\left( \sum\limits_{k,\ell=1}^d u_\ell \mathcal{B}_{ijk}A_{\ell k}\right)_{ij}- \mathcal{R}\Pm\left(\sum\limits_{i,k,\ell=1}^d\partial_{x_i}u_\ell \mathcal{B}_{ijk }A_{ \ell k }\right)_j ,
\end{gather*}
with
\begin{align*}
\mathcal{B}_{ijk}:=\partial_{x_i}\Delta^{-1}\delta_{jk}+\partial_{x_j}\Delta^{-1}\delta_{ik}-\frac{1}{d-1}\Big( (d-2) \partial_{x_i}\partial_{x_j}\Delta^{-1}+\delta_{ij}\Big)\partial_{x_k}\Delta^{-1}
\end{align*}
is a bilinear version of the antidivergence in the sense that
\begin{align*}
\divs \Big(\mathcal{B}\left(A,u \right)\Big)= \Pm (A^Tu)
\end{align*}
holds. Moreover we have
\begin{align*}
\|\mathcal{B}(\mathbb{P}_{\geq \kappa} A,u)\|_{L_x^p}\lesssim \frac{1}{\kappa}\|A\|_{L_x^p} \|u\|_{C_x^1}
\end{align*}
for all $ p \in(1, \infty).$
\end{lemma}
\noindent It follows from Calder\'on--Zygmund's inequality and Mikhlin's multiplier theorem that even the composition of $\mathcal{R}$ with the differential operators $\Delta$, $\divs$ and with the Fourier cut-off operator $\mathbb{P}_{\geq \kappa}$ are bounded operators as well. 
\begin{lemma} \thlabel{Lemma 2.3} 
The composition of operators
\begin{itemize}
\item[i)] $
\mathcal{R}\divs \colon \left(C^\infty\left(\mathbb{T}^d;\mathbb{R}^{d\times d}\right),\|\cdot\|_{L_x^p}\right) \to \left(L^p\left(\mathbb{T}^d; \mathring{\mathbb{R}}_{\text{sym}}^{d\times d}\right), \|\cdot\|_{L_x^p}\right),
$
 \item[ii)]$
\mathcal{R}\Delta \colon \left(C^\infty\left(\mathbb{T}^d;\mathbb{R}^d\right),\|\cdot\|_{W_x^{1,p}}\right) \to \left(L^p\left(\mathbb{T}^d; \mathring{\mathbb{R}}_{\text{sym}}^{d\times d}\right), \|\cdot\|_{L_x^p}\right),
$
\item[iii)]$
\mathcal{R}\mathbb{P}_{\geq\kappa}\colon \left( C^\infty\left(\mathbb{T}^d;\mathbb{R}^d\right) , \|\cdot\|_{L_x^p}\right) \to \left(L^p\left(\mathbb{T}^d; \mathring{\mathbb{R}}_{\text{sym}}^{d\times d}\right), \|\cdot\|_{L_x^p}\right) 
$
\end{itemize}
are for $p \in (1,\infty)$ continuous. Especially we find
\begin{align*}
\|\mathcal{R}\mathbb{P}_{\geq\kappa}\|_{L_x^p\to L_x^p} \lesssim \frac{1}{\kappa}.
\end{align*}
\end{lemma}

\subsection{Non-Newtonian Tensor}
To control the nonlinear term of the power-law flows during the convex integration technique, it is indispensable to find some suitable bounds for the non-Newtonian tensor. The following growth estimates from \cite{BMS21} are recalled in the subsequent lemma.

\begin{lemma}\thlabel{Lemma 2.4}
The non-Newtonian operator $\mathcal{A}$ given by \eqref{Non-Newtonian-Tensor} admits the following growth estimates

\begin{align*}
\|\mathcal{A}(Q)-\mathcal{A}(R)\|_F\leq \left\{\begin{array}{ll}
C_{\nu_1}\|Q-R\|^{\iota-1}_F, & \nu_0=0,\,  \iota \leq 2, \\ 
C_{\nu_0}\|Q-R\|_F, & \nu_0>0,\, \iota \leq 2, \\
    C_{\iota,\nu_0,\nu_1}\|Q-R\|_F\big( 1+\|Q\|_F^{\iota -2}+\|R\|_F^{\iota -2}\big), &  \iota \geq 2, \end{array}\right.
\end{align*}
    and
\begin{align*}
    |\mathcal{A}(Q): Q-\mathcal{A}(R): R|\leq  C_{\iota ,\nu_0,\nu_1}\|Q-R\|_F\big( 1+\|Q\|_F^{\iota -1}+\|R\|_F^{\iota -1}\big) 
\end{align*} 
for all $Q,\, R\in L^\iota\big(\Omega;L^\iota\big([0,\infty);L^\iota \big( \mathbb{T}^d;\mathbb{R}^{d\times d}\big)\big)\big) $ and some $C_{\nu_0}, C_{\nu_1}, C_{\iota,\nu_0,\nu_1}\geq 1$.
\end{lemma}

\noindent
The constants $C_{\nu_0}, C_{\nu_1}, C_{\iota,\nu_0,\nu_1}$ are not required to be greater than or equal to $1$; it suffices that they are positive. Nevertheless, for the sake of simplicity in our computations, we retain the assumption that these constants satisfy the aforementioned condition. Moreover, it is worth mentioning that this result can be extended to mappings of the form $ Q \mapsto \left( \nu_0 + \nu_1 \|Q\|_F^2 \right)^{\frac{\iota-2}{2}} Q $, as well as to other tensor-valued functions defined via appropriate N-functions.

 \chapter{Existence \& Non-Uniqueness of Leray--Hopf Solutions} \label{Leray--Hopf Solutions}
In this section we focus on the existence of infinitely many probabilistically strong Leray--Hopf solutions to~\eqref{PLF}, outlined in the following theorem.
 
\begin{theorem} \thlabel{Theorem Weak Solution}
There exist infinitely many probabilistically strong Leray--Hopf solutions to the power-law system of shear-thinning fluids \eqref{PLF} with power index $\iota \in \left(1,\frac{2d}{d+2} \right)$ and certain deterministic initial data.\\ 
More precisely, for every non-negative energy $e\in C^\infty\left([0,\infty);[\underline{e},\bar{e}]\right)$, satisfying 
\begin{align*}
\underline{e}\leq e(t) \leq \bar{e} \qquad \text{and} \qquad \Big|\frac{d}{dt} e (t) \Big|\leq \widetilde{e}
\end{align*}
 for all $t\in [0,\infty)$, $\bar{e}>\underline{e}>0$ sufficiently large and some $\widetilde{e}>0$,  there exists an analytically weak and probabilistically strong solution $u$ on $\big((\Omega,\mathcal{F},\mathcal{P}),(\mathcal{F}_t)_{t \geq 0},B \big)$ to the power-law system \eqref{PLF}. This solution depends explicitly on the given energy $e$ through the relation
 \begin{align} \label{Energy Equation}
\E\|u(t)\|_{L_x^2}^2+\frac{2H}{t+H}\E\bigg[\int_0^{t} \int_{\mathbb{T}^d} \mathcal{A}(Du(s,x)) \colon Du (s,x)\,dx \,ds\bigg]=e(t)
\end{align}
for some $H>0$ and all $t \in [0,\infty)$ and belongs $\mathcal{P}$-a.s. to the class
\begin{align*} 
C\big([0,\infty);H^\gamma\big( \mathbb{T}^d;\mathbb{R}^d\big) \big)\cap C_{\rm{loc}}^{0,\gamma}\big([0,\infty);L^{2}\big( \mathbb{T}^d;\mathbb{R}^d\big)\cap W^{1+\gamma,\iota}\big( \mathbb{T}^d;\mathbb{R}^d\big)\big)
\end{align*}
for some $\gamma \in \left(0,1-\sqrt{\frac{2d+6}{2d+7}}\,\right)$.\\
Moreover the following consistency result holds:\\
If two distinct energies with the same bounds $\underline{e}, \bar{e}$ and $\widetilde{e}$ coincide on some bounded interval $[0,T]$ with $T\geq 0$ arbitrarily large, then so do the corresponding solutions. For $T=0$ the two solutions distinguish but share the same initial data.\\
In particular, if $e(t)=H\tr(G)$, the solution is additionally of Leray--Hopf class.
\end{theorem}

\begin{remark}
The novel relation \eqref{Energy Equation} guarantees the validity of the energy inequality in the sense of \thref{Definition Leray--Hopf}. Especially, the factor $\frac{H}{t+H}$ multiplying the nonlinear term provides global in time bounds during the convex integration scheme, which not only ensures the construction of Leray--Hopf solutions but also underlies their ergodicity.
\end{remark}

\noindent A direct consequence of its construction is mentioned in the subsequent remark (cf. Section~\ref{Regularity and Bound}).

\begin{remark}\thlabel{Remark Weak Solution}
 For some $\epsilon>0$ arbitrary small and $J\geq 1$ arbitrary large the solution $u$, constructed in \thref{Theorem Weak Solution}, even obeys
\begin{align*}
\sup_{|I|=1}\|u\|_{L_\Omega^{2J} C_I H_x^\gamma}+\sup_{|I|=1}\|u\|_{L_\Omega^{2J} C^{0,\gamma}_I L_x^2}+\sup_{|I|=1}\|u\|_{L_\Omega^{2J} C^{0,\gamma}_I W_x^{1+\gamma,\iota}}<\infty
\end{align*}
as well as
\begin{align*}
\sup_{|I|=1}\left\|u-\int_0^\cdot e^{-(\Id-\Delta) (\cdot-r)}dB_r\right\|_{L_\Omega^{2J}C_I W_x^{1,\iota}} \leq \epsilon.
\end{align*}
\end{remark}

\noindent The proof of Theorem \ref{Theorem Weak Solution} relies on the method of convex integration, emerged from the classical PDE theory. That means we construct a solution $u_q$ to \eqref{PLF}, perturbed by an error term $\mathring{R}_q$, at level $q \in \mathbb{N}_0$, in such a way that the Reynolds stress $\mathring{R}_q$ tends to zero, while the velocity $u_q$ approximates the desired solution $u$.\\
 The primary challenge in the stochastic setting is effectively managing the stochastic noise, with the objective of mitigating its influence as much as possible. The subsequent section is dedicated to addressing this issue.

\section{Factorization} \label{Factorization}
In the field of SPDE's, a widely used strategy is to eliminate the stochastic integral from the equation, a method commonly referred to as the Da Prato--Debusche trick.\\
In our case as well, we aim to leverage this technique by decomposing the solution into the sum of a process $z$, solving the linear SPDE 
\begin{align} \label{Stochastic Heat Equation}
\begin{split}
dz+(\Id-\Delta) z \, dt&=  dB_t,\\
z_{|\,t=0}&=0
\end{split}
\end{align}
and a process $v$, which is a solution to the nonlinear but random system
\begin{align}\label{Convex Integration Equation}
\begin{split}
\partial_t v+\divs\big((v+z)\otimes (v+z) \big)-\divs\left(\mathcal{A}(Dv+Dz) \right)+\Delta z-z+ \nabla P&=0,\\
\divs(v)&=0.
\end{split}
\end{align}
\noindent
The existence of a mild solution to the modified stochastic heat equation \eqref{Stochastic Heat Equation} is a classical result in stochastic analysis. As we will see in the subsequent proposition, this solution especially inherits the incompressibility condition of the underlying Brownian motion. This is also where the required regularity of $G$ comes into play. It ensures the weak differentiability of $z$ with respect to spatial variables, thereby placing $Dz$ within the domain of the non-Newtonian tensor $\mathcal{A}$, making the decomposition of $u$ permissible.\\ We formalize this result in the following proposition, which is a refinement of Theorem~5.4 in \cite{DPZ14}, supplemented by an application of Kolmogorov’s continuity theorem.

\begin{prop}\thlabel{Proposition SHE}
Under the assumption $\tr\big((\Id-\Delta)^{2\sigma} G\big)<\infty$ for some $\sigma\in(0,1)$ there exists a uniquely determined mild solution 
\begin{align} \label{SHE Solution}
z(\omega,t,x)=\int_0^t e^{-(\Id-\Delta) (t-r)}dB_r(\omega,x)
\end{align}
to \eqref{Stochastic Heat Equation}, belonging $\mathcal{P}$-a.s. to the class 
\begin{align*}
  C_{\rm{loc}}^{0,\frac{1}{2}-\varepsilon}\big([0,\infty);L^2\big( \mathbb{T}^d;\mathbb{R}^d\big)\big) \cap C_{\rm{loc}}^{0,\frac{\sigma}{2}-\varepsilon}\big([0,\infty);H^{1+\sigma}\big( \mathbb{T}^d;\mathbb{R}^d\big)\big) 
\end{align*}
 for some $\varepsilon \in \left(0, \frac{\sigma}{2}\right)$. Particularly, one has
\begin{align*}
\|z\|_{L_\Omega^j C^{0,\frac{1}{2}-\varepsilon}_{I}L_x^{2}}+\|z\|_{L_\Omega^j C^{0,\frac{\sigma}{2}-\varepsilon}_{I}H_x^{1+\sigma}}\leq  K_G(j-1)^{1/2}
\end{align*}
for every bounded interval $I \subset [0,\infty)$ of length $1$, all $j\geq 2$ and some constant $K_G\geq 1$.
\end{prop}
\noindent \medskip \\
This result allows to find a solution $u= v+z$ to \eqref{PLF}, whenever $v$ solves \eqref{Convex Integration Equation}.

\section{Extension} \label{Extension}
In the further proceeding we define the Reynolds stress at the level $q+1$ by incorporating the mollification of several processes. To give meaning to the mollification around time $t=0$, we need to extend the energy $e$, the Brownian motion $B$ and the process $z$ to negative times by assigning them equal to their value at time $t=0$.

\section{Main Iteration} 
Before we start with determining all the parameters, which are mandatory for the convex integration scheme, we want to remember that the space 
\begin{align*}
L^j\big(\Omega;\,C_{\rm{loc}}^{N,\gamma}\big([0,\infty);\,W^{s,p}\big( \mathbb{T}^d;\mathbb{R}^d\big)\big)\big)
\end{align*}
is for all  $j\in[1,\infty],\, N\in \mathbb{N}_0\cup \{\infty\},\, \gamma\in (0,1],\, s \in \mathbb{R}$ and $p\in[1, \infty]$ equipped with the norm
\begin{align*}
\sup_{|I|=1}\|\cdot\|_{L_\Omega^jC^{N,\gamma}_I W_x^{s,p}}
\end{align*}
(see Section \ref{Function Spaces}). This will become even more important in the course of the work.

\subsection{Choice of Parameters}\label{Choice of Parameters} 
Let us first mention that the function $f_A \colon [0,\infty) \to \mathbb{R}$, $f_A(x):=x\left(\frac{A}{A+1}\right)^x$ attains its maximum for all $A>0$ at $ x_{\text{max}}=\frac{1}{\ln(A+1)-\ln(A)}$. This allows to state the following inequality  
\begin{align} \label{parameter bound}
 q \left( \frac{A}{A+1}\right)^q \leq c_A
\end{align}
for some $c_A\geq f_A(x_{max})$.\\
 For sufficiently large $a \in \mathbb{N},\, b\in 2\varsigma_d \mathbb{N}$ and sufficiently small $\alpha,\beta \in (0,1)$ we require \pagebreak\\
\begin{align*}
\bullet \, & (90d+208)\alpha\leq \sigma-2\varepsilon, && \bullet \, a^{\beta b^2}\geq 3,\\
\bullet \, & \beta \leq \frac{(2d+6)(\sigma -2\varepsilon)(4-\sigma+2\varepsilon)}{(2-\sigma+2\varepsilon)^2}, && \bullet \, (\iota-1) \sigma \alpha >2\beta b^3,  \\
\bullet \, & 10\leq   A\leq b-1, && \bullet \, \alpha b>2d+4N^\ast,\\
\bullet \, & a \geq \Big((8N^\ast+32) JLA\Big)^c.
\end{align*}
with  $\varsigma_d \in \mathbb{N}$ as in \eqref{rational number}, $c:=\max\limits_{A\in [10,b-1]}c_A$ satisfying \eqref{parameter bound}, $\sigma\in [\iota-1,1)$ and $\varepsilon\in \left( 0, \frac{\sigma}{2}\right)$ as in \thref{Proposition SHE} and $N^\ast\geq 2$ as in \eqref{Nast}. Moreover, for any $q \in \mathbb{N}_0$ let us define \\

\begin{align*}
&\lambda_q := a^{(b^q)}, & &\ell^{\sigma-2\varepsilon}_q := \lambda_{q+1}^{-2\alpha} \lambda_q^{-2d-10}, \\
&\delta_q := \left\{\begin{array}{ll} 
    \frac{1}{4}, & q\in \{0,1\}, \\
    \frac{1}{16}\lambda_2^{2\beta}\lambda_q^{-2\beta}, & q \geq 2,
\end{array}\right. & &o_q := -\sum\limits_{r=q}^{\infty} \ell_r.
\end{align*}
 To summarize the above considerations we have constructed non-increasing null-sequences \linebreak $\left(\delta_q\right)_{q \in \mathbb{N}_0},\, \left(\ell_q\right)_{q \in \mathbb{N}_0}\subseteq (0,1)$ and an increasing sequence $\left(\lambda_q\right)_{q \in \mathbb{N}_0}\subseteq \mathbb{N}$, which diverges to $\infty$.
 It is also worth mentioning that the series $\sum_{r=1}^\infty\delta_r^{1/2}$ will not exceed~$1$. It relies on the required $a^{\beta b^2} \geq 3$. Specifically, we have
\begin{align} \label{delta series bound}
4\sum_{r=1}^\infty \delta_r^{1/2}&= 2+ \sum_{r=2}^\infty  a^{\beta b(b-b^{r-1})} \leq 2+ \sum_{r=2}^\infty  a^{\beta b(b-(r-1)b)}=2+ \sum_{r=0}^\infty  \left(a^{-\beta b^2}\right)^r\\
&=2+ \frac{1}{1-a^{-\beta b^2}}\leq 4 \notag.
\end{align}
So as one of the consequences the sequence $(o_q)_{q\in\mathbb{N}_0}$ is bounded from below by $-1$, because of
\begin{align} \label{l series bound}
\sum_{r=0}^\infty  \ell_r \leq \sum_{r=0}^\infty \lambda_{r+1}^{-2\alpha} \leq \lambda_1^{-2\alpha}+\sum_{r=2}^\infty \lambda_r^{-\beta} \leq  a^{-2\alpha b}+4 \lambda_2^{-\beta} \sum_{r=1}^\infty \delta_r^{1/2} -4\lambda_2^{-\beta} \delta_1^{1/2}\leq 1.
\end{align}
Especially, it converges to zero as $q \to \infty$.\\
Additionally we want to emphasize that the choice of parameters ensure the relations 
\begin{align*}
\ell_q &\leq  \lambda_{q+1}^{-2\alpha}, \\
\ell_q^{-1} &\leq \lambda_{q+1}^{4\frac{\alpha}{\sigma-2\varepsilon}},
\end{align*}
which will frequently used in the upcoming computations. For simplicity we will often write $\ell$ instead of $\ell_q$. \\
Choosing $a,\,b$ sufficiently large and $\alpha,\,\beta$ sufficiently small enough, permits to absorb several implicit constants, appearing throughout this work. Further indispensable bounds are 
\begin{subequations}
\begin{align}
r_\perp^{-\frac{1}{2}}\lambda_{q+1}^{-\frac{1}{2}}&\leq \lambda_{q+1}^{-(10d+24)\frac{\alpha}{\sigma-2\varepsilon} }, \label{essential bound a}\\
r_\perp^{\frac{d-1}{\iota}-\frac{d-1}{2}} r_\parallel^{\frac{1}{\iota}-\frac{1}{2}} \lambda_{q+1} &\leq \lambda_{q+1}^{-(14d+36)\frac{\alpha}{\sigma-2\varepsilon} }\label{essential bound b},\\
r_\perp^{\frac{d-1}{p}-\frac{d-3}{2}}r_\parallel^{\frac{1}{p}-\frac{3}{2}}\mu &\leq \lambda_{q+1}^{-(10d+24)\frac{\alpha}{\sigma-2\varepsilon} }\label{essential bound c},\\
f^{\frac{d-2}{2}}(q+1)r_\perp^{\frac{d-1}{p}-\frac{d-1}{2}}r_\parallel^{\frac{1}{p}-\frac{1}{2}} &\leq \lambda_{q+1}^{-(10d+24)\frac{\alpha}{\sigma-2\varepsilon} } \label{essential bound d}, \\
 f^{\frac{d-2}{2}}(q)r_\perp^{\frac{d-1}{p}-(d-1)}r_\parallel^{\frac{1}{p}-1} \mu^{-1}&\leq \lambda_{q+1}^{-(12d+26)\frac{\alpha}{\sigma-2\varepsilon}}\label{essential bound e},\\ 
r_\perp^{\frac{d-1}{p}-d}r_\parallel^{\frac{1}{p}-1}\lambda_{q+1}^{-1}&\leq \lambda_{q+1}^{-(20d+46)\frac{\alpha}{\sigma-2\varepsilon} } \label{essential bound f},\\
r_\perp r_\parallel^{-1}&\leq \lambda_{q+1}^{-(20d+46)\frac{\alpha}{\sigma-2\varepsilon} } \label{essential bound g},\\
r_\perp^{-\frac{d-1}{2}}r_\parallel^{-\frac{1}{2}}\mu^{-1} &\leq \lambda_{q+1}^{-(20d+46)\frac{\alpha}{\sigma-2\varepsilon} }, \label{essential bound h}
\end{align}
\end{subequations}
where
\begin{align*}
\lambda:=\lambda_{q+1}, \qquad r_\perp:=\lambda_{q+1}^{-2\varsigma}, \qquad r_\parallel:=\lambda_{q+1}^{-\tau}, \qquad \mu:=\lambda_{q+1}^{d\varsigma}
\end{align*}
appear in the construction of the intermittent jets in Section \ref{Perturbation} and $f(q)$ is given as in \eqref{frequency z}. This conditions
are satisfied, whenever
\begin{gather*}
\varsigma \in \left[\frac{1}{3}+ \frac{2}{3}\frac{\alpha}{\sigma-2\varepsilon},\frac{1}{2}-(15d+34) \frac{\alpha}{\sigma-2\varepsilon} \right]\cap \mathbb{Q}, \\
 \tau \in \left[(20d+48) \frac{\alpha}{\sigma-2\varepsilon}, 6\varsigma-2+(20d+44)\frac{\alpha}{\sigma-2\varepsilon} \right] \\
\intertext{and}
p>1,\\
 p\leq \frac{2(d-1)\varsigma+\tau}{2d\varsigma+\tau-1+(20d+46)\frac{\alpha}{\sigma-2\varepsilon}} \wedge \frac{4(d-1)\varsigma +2\tau}{2(d-1)\varsigma+\tau+(d-2)\frac{\ln(f(q+1))}{\ln(\lambda_{q+1})}+(20d+48) \frac{\alpha}{\sigma-2\varepsilon}}
\end{gather*}
hold, which lead finally to the choice of the above parameters. More details are provided in Appendix \ref{Appendix A.1}. \\ 
 We furthermore require 
\begin{align} \label{rational number}
\varsigma:=\frac{\varsigma_n}{\varsigma_d}
\end{align}
to be a rational number with $\varsigma_n,\, \varsigma_d \in \mathbb{N}$. This ensures $\lambda r_\perp \in \mathbb{N}$, which is needed to construct the intermittent jets in Section \ref{Perturbation}.\\
Keeping this in mind allows to formulate the centerpiece of the convex integration scheme by fixing the iterative assumptions in the following section.

  \subsection{Iterative Assumptions}
All necessary parameters to run the convex integration scheme on the decomposed power-law equations \eqref{Convex Integration Equation} are settled. That is, by induction on $q$ we intend to find a process $(v_q,\mathring{R}_q)$, which solves
\begin{align} \label{Iterative Equation}
\begin{split}
\partial_t v_q+\divs\big((v_q+z_q)\otimes (v_q+z_q) \big)-\divs\left(\mathcal{A}(Dv_q+Dz_q) \right)+\Delta z_q-z_q+ \nabla p_q&=\divs(\oRq),\\
\divs(v_q)&=0
\end{split}
\end{align}
on $\Omega \times [o_q,\infty)\times \mathbb{T}^d $, where $z_q$ is defined as
\begin{align} \label{iteration z}
z_q:=\mathbb{P}_{\leq f(q)}z
\end{align}
on $\Omega \times [o_q,\infty)\times \mathbb{T}^d $ with frequency
\begin{align} \label{frequency z}
f(q)\in \Big( \lambda_q^{\frac{2\beta b^3}{(\iota-1)\sigma}},\lambda_q^{\frac{2\alpha b (\iota -1 )-2\beta b^3}{(\iota -1) (1-\sigma)}} \Big)\cap \Big( \lambda_q^{\frac{2\beta b^3}{(\iota-1)\sigma}},\lambda_q^{2\varsigma} \Big).
\end{align}
Note that the condition $(\iota-1)\sigma\alpha> 2\beta b^3 $ bewares the set on the right hand side of \eqref{frequency z} to be empty. The choice of $f(q)$ is crucial for controlling the energy, as well as the Linear and Nonlinear Error in Section \ref{Control of the Energy} and Section \ref{Inductive Estimates for the Reynolds Stress}, respectively (cf. Appendix \ref{Appendix A.1}).\\
Throughout the paper we will often need the following bounds
\begin{lemma}\thlabel{Lemma Properties z}
The process $z_q$ and its mollification $z_\ell$, defined in \eqref{iteration z} and Section~\ref{Mollification}, respectively, enjoy
\begin{subequations}
\begin{align}
&\|z_q\|_{L_\Omega^j C_I L_x^\infty}\lesssim  f^{\frac{d-2}{2}}(q)\|z\|_{L_\Omega^jC_IH_x^{1+\sigma}},\label{property za} \\ 
&\|z_{q+1}-z_q\|_{L_\Omega^j C_I H_x^s}\lesssim  \left\{\begin{array}{ll}
f^{s-(1+\sigma)}(q) \|z\|_{L_\Omega^jC_I H_x^{1+\sigma}}, & s\in [0,1+\sigma),\\
f^{s-(1+\sigma)}(q+1) \|z\|_{L_\Omega^jC_I H_x^{1+\sigma}}, & s\geq 1+\sigma ,\\
\end{array} \right.\label{property zb}\\
&\|z_q-z_\ell\|_{L_\Omega^jC_IW_x^{k,p}} \lesssim \Big( \ell^{\sigma/2-\varepsilon}+\ell f^{1-\sigma}(q) \Big)\| z\|_{L_\Omega^jC_\Il^{0,\sigma/2-\varepsilon} H_x^{k+\sigma}}\label{property zc}
\end{align}
\end{subequations}
for all $q \in \mathbb{N}_0$, $j\geq 1$, $k=0,1,\, p\in[1,2]$, every interval $I\subseteq \mathbb{R}$ of length $1$ and $I_\ell$ as in~\eqref{Interval}.
\end{lemma}
\noindent We refrain from providing a detailed proof, as the first assertion is an immediate consequence of the Sobolev embedding $H^{\frac{d+2\sigma}{2}}(\mathbb{T}^d)\subseteq L^\infty(\mathbb{T}^d)$. The derivation of the subsequent two bounds follow by standard arguments. \\
Now let us adopt the convention $\sum\limits_{r=1}^0:=0$. The solution $(v_q,\mathring{R}_q)$ shall then additionally comply with 
\begin{subequations}
\begin{align} 
&\|  v_q\|_{L_\Omega^{2J}C_IL_x^2}\leq M_0 \bar{e}^{1/2}\sum\limits_{r=1}^q \delta_{r}^{1/2}\label{key bound a},\\
&\|  v_q\|_{L_\Omega^{j}C_IL_x^2}\leq M_0(5 jLA^{q-1} )^{5A^{q-1}}+M_0 \bar{e}^{1/2}\sum\limits_{r=1}^q\delta_{r}^{1/2} \label{key bound b},\\
&\|  v_q\|_{L_\Omega^jC_{I,x}^{N} }\leq \Big[2\big(N+4\big)jLA^{q-1} \Big]^{2(N+4) A^{q-1}}\lambda_q^{\frac{N+1}{2}d+N+2} \label{key bound c},\\ 
&\|  v_q\|_{L_\Omega^{2J} C_I W_x^{1,\iota}}\leq  \sum\limits_{r=1}^q\delta_{r}^{1/2}\label{key bound d},\\ 
&\| \mathring{R}_q\|_{L_\Omega^{J}C_{I}L_x^1 }\leq 
 \frac{1}{144}\underline{e}  \,\delta_{q+2} \label{key bound e},\\ 
&\| \mathring{R}_q\|_{L_\Omega^jC_{I} L_x^1}\leq (2 jLA^q )^{2A^q}\label{key bound f}
\end{align}
\end{subequations}
\noindent for every interval $I\subseteq [o_q,\infty)$ of length $1$, every $q\in \mathbb{N}_0$, all $j \geq 1$, any $N=1,2,3,4$ and  fixed $J\geq 1$. $M_0$ and $L$ are universal constants, chosen in such a fashion that they satisfy the constraints \eqref{M_0} and \eqref{L}, respectively.\footnote{ The specific value of $\frac{1}{144}$ in \eqref{key bound e} originates from the estimates in \eqref{specific value of the key bound}.}  \\
To obtain the energy equality \eqref{Energy Equation} and as a byproduct also non-uniqueness of these solutions we have to introduce the auxiliary energy
 \begin{align} 
\mathscr{H}\{u\}(t):=\E\|u(t)\|_{L_x^2}^2+\frac{2H}{t+H}\E\bigg[\int_0^{t} \int_{\mathbb{T}^d} \mathcal{A}(Du(s,x)) \colon Du (s,x)\,dx \,ds\bigg]\label{help energy}
\end{align}
for all $t\in \mathbb{R}$ and $H>0$ sufficiently large. We force the iterations $\mathscr{H}\{v_q+z_q\}$ to gradually approximate $e$. Strictly speaking for all $q \in \mathbb{N}_0$, we additionally require\footnote{Here it becomes apparent that the prescribed energy $e$ must be bounded from below (cf. Section~\ref{Start of the Iteration}).}
\begin{align}\label{approximation energy}
\Big| e(t)\big(1-4\delta_{q+1}\big)- \mathscr{H}\{v_q+z_q\}(t)\Big|\leq  \delta_{q+1}\underline{e}
\end{align}
at every time $t\in [o_q,\infty)$.

\begin{remark}
In contrast to previous works, the auxiliary energy $\mathscr{H}$ is newly introduced and designed to eliminate any explicit dependence of the bounds on $t$. The inclusion of the additional factor 
$\tfrac{H}{t+H}$ in front of the nonlinear term is crucial for establishing global in time bounds for $I$, in particular for \eqref{Reason of H}, and for VIII on p.~35. This refinement not only preserves the energy inequality but also allows us to deduce the ergodicity of the solutions. \\
The constant $H$ must be chosen carefully to respect all requirements during the convex integration scheme, in particular within the Leray--Hopf framework when we show $\mathcal{H}\{u\}\equiv e$ for $e(t)=H \tr(G)$. More precisely, to maintain consistency with \eqref{Consistent}, we choose $H$ in a way that $H \geq \frac{2304 L^2 J + 4 \hbar C_{\iota, \nu_0, \nu_1} K_G^2}{3 \tr(G)}$ holds. At the same time, $H$ must prevent the factor $\frac{H}{t-u+H}$ in \eqref{Reason of H} from blowing up at $t=o_q$ and $u=\ell$, which is the reason why we additionally require $H>2$.
\end{remark}

\noindent The above considerations are summarized as follows. 
 
\begin{prop}[Main Iteration] \thlabel{Proposition Main Iteration}
There exist an $(\mathcal{F}_t)_{t\in [o_q,\infty)}$-adapted solution $\big(v_q,\mathring{R}_q\big)$ to \eqref{Iterative Equation} on $\Omega \times [o_q,\infty)\times \mathbb{T}^d$ satisfying \eqref{key bound a} -- \eqref{key bound f} as well as \eqref{approximation energy} at each level $q \in \mathbb{N}_0$. Moreover, $\big(\mathring{R}_q\big)_{q\in \mathbb{N}_0}$ is a vanishing Cauchy sequence in
 \begin{align*}
\left( L^{J}\big(\Omega;C\big([0,\infty);L^1\big( \mathbb{T}^d;\mathbb{R}^{d\times d}\big)\big)\big),\sup_{|I|=1}\|\cdot\|_{L_\Omega^JC_I L_x^1} \right),
\end{align*}
while also $\left(v_q\right)_{q\in \mathbb{N}_0}$  is Cauchy in
\begin{align*}
\left( L^{2J}\big(\Omega;C\big([0,\infty);L^2\big( \mathbb{T}^d;\mathbb{R}^d\big)\cap W^{1,\iota}\big( \mathbb{T}^d;\mathbb{R}^d\big)\big)\big),\sup_{|I|=1}\|\cdot\|_{L_\Omega^{2J}C_I L_x^2}+\sup_{|I|=1}\|\cdot\|_{L_\Omega^{2J}C_I W_x^{1,\iota}} \right),
\end{align*}
converging to some limit which differs from $0$, but has deterministic initial data.
\end{prop}  
  
\begin{remark}
In contrast to \cite{HZZ25} and \cite{LZ23} it is indispensable to include higher order derivatives in \eqref{key bound c} to establish $C_{\rm{loc}}^{0,\gamma}W^{1+\gamma,\iota}$-regularity and thus to prove stationarity of the final solutions (cf. \eqref{target series} and Section~\ref{Tightness u}).\\ 
Moreover, it is worth mentioning that the extension of $B$ and $e$ to negative times, permits to consider the system \eqref{Iterative Equation} on $[o_q,\infty)$. 
\end{remark}

\section{End of the Proof of Theorem \ref{Theorem Weak Solution}} \label{End of the Proof}
Building upon the preceding result, we present the proof of \thref{Theorem Weak Solution}.
 \begin{proof}[Proof of \thref{Theorem Weak Solution}]~\\
 \vspace{-.5cm}
 \subsubsection{Existence} \label{Existence}
 Let $(v_q,\oRq)_{q \in \mathbb{N}_0}$ be a sequence, constructed as in \thref{Proposition Main Iteration}, so that it converges in
\begin{align*}
L^{2J}\big(\Omega;C\big([0,\infty);L^2\big( \mathbb{T}^d;\mathbb{R}^d\big)\cap W^{1,\iota}\big( \mathbb{T}^d;\mathbb{R}^d\big)\big)\big) \otimes L^{J}\big(\Omega;C\big([0,\infty);L^1\big( \mathbb{T}^d;\mathbb{R}^{d\times d}\big)\big)\big)
\end{align*}
to some limit $(v,0)$, having deterministic initial data. Particularly, let us select a subsequence $(v_q,\oRq)_{q \in Q}$ of $(v_q,\oRq)_{q \in \mathbb{N}_0}$, so that the convergence even takes place in
\begin{align*}
C\big([0,\infty);L^2\big( \mathbb{T}^d;\mathbb{R}^d\big)\cap W^{1,\iota}\big( \mathbb{T}^d;\mathbb{R}^d\big)\big)  \otimes C\big([0,\infty);L^1\big( \mathbb{T}^d;\mathbb{R}^{d\times d}\big)\big)
\end{align*}
 $\mathcal{P}$-almost surely. Moreover, let $z_q$ be the process defined in \eqref{iteration z} with $z$ satisfying \thref{Proposition SHE}, which converges by definition  $\mathcal{P}$-a.s. uniformly in space and time to~$z$.\\
Then $(v,z)$ is an analytically weak and probabilistically strong solution to \eqref{Convex Integration Equation}. That means
\begin{align} 
 \label{Limit Iterative Equation}
&\underbrace{\int_{\mathbb{T}^d}\Big(v_q(t,x)-v_q(0,x)\Big)\cdot\varphi(x) \,dx}_{=:\text{I}}-\underbrace{\int_0^t \int_{\mathbb{
T}^d}\big((v_q+z_q) \otimes ( v_q+z_q) \big)(s,x):\nabla \varphi^T(x)\,dx \,ds}_{=: \text{II}}\\
&\hspace{.3cm}+\underbrace{\int_0^t \int_{\mathbb{
T}^d}\mathcal{A}\big(D(v_q +z_q)\big)(s,x):\nabla \varphi^T(x)\,dx \,ds}_{=: \text{III}}+\underbrace{\int_0^t \int_{\mathbb{T}^d} z_q(s,x)\cdot \Delta\varphi(x)\,dx\,ds}_{=:\text{IV}} \notag\\
&\hspace{.3cm}-\underbrace{\int_0^t \int_{\mathbb{T}^d} z_q(s,x)\cdot \varphi(x)\,dx\,ds}_{=:\text{V}}-\underbrace{\int_0^t\int_{\mathbb{T}^d}  p_q(s,x)\cdot\divs(\varphi(x))\, dx \,ds}_{=:\text{VI}}\notag   \\
&= -\underbrace{\int_0^t \int_{\mathbb{T}^d} \mathring{R}_q(s,x): \nabla \varphi^T(x) \,dx \,ds}_{=:\text{VII}}\notag 
\end{align}
converges $\mathcal{P}$-a.s. at every fixed time $t \in [0,\infty)$ and for any test function $\varphi \in C^\infty\left( \mathbb{T}^d;\mathbb{R}^d\right)$. To be more precise, for any $q\in Q$ we apply Cauchy--Schwarz's inequality and \thref{Lemma 2.4} to accomplish\\
\noindent \begin{tikzpicture}[baseline=(char.base)]
\node(char)[draw,fill=white,
  shape=rounded rectangle,
  drop shadow={opacity=.5,shadow xshift=0pt},
  minimum width=.8cm]
  {\Large I};
\end{tikzpicture} 
\noindent \begin{align*}
&\bigg|\int_{\mathbb{T}^d}\Big( v_q(t,x)-v_q(0,x)\Big) \cdot\varphi(x) \,dx-\int_{\mathbb{T}^d}\Big(v(t,x)-v(0,x)\Big)\cdot\varphi(x) \,dx \bigg|\\ &\hspace{.5cm}\lesssim \sup_{|I|=1}\|v_q-v\|_{C_IL_x^2}\|\varphi\|_{C_x}  \overset{q \to \infty}{\longrightarrow} 0,
\end{align*}
\noindent \begin{tikzpicture}[baseline=(char.base)]
\node(char)[draw,fill=white,
  shape=rounded rectangle,
  drop shadow={opacity=.5,shadow xshift=0pt},
  minimum width=.8cm]
  {\Large II};
\end{tikzpicture}

\begin{align*}
 &\bigg|\int_0^t \int_{\mathbb{
T}^d}\big((v_q+z_q) \otimes (v_q+z_q)\big)(s,x):\nabla \varphi^T(x)\,dx \,ds\\
&\hspace{6.5cm}-\int_0^t \int_{\mathbb{
T}^d}\big((v+z) \otimes (v+z)\big)(s,x):\nabla \varphi^T(x)\,dx \,ds\bigg|\\
&\lesssim \|\varphi\|_{C_x^1} \, \int_0^t \int_{\mathbb{T}^d} \big\|(v_q+z_q) \otimes  (v_q+z_q-v-z)-(v+z-v_q-z_q)\otimes (v+z)\big\|_F(s,x)\,dx \, ds  \\
&\lesssim \|\varphi\|_{C_x^1} \int_0^t \int_{\mathbb{T}^d} \Big( \big|(v_q+z_q) (s,x)\big|+\big| (v+z)(s,x)\big|\Big) \, \big|(v_q-v)(s,x)+(z_q-z)(s,x)\big|\,dx \, ds\\
& \lesssim t\|\varphi\|_{C_x^1}\, \Bigg( \sup_{|I|=1}\|v_q+z_q\|_{C_{I}L_x^2}+\sup_{|I|=1}\|v+z\|_{C_I L_x^2}\Bigg)\\
&\hspace{7.6cm}\cdot\Bigg( \sup_{|I|=1}\|v_q-v\|_{C_IL_x^2}+\sup_{|I|=1}\|z_q-z\|_{C_I L_x^2}\Bigg)\\
&\hspace{.5cm}\overset{q \to \infty}{\longrightarrow} 0,
\end{align*}
\pagebreak\\ 
\noindent \begin{tikzpicture}[baseline=(char.base)]
\node(char)[draw,fill=white,
  shape=rounded rectangle,
  drop shadow={opacity=.5,shadow xshift=0pt},
  minimum width=.8cm]
  {\Large III};
\end{tikzpicture}

\begin{align*}
 &\bigg|\int_0^t \int_{\mathbb{
T}^d}\mathcal{A}\big(D(v_q + z_q)(s,x)\big):\nabla \varphi^T(x)\,dx \,ds-\int_0^t \int_{\mathbb{
T}^d}\mathcal{A}\big( D(v +z)(s,x)\big):\nabla \varphi^T(x)\,dx \,ds\bigg|\\
&\lesssim \| \varphi\|_{C_x^1} \int_0^t \int_{\mathbb{
T}^d}\|\mathcal{A}\big(D(v_q + z_q)(s,x)\big)-\mathcal{A}\big( D(v +z)(s,x)\big)\|_F \,dx \,ds\\
&\lesssim \| \varphi\|_{C_x^1}\int_0^t \int_{\mathbb{
T}^d} \max_{k=1,\iota-1}\|D(v_q + z_q-v -z)(s,x)\big)\|_F^k \,dx \,ds\\
&\lesssim t\| \varphi\|_{C_x^1}\max_{k=1,\iota-1}  \Big(\sup_{|I|=1}\|v_q -v\|_{C_IW_x^{1,\iota}}+\sup_{|I|=1}\|z_q -z\|_{C_IW_x^{1,\iota}}\Big)^k\\
& \hspace{.5cm}\overset{q \to \infty}{\longrightarrow} 0,
\end{align*}

\noindent \begin{tikzpicture}[baseline=(char.base)]
\node(char)[draw,fill=white,
  shape=rounded rectangle,
  drop shadow={opacity=.5,shadow xshift=0pt},
  minimum width=.8cm]
  {\Large IV};
\end{tikzpicture} 

\begin{align*}
&\bigg|\int_0^t \int_{\mathbb{T}^d} z_q(s,x)\cdot \Delta\varphi(x)\,dx\,ds-\int_0^t \int_{\mathbb{T}^d}  z(s,x)\cdot\Delta\varphi(x)\,dx\,ds \bigg|\\
&\hspace{.5cm}\lesssim t \|\varphi\|_{C_x^2} \sup_{|I|=1}\|z_q-z\|_{C_IL_x^1}\overset{q \to \infty}{\longrightarrow} 0
\end{align*}
as well as \\
\noindent \begin{tikzpicture}[baseline=(char.base)]
\node(char)[draw,fill=white,
  shape=rounded rectangle,
  drop shadow={opacity=.5,shadow xshift=0pt},
  minimum width=.8cm]
  {\Large V};
\end{tikzpicture} 

\begin{align*}
&\bigg|\int_0^t \int_{\mathbb{T}^d} z_q(s,x)\cdot \varphi(x)\,dx\,ds-\int_0^t \int_{\mathbb{T}^d}  z(s,x)\cdot \varphi(x)\,dx\,ds \bigg|\\
&\hspace{.5cm}\lesssim t \|\varphi\|_{C_x}\sup_{|I|=1} \|z_q-z\|_{C_IL_x^1}\overset{q \to \infty}{\longrightarrow} 0
\end{align*}
$\mathcal{P}$-almost surely.\\
\noindent \begin{tikzpicture}[baseline=(char.base)]
\node(char)[draw,fill=white,
  shape=rounded rectangle,
  drop shadow={opacity=.5,shadow xshift=0pt},
  minimum width=.8cm]
  {\Large VI};
\end{tikzpicture} 
The sixth term vanishes because the test function $\varphi$ is chosen to be divergence free and moreover we have\\
\noindent \begin{tikzpicture}[baseline=(char.base)]
\node(char)[draw,fill=white,
  shape=rounded rectangle,
  drop shadow={opacity=.5,shadow xshift=0pt},
  minimum width=.8cm]
  {\Large VII};
\end{tikzpicture} 

\begin{align*}
&\bigg|\int_0^t \int_{\mathbb{T}^d} \mathring{R}_q(s,x): \nabla \varphi^T(x) \,dx \,ds \bigg|\lesssim  t \|\varphi\|_{C_x^1} \sup_{|I|=1} \|\mathring{R}_q\|_{C_IL_x^1}  \overset{q \to \infty}{\longrightarrow} 0
\end{align*}
$\mathcal{P}$-almost surely.\\
Passing to the limit on both sides of \eqref{Limit Iterative Equation} confirms that $(v,z)$ actually solves \eqref{Convex Integration Equation} in an analytically weak and probabilistically strong sense. Taking into account that $z$ is an analytically weak and probabilistically strong solution to \eqref{Stochastic Heat Equation}, so is $u:=v+z$ to \eqref{PLF} by superposition. 

\subsubsection{Regularity \& Bounds} \label{Regularity and Bound}
As we will see now this solution $u$ belongs $\mathcal{P}$-a.s. to the space\footnote{The final discussion in this section shows that it is crucial to select $\gamma \leq \frac{\sigma}{2}-\varepsilon$ sufficiently small in order to ensure that $z$ remains within the target space \eqref{target space}. This leads at last to the precise choice of $\beta$ in Section \ref{Choice of Parameters}.}
\begin{align}\label{target space} 
C\big([0,\infty);H^\gamma\big( \mathbb{T}^d;\mathbb{R}^d\big)\big)\cap C_{\rm{loc}}^{0,\gamma}\big([0,\infty);L^{2}\big( \mathbb{T}^d;\mathbb{R}^d\big)\cap W^{1+\gamma,\iota}\big( \mathbb{T}^d;\mathbb{R}^d\big) \big)
\end{align}
for some $\gamma \in \left(0, 1- \sqrt{\frac{2d+6}{2d+6+\beta}}\, \right)$. \\
Remembering that $H^\gamma$ is the real interpolation space between $L^2$ and $H^1$ with exponent $\gamma$ we invoke \eqref{property va} to infer 
\begin{align*}
\sum_{q\geq 0} \|v_{q+1}-v_q\|_{C_IH_x^\gamma}
& \lesssim \sum_{q\geq 0} \|v_{q+1}-v_q\|_{C_IH_x^1}^\gamma \|v_{q+1}-v_q\|^{1-\gamma}_{C_IL_x^2}\\
&\lesssim \sum_{q\geq 0} \left(\ell \|v_q\|_{C_{\Il,x}^2}+\|w_{q+1}\|_{C_{I,x}^1}\right)^\gamma \left(\ell \|v_q\|_{C_{\Il,x}^1}+\|w_{q+1}\|_{C_IL_x^2}\right)^{1-\gamma},
\end{align*}
whereas the mean value theorem and \eqref{property va} again yield 
\begin{align*}
&\sum_{q\geq 0} \|v_{q+1}-v_q\|_{C^{0,\gamma}_I L_x^2}\\
&\lesssim \sum_{q\geq 0}\bigg\{\|v_{q+1}-v_q\|_{C_IL_x^2}+\sup_{\substack{s,t \in I\\s\neq t }} \left(\frac{\|(v_{q+1}-v_q)(t)-(v_{q+1}-v_q)(s)\|_{L_x^2}}{|t-s|}\right)^{\gamma}\|v_{q+1}-v_q\|_{C_IL_x^2}^{1-\gamma}\bigg\}\\
&\lesssim \sum_{q\geq 0}\|v_{q+1}-v_q\|_{C_I L_x^2}+\sum_{q\geq 0} \|v_{q+1}-v_q\|_{C_{I}^1 L_x^2}^\gamma \|v_{q+1}-v_q\|^{1-\gamma}_{ C_I L_x^2}\\
&\lesssim \sum_{q\geq 0}\left( \ell \|v_q\|_{C_{\Il,x}^1}+\|w_{q+1}\|_{C_IL_x^2}\right)\\
&\hspace{.5cm}+\sum_{q\geq 0}\left(\ell \|v_q\|_{C_{\Il,x}^2}+\|w_{q+1}\|_{C_{I,x}^1}\right)^\gamma \left(\ell \|v_q\|_{C_{\Il,x}^1}+\|w_{q+1}\|_{C_IL_x^2}\right)^{1-\gamma}
\end{align*}
for any interval $I\subseteq [0,\infty)$ of length $1$.\\
Taking additionally into account that $W^{1+\gamma,\iota}$ is the real interpolation space between $W^{1,\iota}$ and $W^{2,\iota}$ with exponent $\gamma$, we deduce in a similar manner
\begin{align} \label{target series}
\begin{split}
&\sum_{q\geq 0} \|v_{q+1}-v_q\|_{C^{0,\gamma}_I W_x^{1+\gamma,\iota}}\\ &\hspace{.5cm}\lesssim  \sum_{q\geq 0} \left( \ell \|v_q\|_{C_{\Il,x}^3}+\|w_{q+1}\|_{C_IW_x^{1,\iota}}^{1-\gamma}\|w_{q+1}\|_{C_IW_x^{2,\iota}}^\gamma\right)\\
&\hspace{.5cm}+\sum_{q\geq 0}\left(\ell \|v_q\|_{C_{\Il,x}^4}+\|w_{q+1}\|_{C_{I,x}^3}\right)^\gamma \left(\ell \|v_q\|_{C_{\Il,x}^3}+\|w_{q+1}\|_{C_IW_x^{1,\iota}}^{1-\gamma}\|w_{q+1}\|_{C_IW_x^{2,\iota}}^\gamma\right)^{1-\gamma}.
\end{split}
\end{align}
Consequently Hölder's inequality with exponents $\frac{1}{\gamma}$ and $\frac{1}{1-\gamma}$, \eqref{key bound c}, \eqref{wC}, \eqref{wWiota} and \eqref{wL2a} give 
\begin{align*} 
&\sum_{q\geq 0} \left( \|v_{q+1}-v_q\|_{L_\Omega^{2J} C_I H_x^\gamma}+  \|v_{q+1}-v_q\|_{L_\Omega^{2J} C^{0,\gamma}_I L_x^2}+ \|v_{q+1}-v_q\|_{L_\Omega^{2J} C^{0,\gamma}_I W_x^{1+\gamma,\iota}} \right)\\
&\lesssim  \sum_{q\geq 0} \left(1+\ell \|v_q\|_{L_\Omega^{2J}C_{\Il,x}^4}+\|w_{q+1}\|_{L_\Omega^{2J}C_{I,x}^3}\right)^\gamma \\
&\hspace{.5cm}\cdot\max_{k=1,1-\gamma}\left(\ell \|v_q\|_{L_\Omega^{2J}C_{\Il,x}^4}+\|w_{q+1}\|_{L_\Omega^{2J}C_IW_x^{1,\iota}}^{1-\gamma}\|w_{q+1}\|_{L_\Omega^{2J}C_{I,x}^3}^\gamma+\|w_{q+1}\|_{L_\Omega^{2J}C_IL_x^2}\right)^k\\
&\lesssim \sum_{q\geq 0} \left(1+\ell(32JLA^{q-1})^{16A^{q-1}}\lambda_q^{\frac{5}{2} d+6}+  (28JLA^q)^{14A^q}\lambda_{q+1}^{2 d+5}\right)^\gamma \\
& \cdot \max_{k=1,1-\gamma}\left( \ell (32JLA^{q-1})^{16A^{q-1}}\lambda_q^{\frac{5}{2}d+6}+\big[(14JLA^q)^{7A^q}\lambda_{q+1}^{-2\alpha}\big]^{1-\gamma}\big[(28JLA^q)^{14A^q}\lambda_{q+1}^{2 d+5}\big]^\gamma\right.\\
&\hspace{1.5cm}+\left.\delta_{q+1}^{1/2}+(10JLA^q)^{5A^q}\lambda_{q+1}^{-2\alpha}\right)^k\\
&\lesssim \sum_{q\geq 0} \left(1+\lambda_{q-1}^{16}\lambda_q^{\frac{d}{2}-4} \lambda_{q+1}^{-2\alpha }+  \lambda_q^{14}\lambda_{q+1}^{2d+5}\right)^\gamma \\
&\hspace{.5cm}\cdot \max_{k=1,1-\gamma}\left(  \lambda_{q-1}^{16}\lambda_q^{\frac{d}{2}-4}\lambda_{q+1}^{-2\alpha}+\big[\lambda_q^7 \lambda_{q+1}^{-\alpha}\lambda_{q+1}^{-\beta}\big]^{1-\gamma}\big[ \lambda_q^{14} \lambda_{q+1}^{2d+5}\big]^\gamma+\lambda_{q+1}^{-\beta}+\lambda_q^5 \lambda_{q+1}^{-2\alpha}\right)^k\\
&\lesssim \sum_{q\geq 0}  \lambda_{q+1}^{(2 d+6)\gamma} \lambda_{q+1}^{[(2d+6+\beta)\gamma-\beta](1-\gamma)}\\
&\lesssim  \sum_{q\geq 0} \lambda_{q+1}^{-(2d+6+\beta)\left\{(\gamma-1)^2-\frac{2d+6}{2d+6+\beta}\right\}}\\
&<\infty.
\end{align*}
That means 
\begin{align*}
\left(\sum\limits_{q=0}^n \left( \|v_{q+1}-v_q\|_{L_\Omega^{2J}C_{I}H_x^{\gamma}} +\|v_{q+1}-v_q\|_{L_\Omega^{2J}C^{0,\gamma}_{I}L_x^2}+\|v_{q+1}-v_q\|_{L_\Omega^{2J}C^{0,\gamma}_{I}W_x^{1+\gamma,\iota}} \right)\right)_{n \in \mathbb{N}_0}
\end{align*}
converges in $\mathbb{R}$ along a subsequence, is consequently Cauchy and hence we find a subsequence of $(v_q)_{q\in Q}$ that converges $\mathcal{P}$-a.s. to $v$ in \eqref{target space}. More precisely, for every $\widetilde{\varepsilon}>0$ we find some $N \in \mathbb{N}_0$, so that 
\begin{align*}
&\|v_n-v_k\|_{L_\Omega^{2J}C_{I}H_x^{\gamma}} +\|v_n-v_k\|_{L_\Omega^{2J}C^{0,\gamma}_{I}L_x^2} +\|v_n-v_k\|_{L_\Omega^{2J}C^{0,\gamma}_{I}W_x^{1+\gamma,\iota}}\\
 &\hspace{.5cm}\leq \sum_{q=k}^{n-1} \left( \|v_{q+1}-v_q\|_{L_\Omega^{2J} C_I H_x^\gamma}+  \|v_{q+1}-v_q\|_{L_\Omega^{2J} C^{0,\gamma}_I L_x^2}+\|v_{q+1}-v_q\|_{L_\Omega^{2J}C^{0,\gamma}_{I}W_x^{1+\gamma,\iota}}  \right)\leq \widetilde{\varepsilon}
\end{align*}
holds for all $N \leq k < n$. \\
Recalling that $z$ is $\mathcal{P}$-a.s. of class $C_{\rm{loc}}^{0,\frac{\sigma}{2}-\varepsilon}\big([0,\infty);H^{1+\sigma}\big( \mathbb{T}^d;\mathbb{R}^d\big)\big)$ (see \thref{Proposition SHE}) and noting that our choice of $\beta \leq \frac{(2d+6)(\sigma-2\varepsilon)(4-\sigma +2\varepsilon)}{(2-\sigma+2\varepsilon)^2}$ ensures $ \gamma \leq \frac{\sigma}{2}-\varepsilon$, we conclude that also $u=v+z$ belongs $\mathcal{P}$-a.s. to \eqref{target space}. \\
Moreover, \eqref{vWiota} yields 
\begin{align*}
\|v\|_{L_\Omega^{2J} C_I W_x^{1,\iota}}&\lesssim  \|v-v_{q+1}\|_{L_\Omega^{2J} C_I W_x^{1,\iota}} + \|v_{q+1}\|_{L_\Omega^{2J} C_I W_x^{1,\iota}}\\ 
&\lesssim \widetilde{\varepsilon} + \sum_{r=0}^{q} \Big( \lambda_{r}^{\frac{3}{2}d+5-\alpha b}+\lambda_{r}^{7-\alpha b} \Big) \lambda_{r+1}^{-\beta }\\
&\lesssim  \widetilde{\varepsilon} + \sum_{r=0}^{q+1}(a^{-\beta b})^r -1\\
&\lesssim \widetilde{\varepsilon}+ \frac{1}{a^{\beta b}-1}
\end{align*}
for sufficiently large $q\in \mathbb{N}_0$, which confirms \thref{Remark Weak Solution}.
 \subsubsection{Energy Results} \label{Energy Results}
 In order to show that $u$ is even of Leray--Hopf class, we employ \thref{Lemma 2.4}, Hölder's inequality with $\iota,\, \frac{\iota}{\iota-1}$, \eqref{OPa} and \eqref{OPb} to accomplish \pagebreak
 \begin{align*} 
&|\mathscr{H}\{v_q+z_q\}(t)-\mathscr{H}\{v+z\}(t)|\\
&\lesssim \sup_{|I|=1} \|v_q+z_q-v-z\|_{L_\Omega^{2J} C_I L_x^2}\sup_{|I|=1} \|v_q+z_q+v+z\|_{L_\Omega^{2J} C_I L_x^2}   \\
&\hspace{.5cm}+\frac{H}{t+H}\E\bigg[\int_0^{t}  \|(v_q+z_q-v-z)(s) \|_{W_x^{1,\iota}} \left(1+\|(v_q+z_q)(s)\|^{\iota-1}_{W_x^{1,\iota}}\right)\,ds\bigg]\\
&\hspace{.5cm}+\frac{H}{t+H}\E\bigg[\int_0^{t}  \|(v_q+z_q-v-z)(s) \|_{W_x^{1,\iota}} \|(v+z)(s)\|^{\iota-1}_{W_x^{1,\iota}}\,ds\bigg]\\
&\lesssim  \Big(\sup_{|I|=1} \|v_q-v\|_{L_\Omega^{2J} C_I L_x^2}+\sup_{|I|=1} \|z_q-z\|_{L_\Omega^{2J} C_I L_x^2}\Big)\\
&\hspace{1cm} \cdot \Big( \sup_{|I|=1} \|v_q\|_{L_\Omega^{2J} C_I L_x^2}+\sup_{|I|=1} \|z\|_{L_\Omega^{2J} C_I L_x^2}+\sup_{|I|=1} \|v\|_{L_\Omega^{2J} C_I L_x^2}\Big)\\
&\hspace{.5cm}+ \frac{H}{t+H}\Big(\sup_{|I|=1} \|v_q-v \|_{L_\Omega^{2J}C_I W_x^{1,\iota}}+\sup_{|I|=1}  \|z_q-z \|_{L_\Omega^{2J}C_IH_x^{1+\sigma}} \Big)\\
&\hspace{1cm}\cdot \left(1+\sup_{|I|=1} \|v_q\|^{\iota-1}_{L_\Omega^{2J}C_I W_x^{1,\iota}}+\sup_{|I|=1} \|z\|^{\iota-1}_{L_\Omega^{2J}C_I H_x^{1+\sigma}}+\sup_{|I|=1} \|v\|^{\iota-1}_{L_\Omega^{2J}C_I W_x^{1,\iota}}\right)\\
&\overset{q \to \infty}{\longrightarrow}0
\end{align*}
at every time $t \in [0,\infty)$. In view of \eqref{approximation energy} we conclude $e\equiv\mathscr{H}\{u\}$. If we set $e(t)=H\tr(G)$, we actually get
\begin{align*}
\mathscr{E}\{u\}(t)&\leq\frac{t+H}{H}\mathscr{H}\{u\}(t)-\tr(G)t=\mathscr{H}\{u\}(0)=\mathscr{E}\{u\}(0)
 \end{align*}
\noindent for all $t \in [0,\infty)$.\\
Now it is easy to see that the additional factor $\frac{H}{t+H}$ in the definition of the auxiliary energy prevents to verify the energy inequality for more general times $s\leq t$ instead of $0$ (cf. Remark~\ref{Remark Leray--Hopf}). 
\subsubsection{Consistency}
We will now prove non-uniqueness of solutions. For this cause let us take two distinct energy profiles $e_1,\, e_2$, which satisfy all the requirements of \thref{Theorem Weak Solution} and in such a fashion that they coincide on $[0,T]$ for an arbitrary $T\geq 0$. Let $v_q^1,\, v_q^2$ the corresponding sequences, constructed via the convex integration scheme in the same way, but once depending on the energy $e_1$ and once on $e_2$. That means the intermittent jets and all the chosen parameters appearing in this scheme are identical for both sequences; only the underlying energy profiles differ from each other.\\
We will proceed inductively. So let us first mention that by our choice of the starting point \eqref{starting point}, we obviously know that $v_0^1=0$ and $v_0^2=0$ and hence $u_0^1=z_0$ and $u_0^2=z_0$ coincide on $\big[-\sum_{r=0}^\infty \ell_r,T\big]$. The series $\sum_{r=0}^\infty \ell_r$ is because of \eqref{l series bound} actually bounded by $1$, which gives rise to assume that $v_q^1$ equals to $v_q^2$ on $\big[ -\sum_{r=q}^\infty \ell_r, T\big]$ for some $q \in \mathbb{N}_0$.\\
In view of Figure \ref{Implication scheme} it is easy to see that the next iteration step $v_{q+1}^i$ depends on $v^i_{\ell}$ for $i=1,2$. Respecting the mollification of the velocity fields, implies together with the inductive assumption, that $v_{q+1}^1$ and $v_{q+1}^2$ coincide on $\big[- \sum_{r=q+1}^\infty \ell_r, T\big]$.\\
Since this is true for every $q\in \mathbb{N}_0$, we conclude $v^1(t)=v^2(t)$ for any $t \in [0, T]$ and therefore also $u^1\equiv u^2$ on $[0,T]$, completing the proof of \thref{Theorem Weak Solution}.
\begin{figure}[H]
\begin{center}
\includegraphics[scale=0.089]{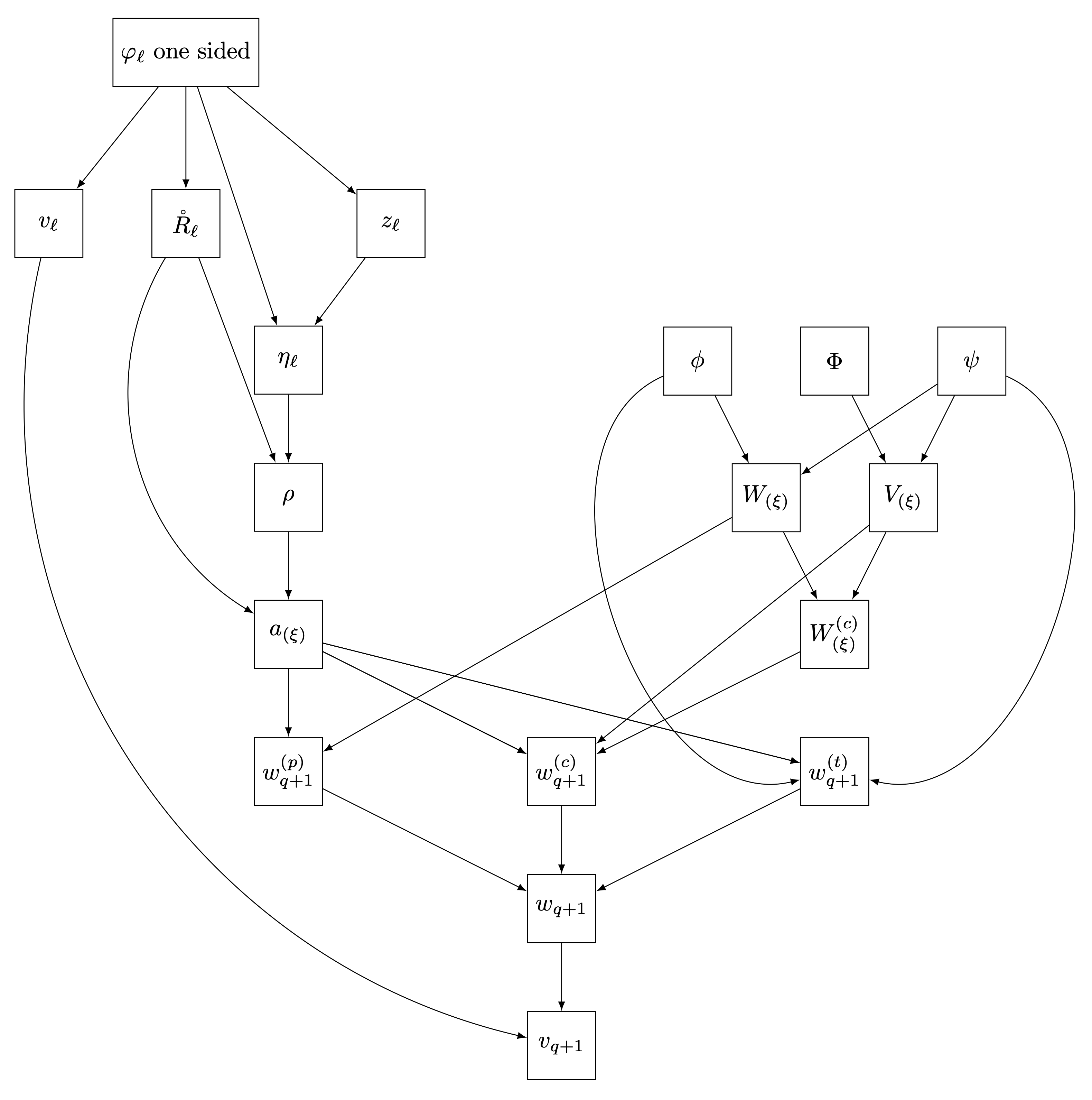}
\caption{Relationships among the components of the next iteration step $v_{q+1}$.}
\label{Implication scheme}
\end{center}
\end{figure}
 \end{proof}

\chapter{Convex Integration Scheme} \label{Convex Integration Scheme}
This section is dedicated to the proof of \thref{Proposition Main Iteration}.

\begin{proof}[Proof of \thref{Proposition Main Iteration}]~\\
The result will be established through mathematical induction. Initially, the focus will therefore be on determining an appropriate base case to initiate the process.
\section{Start of the Iteration}\label{Start of the Iteration}
Taking \thref{Lemma 2.1} into account permits to deduce 
\begin{align*}
\mathcal{R}\Delta v= \nabla v^T+(\nabla v^T)^T=2Dv
\end{align*}
for any divergence free vector field $v\in C^\infty\left( \mathbb{T}^d;\mathbb{R}^d\right)$, so that the pair 
\begin{align} \label{starting point}
(v_0,\mathring{R}_0):=(0,z_0\ootimes z_0-\mathcal{A}(Dz_0)+2D z_0-\mathcal{R}z_0)
\end{align}
 seems to be a suitable starting point, provided $p_0:=-\frac{1}{d}|z_0|^2$.\\
 While the velocity $v_0$ obviously respects the key bounds \eqref{key bound a} -- \eqref{key bound d}, \eqref{key bound e} and \eqref{key bound f} follow from the previous considerations together with \thref{Lemma 2.4} and \thref{Lemma 2.1}. More precisely, we have
\begin{align*}
\|\mathring{R}_0\|_{L_x^1}\lesssim \|z_0\|_{L_x^2}^2+C_{\nu_0}\|Dz_0\|_{L_x^2}+C_{\nu_1}\|Dz_0\|_{L_x^2}^{\iota-1}+\|Dz_0\|_{L_x^2}+\|z_0\|_{L_x^2},
\end{align*}
resulting, thanks to \eqref{OPa}, \eqref{OPb} and \thref{Proposition SHE}, in
\begin{align*}
\|\mathring{R}_0\|_{L_\Omega^{j}C_IL_x^1}\leq S\left[\|z\|_{L_\Omega^{2j}C_IL_x^2}^2+C_{\nu_0}\|z\|_{L_\Omega^jC_I H_x^1}+C_{\nu_1}\|z\|_{L_\Omega^{j(\iota-1)}C_I H_x^1}^{\iota-1}\right]\leq 2S(C_{\nu_0}+C_{\nu_1}) K^2_Gj
\end{align*}
 for any interval $I\subseteq [o_0,\infty)$ of length $1$, all $j\geq 1$ and some $S\geq 0$.\\
So choosing $L$ as 
\begin{align}\label{L}
L:=L_{G,\nu_0,\nu_1}:=\sqrt{2S(C_{\nu_0}+C_{\nu_1})} K_G
\end{align}
and $\underline{e}$ in a way that it enjoys $\underline{e} \geq 2304L^2J$ confirms \eqref{key bound e} and \eqref{key bound f}. \\
To verify \eqref{approximation energy} at level $q=0$ we invoke \thref{Lemma 2.4}, \eqref{OPa}, \eqref{OPb} and \thref{Proposition SHE} to derive
\begin{align*}
&\mathscr{H}\{z_0\} (t)\\
&\hspace{.3cm}\lesssim \sup_{|I|=1} \|z_0\|^2_{L_\Omega^2C_I L_x^2}+\frac{H}{t+H}\E \Bigg[\int_0^t \int_{\mathbb{T}^d} C_{\iota, \nu_0,\nu_1}\Big(\|Dz_0(s,x)\|_F+\|Dz_0(s,x)\|_F^\iota \Big) \,dx \,ds \Bigg]\\
&\hspace{.3cm}\lesssim \sup_{|I|=1} \|z_0\|^2_{L_\Omega^2C_I L_x^2}+C_{\iota, \nu_0,\nu_1}\Big(\sup_{|I|=1}\|z_0\|_{L_\Omega^1C_IW_x^{1,1}}+\sup_{|I|=1}\|z_0\|_{L_\Omega^\iota C_IW_x^{1,\iota}}^\iota \Big) \\
&\hspace{.3cm}\lesssim K_G^2+ C_{\iota, \nu_0,\nu_1}( K_G+K_G^\iota)\\
&\hspace{.3cm}\leq \hbar C_{\iota, \nu_0,\nu_1} K_G^2
\end{align*}
for all $t\in[o_0,\infty)$ and some constant $\hbar>0$. Consequently \eqref{approximation energy} remains true, whenever the given energy $e$ is bounded from below by $4\hbar C_{\iota, \nu_0,\nu_1} K_G^2 $.\\
That means, each of the iterative assumptions \eqref{key bound a} -- \eqref{key bound f} and \eqref{approximation energy} is fulfilled on the level $q=0$, if the underlying energy profile satisfies 
\begin{align}\label{Consistent}
\underline{e}\geq 2304L^2J+ 4 \hbar C_{\iota, \nu_0,\nu_1} K_G^2.
\end{align} 
Additionally, if we set $\mathcal{F}_t=\mathcal{F}_0$ for any $t\in [o_0,0]$, it is clear that the velocity $v_0(t)$ is $\left(\mathcal{F}_t\right)_{t\in [o_0,\infty)}$-adapted at any time $t \in [o_0,\infty)$.

\section{Construction of $v_{q+1}$}
For the induction step, we will construct the velocity at level $q+1$ in dependence of the previous mollified iteration step $v_\ell$, perturbed by a highly oscillatory error term $w_{q+1}$. 
For short we set 
\begin{align*}
v_{q+1}:=v_\ell+w_{q+1}.
\end{align*}
The building block of the perturbation will be chosen to be intermittent in the sense of intermittent jets, which are originally introduced in \cite{BCV21}. However, we will need the generalizations in $d$-dimensions, worked out in \cite{LZ23}. The idea to define the velocity to be intermittent stems from the famous K41 theory, roughly speaking, from the experimental deviation of the K41-laws (cf. \cite{Fr95}). A detailed summary of the physical motivation behind intermittency can be found in \cite{BV19a}.

\subsection{Mollification} \label{Mollification}

Let us first have a closer look to the mollified velocity field $v_\ell$ in space and time. We introduce a spatial mollifier whose support is contained within the unit ball, and, to preserve the $\left( \mathcal{F}_t\right)_{t \in [o_{q+1}, \infty)}$-adaptedness, we introduce a temporal mollifier whose support is restricted to the positive interval $(0,1)$ (cf. \cite{Be23}). Strictly speaking, we consider $\phi \in C^\infty\big(\mathbb{R}^d;[0,\infty)\big)$ and $\varphi \in C^\infty\big(\mathbb{R};[0,\infty)\big)$, complying with $ \supp \phi \subseteq B_1^{\mathbb{R}^{d}}(0),\, \supp \varphi \subseteq (0,1)$ and normalize them in a way that $\int_{\mathbb{R}^d}\phi(x)\, dx=1$ and $\int_{\mathbb{R}}\varphi(t)\, dt=1$ hold. Finally, we set $\phi_\ell:=\frac{1}{\ell^d}\phi(\frac{\cdot}{\ell})$ and $\varphi_\ell:=\frac{1}{\ell}\varphi(\frac{\cdot}{\ell})$ with $\ell=\ell_q$. With these at hand we mollify the velocity field as well as the Reynolds stress and $z_q$ as follows
\begin{align*}
v_\ell:=(v_q \ast_t \varphi_\ell)\ast_x\phi_\ell, \qquad \mathring{R}_\ell:= (\mathring{R}_q \ast_t \varphi_\ell)\ast_x\phi_\ell, \qquad z_\ell:=(z_q \ast_t \varphi_\ell)\ast_x\phi_\ell.
\end{align*}
The mollification has a slight impact on the velocity field in the relevant Hölder and Sobolev spaces, as detailed in the subsequent lemma. It relies primarily on standard mollification estimates, which is why we will not elaborate further.

\begin{lemma} \thlabel{Lemma Properties v}
The mollified velocity field $v_\ell$, defined above, enjoys the following bounds
\begin{subequations}
\begin{align}
&\|v_q-v_\ell\|_{C_I^N W_x^{k,p}}\lesssim \ell \|v_q\|_{C^{N+k+1}_{\Il,x}},\label{property va}\\
&\|v_\ell\|_{ C_I W_x^{s,p}}\leq \|v_q\|_{  C_{I_\ell} W_x^{s,p}},\label{property vb}\\
&\|v_\ell\|_{C_{I,x}^N}\lesssim \ell^{-N+n}\|v_q\|_{C_{I_\ell,x}^n},\label{property vc}\\
\begin{split}
&\|v_q+z_q-v_\ell-z_\ell\|_{C_IW_x^{s,p}}\lesssim \ell\Big[ \|\partial_t v_q\|_{C_\Il W_x^{s,p}}+ \|\nabla v_q\|_{C_\Il W_x^{s,p}}+ \|\nabla z_q\|_{C_\Il W_x^{s,p}}\Big]\\
&\hspace{4.5cm}+ \ell^{\sigma/2-\varepsilon} \|z_q\|_{C^{0,\sigma/2-\varepsilon}_\Il W_x^{s,p}}\label{property vd}
\end{split}
\end{align}
for any $1\leq p \leq \infty$, $s\geq 0,\, k\in \mathbb{N}_0,\, 0 \leq n\leq N$, all intervals $I\subseteq [o_{q+1},\infty)$ of length $1$ and $\Il$ as in \eqref{Interval}.
\end{subequations}
\end{lemma}
\noindent

\subsection{Perturbation} \label{Perturbation}
We now devote ourselves to the construction of the perturbation $w_{q+1}$, whose principal part is a superposition of the high oscillatory intermittent jets with lower frequent amplitudes.\\
The ensuing geometric lemma from \cite{BV19a} is indispensable for the definition of these amplitude functions as well as of the intermittent jets themselves. 
\begin{lemma} [Geometric Lemma]\thlabel{Geometric Lemma}
For a set of finite directions $\Lambda \subseteq \mathbb{S}^{d-1}\cap\mathbb{Q}^d$ there exists a family of smooth real-valued functions $\left(\gamma_\xi\right)_{\xi \in \Lambda}$ complying with 
\begin{align*}
R=\sum_{\xi\in\Lambda}\gamma_\xi^2(R)(\xi \otimes \xi)
\end{align*}
for every $R \in \mathbb{R}_{\text{symm}}^{d\times d} \cap \overline{B_\frac{1}{2}} (\Id)$.\footnote{Here $\mathbb{S}^{d-1}$ denotes the unit sphere in $d$-dimensions, whereas $\overline{B_\frac{1}{2}} (\Id)$ denotes the closed ball of radius $\frac{1}{2}$ around the identity matrix $\Id$ in the space of $d\times d$-matrices.}
\end{lemma}
\vspace{.5cm}

\noindent In the further course, the constant
\begin{align}
M:= 8 |\Lambda|(1+(2\pi)^d)^{1/2}\Big( \sup_{\xi \in \Lambda}\|\gamma_\xi\|_{C^{N}}+1\Big) \label{M}
\end{align}
with sufficiently large $N\geq 2N^\ast$ and $N^\ast \geq 2$ as in \eqref{Nast} becomes more important, especially when we estimate the perturbation and thus the next iteration step $v_{q+1}$.\\
 For a fixed $\xi\in \Lambda$ we choose further vectors $(A^i_\xi)_{i\in\{1,\ldots,d-1\}} \subseteq \mathbb{S}^{d-1} \cap \mathbb{Q}^d$ and a non-negative integer $n_\ast\in \mathbb{N}$, so that $\big\{n_\ast \xi, n_\ast A^1_\xi,\ldots, n_\ast A_\xi^{d-1}\big\}\subseteq \mathbb{Z}^d$ forms an orthogonal basis of $\mathbb{R}^d$.\\

\subsubsection*{Amplitude Functions}
The amplitude functions are defined in the spirit of \cite{HZZ25}. The main difference lies in the inclusion of the non-Newtonian part to ensure that not only the mean kinetic energy but even the auxiliary energy \eqref{help energy} of the iterated velocity approximates the prescribed energy $e$. More precisely, let the amplitude functions $a_{(\xi)}:=a_{\xi,q+1}$ be defined as 
\begin{align*}
a_{(\xi)}(\omega,t,x):=(2\pi)^{d/2} \rho^{1/2}(\omega,t,x)\gamma_\xi\left(\Id-\frac{\mathring{R_\ell}(\omega,t,x)}{\rho(\omega,t,x)}\right)
\end{align*}
with 
\begin{align*}
\rho(\omega,t,x)&:=2\sqrt{\ell^2+\|\mathring{R_\ell}(\omega,t,x)\|_F^2}+\eta_{\ell}(t)
\end{align*}
and
\begin{align*}
\eta_q(t)&:=  \frac{1}{3 (2\pi)^d}\Big[ e(t)\big(1-4\delta_{q+2}\big)- \mathscr{H}\{v_q+z_q\}(t)\Big]
\end{align*} 
for almost every $\omega \in \Omega$, any time $t \in [o_{q+1},\infty)$ and all $x \in \mathbb{T}^d$. The indexed $\ell$ in $\eta_\ell$ denotes as usual the mollification of $\eta_q$ in time. \\
In the light of \eqref{approximation energy} and of the required $a^{\beta b^2}\geq 3$ and $b \geq 2\varsigma_d$ (cf. Section \ref{Choice of Parameters}) we obtain 
\begin{align*}
3(2\pi)^d \eta_q\geq  e (3\delta_{q+1}-4 \delta_{q+2})\geq  0,
\end{align*}
which entails
\begin{align} \label{rho bound}
\rho \geq 2 \|\mathring{R}_\ell\|_\text{F} \qquad \text{and} \qquad \rho \geq 2\ell 
\end{align}
on $\Omega \times  [o_{q+1},\infty)\times \mathbb{T}^d$ for all $q\geq 0$.\\
These facts confirm that $\Id-\frac{\mathring{R}_\ell}{\rho}$ is actually in the domain of $\gamma_\xi$ and additionally permit to state the following lemma.

 \begin{lemma} \thlabel{Lemma Amplitude Functions}
The amplitude functions enjoy the following bounds
\begin{subequations}
\begin{align}
&\|a_{(\xi)}\|_{C_I L_x^2}\lesssim \frac{M}{|\Lambda|} \left( \|\oRq\|^{1/2}_{C_\Il L_x^1}+\bar{e}^{1/2}\delta_{q+1}^{1/2}\right) \label{amplitude function a} ,\\
&\|a_{(\xi)}\|_{C_{I,x}^{N}}\lesssim \left\{\begin{array}{ll} \frac{M}{|\Lambda|} \ell^{-\frac{d+1}{2}} \left( 1+\|\oRq\|_{C_\Il L_x^1}^{\frac{1}{2}}\right), & \text{if } N=0, \\
 \frac{M}{|\Lambda|} \ell^{-(d+3)N-\frac{3d+5}{2}} \left( 1+ \|\oRq\|_{C_\Il L_x^1}^{N+\frac{3}{2}} \right)   , & \text{if } N\geq 1,   \end{array}\right. \label{amplitude function b}  
\end{align}
for any interval $I\subseteq [o_{q+1},\infty)$ of length $1$ and $N\in \mathbb{N}_0$.
\end{subequations}
\end{lemma}
\noindent 
We refrain from providing a proof of this lemma as it closely follows the arguments in \cite{HZZ25}. Instead, we prefer to offer a brief summary of the construction of intermittent jets now, which, as previously mentioned, were originally introduced in \cite{BCV21} and later generalized in \cite{LZ23}.

 \subsubsection*{Intermittent Jets}
Let us start with three smooth functions
\begin{align*}
\Phi \colon \mathbb{R}^{d-1} \to \mathbb{R},\quad  \quad \phi \colon \mathbb{R}^{d-1} \to \mathbb{R}, \quad  \quad \psi \colon \mathbb{R} \to \mathbb{R}
\end{align*}
which satisfy the following properties\footnote{$B^{\mathbb{R}^{d-1}}_1(0)$ and $B^{\mathbb{R}}_1(0)$ denote the open balls of radius $1$, centered at the origin in $\mathbb{R}^{d-1}$ and $\mathbb{R}$, respectively.} 
  
\begin{subequations}
 \begin{flalign}
 & \bullet\, \Delta \Phi=-\phi, \label{building block a}\\[5pt]
&\bullet \, \supp \Phi,\, \supp \phi \subset B^{\mathbb{R}^{d-1}}_1(0), \quad \supp \psi \subset B^{\mathbb{R}}_1(0), \label{building block b}\\[5pt]
&\bullet\, \int_{B_1^{\mathbb{R}^{d-1}}(0)} \phi^2(x_1,\ldots,x_{d-1}) \,d(x_1,\ldots,x_{d-1})=(2 \pi)^{d-1}, \,  \int_{B_1^{\mathbb{R}}(0)} \psi^2(x_d) \,dx_d=2\pi,  \label{building block c}\\[5pt]
&\bullet \, \int_{\mathbb{T}^{d-1}} \phi (x_1,\ldots, x_{d-1})\, d(x_1,\ldots,x_{d-1})=0,\quad  \int_{\mathbb{T}} \psi (x_d)\, dx_d=0. \label{building block d}
\end{flalign}
\end{subequations}

\noindent 
These functions form the basic structure of the intermittent jets. We want the intermittent jets themselves to be supported in tubes of short length and thin radius. Moreover, they should oscillate with high frequency, depending on the parameters specified in Section \ref{Choice of Parameters}, which satisfy
\begin{align*}
0<\mu^{-1}\ll  \lambda^{-1} \ll r_\perp \ll r_\parallel \ll 1,\quad \quad \quad \lambda r_\perp \in \mathbb{N}. 
\end{align*}

\noindent To obtain the aforementioned intermittency we first rescale these functions as  
\begin{align*}
\phi_{r_\perp}(x_1,\ldots, x_{d-1})&:=\frac{1}{r_\perp^{(d-1)/2}}\phi\left(\frac{x_1}{r_\perp},\ldots, \frac{x_{d-1}}{r_\perp} \right),\\
\Phi_{r_\perp}(x_1,\ldots, x_{d-1})&:=\frac{1}{r_\perp^{(d-1)/2}}\Phi\left(\frac{x_1}{r_\perp},\ldots, \frac{x_{d-1}}{r_\perp} \right),\\
 \psi_{r_\parallel}(x_d)&:= \frac{1}{r_\parallel^{1/2}}\psi\left(\frac{x_d}{r_\parallel}\right)
\end{align*}
and periodize them on $\mathbb{T}^{d-1}, \mathbb{T}^{d-1}$ and $\mathbb{T}$, respectively. \\
In the light of \thref{Lemma A.1} we finally shift the functions by $\left(\alpha_\xi\right)_{\xi \in \Lambda} \subseteq \mathbb{R}^{d-1}$ in such a way that the families $\left(\phixi\right)_{\xi \in \Lambda}$ and $\left(\Phixi\right)_{\xi \in \Lambda}$ given by 
\begin{subequations}\label{rescaled building blocks}
\begin{align} 
\phixi(x)&:=\phi_{\xi,r_\perp,\lambda}(x):= \phi_{r_\perp}(n_\ast r_\perp \lambda(x-\alpha_\xi)\cdot A^1_\xi,\ldots,n_\ast r_\perp\lambda(x-\alpha_\xi)\cdot A_\xi^{d-1}),\label{rescaled building blocks a}\\
\Phixi(x)&:=\Phi_{\xi,r_\perp,\lambda}(x):= \Phi_{r_\perp}(n_\ast r_\perp \lambda(x-\alpha_\xi)\cdot A^1_\xi,\ldots,n_\ast r_\perp\lambda(x-\alpha_\xi)\cdot  A^{d-1}_\xi),\label{rescaled building blocks b}\\
\intertext{have mutually disjoint support, whereas $\psixi$ should additionally depend on time. For this reason we set}
\psixi(t,x)&:=\psi_{\xi,r_\perp,r_\parallel,\lambda,\mu}(t,x):=\psi_{r_\parallel}(n_\ast r_\perp \lambda (x\cdot \xi+\mu t)) \label{rescaled building blocks c}
\end{align}
\end{subequations}
for any $t\in \mathbb{R}$.\\
That is all we need to define an intermittent jet as
\begin{align*}
W_{(\xi)}(t,x):=W_{\xi,r_\perp,r_\parallel,\lambda, \mu}(t,x):=\xi \psixi(t,x)\phixi(x)
\end{align*}
and its incompressibility corrector as\footnote{To obtain the expression in \eqref{incompressibility corrector} we use \eqref{building block a} and the fact that $\xi, A_\xi^1,\ldots, A_\xi^{d-1}$ are orthogonal.}
\begin{align} \label{incompressibility corrector}
W_{(\xi)}^{(c)}(t,x)&:=\frac{1}{(n_\ast \lambda)^2} \nabla \Phixi (\xi \cdot \nabla) \psixi= \divs{V_{(\xi)}(t,x)}-W_{(\xi)}(t,x),
\end{align}
with $V_{(\xi)}(t,x):=\frac{1}{(n_\ast\lambda)^2}\Big(\nabla \Phixi (x) \otimes \xi -\xi \otimes \nabla \Phixi(x)\Big)\psixi(t,x)$ skew-symmetric, so that $W_{(\xi)}+W_{(\xi)}^{(c)}$ becomes divergence free.\\
\thref{Lemma A.1} then guarantees indeed that their supports remain mutually disjoint; we might just need to shift $\alpha_{\xi}$ within a hyperplane, spanned by $A^1_\xi, \ldots, A^{d-1}_\xi$. The orthogonal directions of oscillations entail the following important equality from \cite{Be23}, p. 15.
\begin{lemma}  \thlabel{Lemma Building Blocks}
The building blocks $\psixi$ and $\phixi$ obey
\begin{align*}
\|D^\alpha \partial_t^N \psi_{(\xi)}^n D^\beta  \phi_{(\xi)}^m\|_{L_x^p}= (2\pi)^{-d/p}\|D^\alpha \partial_t^N \psi_{(\xi)}^n\|_{L_x^p} \|D^\beta \phi_{(\xi)}^m\|_{L_x^p} 
\end{align*}
for each $n,m \in \mathbb{N}$, $p \in [1,\infty)$, $N \in \mathbb{N}_0$ and all multi-indices $\alpha,\beta \in \mathbb{N}_0^d$.
\end{lemma}
\noindent 
This leads in turn to the fundamental bounds (cf. \cite{LZ23}, p. 7).

\begin{lemma} \thlabel{Lemma Building Blocks Bounds}
For any $N,M \in \mathbb{N}_0 $ and $p \in [1,\infty]$ it holds

\begin{subequations}\label{fundamental bounds}
\begin{align}
&\sum_{|\alpha|\leq N}\|D^\alpha\partial_t^M\psixi\|_{CL_x^p} \lesssim r_{\parallel}^{\frac{1}{p}-\frac{1}{2}}\left(\frac{r_\perp \lambda}{r_\parallel}\right)^N \left(\frac{r_\perp \lambda \mu}{r_\parallel }\right)^M,\label{fundamental bounds a}\\
&\sum_{|\alpha|\leq N}\|D^\alpha\phixi\|_{L_x^p}+\sum_{|\beta|\leq N}\|D^\beta\Phixi\|_{L_x^p}\lesssim r_\perp^{\frac{d-1}{p}-\frac{d-1}{2}}\lambda^N, \label{fundamental bounds b}\\
&\sum_{|\alpha|\leq N} \|D^\alpha\partial_t^M W_{(\xi)}\|_{CL_x^p}+\frac{r_\parallel}{r_\perp}\sum_{|\beta|\leq N} \|D^\beta\partial_t^M W_{(\xi)}^{(c)}\|_{CL_x^p}+\lambda \sum_{|\gamma|\leq N} \|D^\gamma \partial_t^M V_{(\xi)}\|_{CL_x^p} \label{fundamental bounds c}\\  &\hspace{8.8cm}\lesssim r_\perp^{\frac{d-1}{p}-\frac{d-1}{2}}r_\parallel^{\frac{1}{p}-\frac{1}{2}}\lambda^N \left(\frac{r_\perp \lambda \mu}{r_\parallel}\right)^M\notag.
\end{align}
\end{subequations}
In the specific case $N=M=0$ and $p=2$, an intermittent jet even admits $\|\Wxi\|_{L_x^2}= 1$.
\end{lemma}
~\medskip\\
Now we have everything in place to give a precise definition of the perturbation, which consists of the principal part
\begin{align*}
\wpr:= \sum_{\xi \in \Lambda} \axi\Wxi ,
\end{align*}
the corresponding incompressibility corrector
\begin{align*}
\wc:&= \sum_{\xi \in \Lambda} \Big( \Vxi \nabla \axi+\axi\Wcxi\Big)
\end{align*}
and the temporal corrector
\begin{align*}
\wt:&= - \mu^{-1} \sum_{\xi \in \Lambda} \mathbb{P}\mathbb{P}_{\neq 0} \left(a^2_{(\xi)}  \psi^2_{(\xi)} \phi^2_{(\xi)} \xi \right).
\end{align*}
As we will see in Section \ref{Decomposition of the Reynolds Stress} the temporal corrector behaves well under time derivatives. The reason for introducing this corrector is, in fact, to control the highly oscillatory part of the Reynolds stress cf. \eqref{divRosc}. The purpose of the incompressibility corrector is already apparent. Since $V_{(\xi)}$ is by construction skew-symmetric, it is easy to see that
\begin{align} 
\wpr+\wc=  \sum_{\xi \in \Lambda} \divs{ \left(\axi \Vxi \right)}  \label{wp+wc}
\end{align}
becomes solenoidal for big $q \in \mathbb{N}_0$ and has zero mean. Since also the temporal corrector shares these characteristics, they carry over to the total perturbation
\begin{align*}
w_{q+1}:= \wpr+\wc+\wt.
\end{align*} 
Furthermore, we state the following lemma.

\begin{lemma}\thlabel{Lemma Perturbation}
For an arbitrary interval $I\subseteq [o_{q+1},\infty)$ of length $1$ each component of the perturbation $w_{q+1}$ is bounded
\begin{itemize}
\item[a)] in $C_{I} L^p$ for any $p \in (1,\infty)$ as
\begin{subequations}
\begin{align}
&\|\wpr\|_{C_{I} L_x^p }\lesssim M \ell^{-\frac{d+1}{2}} \left(1+ \|\oRq\|^{\frac{1}{2}}_{C_\Il L_x^1}\right)\, r_\perp^{\frac{d-1}{p}-\frac{d-1}{2}} r_\parallel^{\frac{1}{p}-\frac{1}{2}}, \label{wpL} \\
&\|\wc\|_{C_{I} L_x^p }\lesssim M \ell^{-\frac{5d+11}{2}}  \left(1+ \|\oRq\|^{\frac{5}{2}}_{C_\Il L_x^1}\right)\, r_\perp^{\frac{d-1}{p}-\frac{d-3}{2}} r_\parallel^{\frac{1}{p}-\frac{3}{2}}  , \label{wcL}\\
&\|\wt\|_{C_{I} L_x^p }\lesssim  \frac{M^2}{|\Lambda|} \ell^{-(d+1)} \left(1+\|\oRq\|_{C_\Il L_x^1}\right)\, r_\perp^{\frac{d-1}{p}-(d-1)} r_\parallel^{\frac{1}{p}-1} \mu^{-1}. \label{wtL} 
\end{align}
In the specific case $p=2$, the principal part admits the stronger bound
\begin{align}
\begin{split}
\|\wpr\|_{C_{I}L_x^2 }&\lesssim M\left( \|\oRq\|_{C_\Il L_x^1}^{\frac{1}{2}}+\bar{e}^{1/2}\delta_{q+1}^{1/2}\right)+ M \ell^{-\frac{5d+11}{2}} \left(1+ \|\oRq\|^{\frac{5}{2}}_{C_\Il L_x^1}\right)\, r_\perp^{-\frac{1}{2}}\lambda^{-\frac{1}{2}} . \label{wpL2}
\end{split}
\end{align}

\end{subequations}

\item[b)] in $C^N_{I,x}$ for $N\geq 1 $ as
\begin{subequations} 
\begin{align}
&\|\wpr\|_{C_{I,x}^N}\lesssim M \ell^{-(d+3)N-\frac{3d+5}{2}}  \left(1+ \|\oRq\|^{N+\frac{3}{2}}_{C_\Il L_x^1}\right)\, r_\perp^{-\frac{d-(2N+1)}{2}}r_\parallel^{-\frac{2N+1}{2}}  \lambda^{N} \mu^N , \label{wpC} \\
&\|\wc\|_{C_{I,x}^N}\lesssim M \ell^{-(d+3)N-\frac{5d+11}{2}}  \left(1+ \|\oRq\|^{N+\frac{5}{2}}_{C_\Il L_x^1}\right)\, r_\perp^{-\frac{d -(2N+3)}{2}}r_\parallel^{-\frac{2N+3}{2}} \lambda^{N} \mu^N  , \label{wcC}\\
&\|\wt\|_{C_{I,x}^N}\lesssim  \frac{M^2}{|\Lambda|} \ell^{-(d+3)N-(4d+8)}  \left(1+\|\oRq\|^{N+4}_{C_\Il L_x^1}\right)\, r_\perp^{-d+N+1}r_\parallel^{-(N+1)}  \lambda^{N+1} \mu^{N-1}  \label{wtC} .
\end{align}
\end{subequations}

\item[c)] in $C_{I}W^{k,p}$ for any $p\in (1,\infty)$ and $k\geq 1$ as 
\begin{subequations} 
\begin{align}
&\|\wpr+\wc\|_{C_{I}W_x^{k,p}} \lesssim M \ell^{-(d+3)k-\frac{5d+11}{2}}  \left(1+\|\oRq\|_{C_\Il L_x^1}^{k+\frac{5}{2}}\right) r_\perp^{\frac{d-1}{p}-\frac{d-1}{2}}r_\parallel^{\frac{1}{p}-\frac{1}{2}}\lambda^k ,\label{wpwcW} \\
&\|\wt\|_{C_{I}W_x^{k,p}} \lesssim \left\{\begin{array}{ll} \frac{M^2}{|\Lambda|} \ell^{-(3d+6)}  \left(1+\|\oRq\|_{C_{I_\ell} L_x^1}^{3}\right)  r_\perp^{\frac{d-1}{p}-(d-1)}r_\parallel^{\frac{1}{p}-1}\lambda \mu^{-1}, & k=1, \\
        \frac{M^2}{|\Lambda|} \ell^{-(d+3)k-(3d+5)}  \left(1+\|\oRq\|_{C_{I_\ell} L_x^1}^{k+3}\right)  r_\perp^{\frac{d-1}{p}-(d-1)}r_\parallel^{\frac{1}{p}-1}\lambda^k \mu^{-1} \label{wtW}, & k>1.\end{array}\right. 
\end{align}
\end{subequations}
\end{itemize}
\end{lemma}
\noindent The proof of this lemma essentially adheres to the methodology employed in \cite{HZZ23c} and \cite{Be23}.
\begin{proof}
 \underline{\eqref{wpL}:} \eqref{wpL} is a direct consequence of \eqref{amplitude function b} and \eqref{fundamental bounds c}.\\
 \underline{\eqref{wcL}:}  \eqref{wcL} follows also from \eqref{amplitude function b} and \eqref{fundamental bounds c}
 \begin{align*}
 \|\wc\|_{C_I L_x^p} &\lesssim \sum_{\xi \in \Lambda
 } \left(\|\axi\|_{C_{I,x}^1}\|\Vxi\|_{C_I L_x^p}+ \|\axi\|_{C_{I,x}^0}\|\Wcxi\|_{C_I L_x^p} \right)\\
 &\lesssim M \ell^{-\frac{5d+11}{2}} \left(1+\|\oRq\|^\frac{5}{2}_{C_\Il L_x^1}\right) r_\perp^{\frac{d-1}{p}-\frac{d-3}{2}}r_\parallel^{\frac{1}{p}-\frac{3}{2}}.
 \end{align*}
 \underline{\eqref{wtL}:} Keeping in mind that the Leray projection $\mathbb{P}$ and the projection onto mean free functions $\mathbb{P}_{\neq 0}$ are bounded on $L^p$, \thref{Lemma Building Blocks}, \eqref{amplitude function b}, \eqref{fundamental bounds a}, \eqref{fundamental bounds b} permit to deduce
 \begin{align*}
 \|\wt\|_{C_I L_x^p} &\lesssim \mu^{-1} \sum_{\xi \in \Lambda} \|\axi\|_{C_{I,x}^0}^2 \|\psixi\|^2_{C_I L_x^{2p}} \|\phixi\|_{ L_x^{2p}}^2\\
 &\lesssim  \frac{M^2}{|\Lambda|} \ell^{-(d+1)}\left(1+\|\oRq\|_{C_\Il L_x^1} \right)r_\perp^{\frac{d-1}{p}-(d-1)}r_\parallel^{\frac{1}{p}-1}\mu^{-1}.
 \end{align*}
 \underline{\eqref{wpL2}:} If $p=2$ we make use of the improved version of Hölder's inequality \thref{Lemma A.2} to obtain
 \begin{align*}
 \|\wpr\|_{C_I L_x^2}& \lesssim \sum_{\xi \in \Lambda} \left(\|\axi\|_{C_I L_x^2}+(r_\perp \lambda)^{-1/2}\|\axi\|_{C_{I,x}^1}\right)\bigg\|W_{(\xi)}\Big(\frac{\cdot}{ r_\perp \lambda}\Big)\bigg\|_{C_I L_x^2} \\
 & \lesssim M\left( \|\oRq\|_{C_\Il L_x^1}^{\frac{1}{2}}+\bar{e}^{1/2}\delta_{q+1}^{1/2}\right)+ M \ell^{-\frac{5d+11}{2}} \left(1+ \|\oRq\|^{\frac{5}{2}}_{C_\Il L_x^1}\right)\, r_\perp^{-\frac{1}{2}}\lambda^{-\frac{1}{2}},
 \end{align*}
 where the second step relies once again on \thref{Lemma Building Blocks Bounds} and \thref{Lemma Amplitude Functions}.\\
  \underline{\eqref{wpC}:} 
  We invoke Leibniz's formula, \eqref{amplitude function b} and \eqref{fundamental bounds c} to derive
  \begin{align*}
  \|\wpr\|_{C_{I,x}^N} &\lesssim \sum_{\xi \in \Lambda} \sum_{0 \leq  n\leq  N} \|\axi\|_{C_{I,x}^n} \|\Wxi\|_{C_{I,x}^{N-n}} \\
  & \lesssim M \ell^{-\frac{d+1}{2}} \left(1+\|\oRq\|_{C_\Il L_x^1}^{1/2}\right) r_\perp^{-\frac{d-1}{2}} r_\parallel^{-\frac{1}{2}} \left(\frac{r_\perp \lambda \mu }{r_\parallel} \right)^N\\
  &\hspace{.5cm}+ \sum_{0<n\leq N} M \ell^{-(d+3)n-\frac{3d+5}{2}}\left(1+\|\oRq\|_{C_\Il L_x^1}^{n+\frac{3}{2}}\right) r_\perp^{-\frac{d-1}{2}}r_\parallel^{-\frac{1}{2}} \left(\frac{r_\perp \lambda \mu}{r_\parallel} \right)^{N-n}\\
  &\lesssim M \ell^{-(d+3)N-\frac{3d+5}{2}}  \left(1+ \|\oRq\|^{N+\frac{3}{2}}_{C_\Il L_x^1}\right)\, r_\perp^{-\frac{d-(2N+1)}{2}}r_\parallel^{-\frac{2N+1}{2}}  \lambda^{N} \mu^N.
  \end{align*}
  \underline{\eqref{wcC}:} Owning to Leibniz's formula, \eqref{amplitude function b} and \eqref{fundamental bounds c} again we have
  \begin{align*}
  &\|\wc\|_{C_{I,x}^N}\\
   &\hspace{.5cm}\lesssim \sum_{\xi \in \Lambda}\sum_{0\leq n\leq N}\left( \|\axi\|_{C_{I,x}^{n+1}}\|\Vxi\|_{C_{I,x}^{N-n} }+  \|\axi\|_{C_{I,x}^n} \|\Wcxi\|_{C_{I,x}^{N-n}}\right)\\
 &\hspace{.5cm} \lesssim \sum_{0\leq n \leq N} M \ell^{-(d+3)n-\frac{5d+11}{2}} \left( 1+\|\oRq\|_{C_\Il L_x^1}^{n+\frac{5}{2}}\right) \left(\frac{r_\perp \lambda \mu}{r_\parallel} \right)^{N-n}  r_\perp^{-\frac{d-1}{2}}r_\parallel^{-\frac{1}{2}} \bigg(\lambda^{-1}+\frac{r_\perp}{r_\parallel}\bigg)\\
 &\hspace{.5cm}\lesssim M \ell^{-(d+3)N-\frac{5d+11}{2}}  \left(1+ \|\oRq\|^{N+\frac{5}{2}}_{C_\Il L_x^1}\right)\, r_\perp^{-\frac{d -(2N+3)}{2}}r_\parallel^{-\frac{2N+3}{2}} \lambda^{N} \mu^N.
  \end{align*}

\noindent  \underline{\eqref{wtC}:} In order to control $\wt$ in $C^N$ let us first mention that 
\begin{align*}
\partial_t \mathbb{P}_{\neq 0} \big(\axi^2 \psixi^2 \phixi^2\big)=\mathbb{P}_{\neq 0}\partial_t  \big(\axi^2 \psixi^2 \phixi^2\big)
\end{align*}
holds true since $ x\mapsto \big(\axi^2(x) \psixi^2(x) \phixi^2(x)\big)$ and $ x\mapsto \partial_t\big(\axi^2(x) \psixi^2(x) \phixi^2(x)\big)$ are as smooth functions on the torus bounded. Invoking additionally the Sobolev embedding $W^{1,p}(\mathbb{T}^d)\subset L^\infty(\mathbb{T}^d)$ for any $p\in (d, \infty) $, we find thanks to \eqref{OPa}, Leibniz's formula, \eqref{amplitude function b}, \eqref{fundamental bounds a} and \eqref{fundamental bounds b}
\begin{align*}
& \|\wt\|_{C_{I,x}^N}\\
  &\hspace{.5cm}\lesssim \mu^{-1} \sum_{0\leq |\alpha|\leq N} \sum_{\xi \in \Lambda} \left( \| \mathbb{P} \Pm  D_t^\alpha(\axi^2 \psixi^2 \phixi^2)\|_{C_I L_x^{p} }+\| \mathbb{P} \Pm \nabla D_t^\alpha(\axi^2 \psixi^2 \phixi^2)\|_{C_I L_x^{p} }\right)\\
  & \hspace{.5cm} \lesssim  \mu^{-1}\sum_{\substack{0\leq |\alpha|\leq N\\ \beta \leq \alpha \\ \gamma \leq \beta}} \sum_{\xi \in \Lambda}\| D_t^{\alpha-\beta}  \axi^2 D_t^{\beta-\gamma }\psixi^2 D_t^\gamma \phixi^2\|_{C_I W_x^{1,\infty }}\\
  &\hspace{.5cm}\lesssim  \frac{M^2}{|\Lambda|} \ell^{-(d+3)N-(3d+6)} \left(1+\|\oRq\|_{C_\Il L_x^1}^{N+3} \right)  \left(\frac{r_\perp \lambda \mu}{r_\parallel}\right)^N r_\perp^{-d+1}r_\parallel^{-1}\lambda \mu^{-1}\\
  & \hspace{1cm} +\frac{M^2}{|\Lambda|} \ell^{-(d+3)N-(4d+8)}\left(1+\|\oRq\|_{C_\Il L_x^1}^{N+4} \right) \left(\frac{r_\perp \lambda \mu}{r_\parallel}\right)^N r_\perp^{-d+1}r_\parallel^{-1}\lambda \mu^{-1}\\
  &\hspace{.5cm}\lesssim \frac{M^2}{|\Lambda|} \ell^{-(d+3)N-(4d+8)}\left(1+\|\oRq\|_{C_\Il L_x^1}^{N+4} \right)  r_\perp^{-d+N+1}r_\parallel^{-(N+1)}\lambda^{N+1} \mu^{N-1}.
 \end{align*}  
  \underline{\eqref{wpwcW}:} A direct computation gives
  \begin{align*}
 & \|\wpr+ \wc\|_{C_I W_x^{k,p}} \\
  &\hspace{.5cm} \lesssim \sum_{\xi \in \Lambda } \sum_{ 0\leq n\leq k+1} \|\axi\|_{C_{I,x}^{k+1-n}}\|\Vxi\|_{C_I W_x^{n,p}}\\
  &\hspace{.5cm}\lesssim M \ell^{-\frac{d+1}{2}} \left(1+\|\oRq\|_{C_\Il L_x^1}^{\frac{1}{2}} \right)r_\perp^{\frac{d-1}{p}-\frac{d-1}{2}}r_\parallel^{\frac{1}{p}-\frac{1}{2}} \lambda^k\\
  &\hspace{1cm}+ \sum_{ 0\leq n\leq k} M \ell^{-(d+3)(k+1-n)-\frac{3d+5}{2}} \left(1+\|\oRq\|_{C_\Il L_x^1}^{k+1-n+\frac{3}{2}} \right)r_\perp^{\frac{d-1}{p}-\frac{d-1}{2}}r_\parallel^{\frac{1}{p}-\frac{1}{2}} \lambda^{n-1}\\
  &\hspace{.5cm}\lesssim M \ell^{-(d+3)k-\frac{5d+11}{2}} \left(1+\|\oRq\|_{C_\Il L_x^1}^{k+\frac{5}{2}} \right)r_\perp^{\frac{d-1}{p}-\frac{d-1}{2}}r_\parallel^{\frac{1}{p}-\frac{1}{2}} \lambda^k,
  \end{align*}
  where we took \eqref{wp+wc}, \eqref{amplitude function b} and \eqref{fundamental bounds c} into account.\\
  \underline{\eqref{wtW}:} Proceeding in the same way but involving additionally \eqref{OPa} and \thref{Lemma Building Blocks} furnishes 
  \begin{align*}
  \|\wt\|_{C_I W_x^{1,p}} &\lesssim \mu^{-1} \sum_{\xi \in \Lambda}\|\axi\|_{C_{I,x}^0}^2 \|\psixi\|_{CL_x^{2p}}^2 \|\phixi\|_{L_x^{2p}}^2\\
  &\hspace{.5cm}+\mu^{-1} \sum_{\xi \in \Lambda}\|\axi\|_{C_{I,x}^1} \|\axi\|_{C_{I,x}^0} \|\psixi\|_{CL_x^{2p}}^2 \|\phixi\|_{L_x^{2p}}^2\\
  &\hspace{.5cm}+\mu^{-1} \sum_{\xi \in \Lambda}\|\axi\|_{C_{I,x}^0}^2 \|\psixi\|_{CW_x^{1,2p}}\|\psixi\|_{CL_x^{2p}} \|\phixi\|_{L_x^{2p}}^2\\
   &\hspace{.5cm}+\mu^{-1} \sum_{\xi \in \Lambda}\|\axi\|_{C_{I,x}^0}^2 \|\psixi\|_{CL_x^{2p}}^2 \|\phixi\|_{W_x^{1,2p}} \|\phixi\|_{L_x^{2p}} \\
&\lesssim   \frac{M^2}{|\Lambda|} \ell^{-(3d+6)}  \left(1+\|\oRq\|_{C_{I_\ell} L_x^1}^{3}\right)  r_\perp^{\frac{d-1}{p}-(d-1)}r_\parallel^{\frac{1}{p}-1}\lambda \mu^{-1} ,
  \end{align*}
  as well as \pagebreak
  \begin{align*}
   \|\wt\|_{C_I W_x^{k,p}}&\lesssim   \mu^{-1} \sum_{\substack{0\leq |\alpha|\leq k\\ \beta \leq \alpha \\ \gamma \leq \beta}} \sum_{\xi \in \Lambda}\| D^{\alpha-\beta}  \axi^2\|_{C_{I,x}^0} \|D^{\beta-\gamma }\psixi^2 \|_{CL_x^p} \|D^\gamma \phixi^2\|_{ L_x^p}\\ 
  &\lesssim  \frac{M^2}{|\Lambda|} \ell^{-(d+3)k-(3d+5)}  \left(1+\|\oRq\|_{C_{I_\ell} L_x^1}^{k+3}\right)  r_\perp^{\frac{d-1}{p}-(d-1)}r_\parallel^{\frac{1}{p}-1}\lambda^k \mu^{-1}
  \end{align*}
  for any $k>1$.
\end{proof}

\section{Inductive Estimates of the Velocity $v_{q+1}$}
\thref{Lemma Perturbation} provides the necessary bounds to verify the key inequalities \eqref{key bound a} through \eqref{key bound f} at the level $ q+1$, as detailed in Section \ref{Verifying the Key Bounds v} and Section \ref{Verifying the Key Bounds R}. Additionally, it is crucial for establishing the convergence of the sequences $(v_q)_{q \in \mathbb{N}_0}$ and $(\mathring{R}_q)_{q \in \mathbb{N}_0}$ to the limits $v$ and $0$, respectively, as discussed in Section \ref{Convergence of the Sequence v} and Section \ref{Convergence of the Sequence R}.

\subsection{Verifying the Key Bounds on the Level $q+1$} \label{Verifying the Key Bounds v}
Hereinafter, we are dedicated to verifying the key bounds \eqref{key bound a} -- \eqref{key bound d}, which pertain solely to the velocity. To this purpose we make use of the requirements
\begin{align*}
 \alpha b>2d+4N^\ast , \qquad  (\iota-1)\sigma \alpha >2\beta b^3 \qquad \text{and} \qquad a \geq \Big((8N^\ast+32) JLA\Big)^c, 
\end{align*} 
taken in Section \ref{Choice of Parameters}.

\subsubsection*{First Key Bound (\ref{key bound a})} 
First we appeal to \eqref{wpL2}, \eqref{wcL} and \eqref{wtL} to deduce 
\begin{align*}
\|w_{q+1}\|_{C_IL_x^2}&\lesssim M \left( \|\oRq\|_{C_\Il L_x^1}^{\frac{1}{2}}+ \bar{e}^{1/2} \delta_{q+1}^{1/2}\right)\\
&\hspace{.5cm}+ \left( M+\frac{M^2}{|\Lambda|}\right) \ell^{-\frac{5d+11}{2}} \left(1+\|\oRq\|_{C_\Il L_x^1}^{\frac{5}{2}}\right)\\
&\hspace{1cm}\cdot\left(r_\perp^{-1/2} \lambda^{-1/2}+r_\perp  r_\parallel^{-1} +r_\perp^{-\frac{d-1}{2}}r_\parallel^{-\frac{1}{2}}\mu^{-1}\right),
\end{align*}
which acquires thanks to \eqref{essential bound a}, \eqref{essential bound g}, \eqref{essential bound h}, \eqref{key bound e} and \eqref{key bound f} the form
\begin{align} \label{wL2a}
\|w_{q+1}\|_{L_\Omega^{2J} C_I L_x^2}&\lesssim  M \left( \|\oRq\|_{L_\Omega^J C_\Il L_x^1}^{\frac{1}{2}}+ \bar{e}^{1/2} \delta_{q+1}^{1/2}\right)\\
&\hspace{.5cm}+\left( M+\frac{M^2}{|\Lambda|}\right)  \left(1+\|\oRq\|_{L_\Omega^{5J}C_\Il L_x^1}^{\frac{5}{2}}\right) \lambda_{q+1}^{-2\alpha}\notag\\
&\lesssim M \bar{e}^{1/2} \delta_{q+1}^{1/2}+\left(M+\frac{M^2}{|\Lambda|}\right) (10JLA^q)^{5A^q}\lambda_{q+1}^{-2\alpha}\notag\\
& \lesssim M \bar{e}^{1/2} \delta_{q+1}^{1/2}+\left(M+\frac{M^2}{|\Lambda|}\right) \lambda_q^{5-\alpha b}\lambda_{q+1}^{-\beta } \notag.
\end{align}
Thus, by \eqref{property va} and \eqref{key bound c} we end up with \pagebreak
\begin{align} \label{vL2a}
\|v_{q+1}\|_{L_\Omega^{2J}C_IL_x^2}
&\leq  \sum_{r=0}^q\Big(\|v_\ell-v_r\|_{L_\Omega^{2J}C_I L_x^2} +\|w_{r+1}\|_{L_\Omega^{2J}C_IL_x^2}\Big) \\
&\lesssim  \sum_{r=0}^q\Big(\ell_r(20JLA^{r-1})^{10A^{r-1}} \lambda_r^{d+3} +\|w_{r+1}\|_{L_\Omega^{2J}C_IL_x^2}\Big)\notag \\
&\lesssim  \sum_{r=0}^q\Big(\lambda_{r+1}^{-2\alpha}\lambda_{r-1}^{10} \lambda_r^{d+3} +\|w_{r+1}\|_{L_\Omega^{2J}C_IL_x^2}\Big)\notag \\
&\lesssim  \sum_{r=0}^q\Big(\lambda_{r}^{d+4-\alpha b} \lambda_{r+1}^{-\beta }+\|w_{r+1}\|_{L_\Omega^{2J}C_IL_x^2}\Big)\notag \\
&\lesssim  \sum_{r=0}^q\left(M+\frac{M^2}{|\Lambda|}\right)  \bar{e}^{1/2} \delta_{r+1}^{1/2}\notag.
\end{align}
So stipulating $M_0$ to admit 
\begin{align}
M_0\gtrsim M+\frac{M^2}{|\Lambda|}, \label{M_0}
\end{align}
justifies the first key bound at level $q+1$. 
\subsubsection*{Second Key Bound (\ref{key bound b})}
In the same manner we obtain

\begin{align} \label{wL2b}
\|w_{q+1}\|_{L_\Omega^{j} C_I L_x^2}
&\lesssim  M \left( \|\oRq\|_{L_\Omega^{\frac{j}{2}} C_\Il L_x^1}^{\frac{1}{2}}+ \bar{e}^{1/2} \delta_{q+1}^{1/2}\right)\\
&\hspace{.5cm} +\left( M+\frac{M^2}{|\Lambda|}\right)  \left(1+\|\oRq\|_{L_\Omega^{\frac{5}{2} j}C_\Il L_x^1}^{\frac{5}{2}}\right) \lambda_{q+1}^{-2\alpha}\notag\\
& \lesssim M(jLA^q)^{A^q}+ M \bar{e}^{1/2} \delta_{q+1}^{1/2}+\left(M+\frac{M^2}{|\Lambda|}\right) (5jLA^q)^{5A^q}\lambda_{q+1}^{-2\alpha} \notag.
\end{align}
As a result 
\begin{align*}
\|v_{q+1}\|_{L_\Omega^jC_IL_x^2}&\lesssim  \sum_{r=0}^q\left(\ell_r(10jLA^{r-1})^{10A^{r-1}} \lambda_r^{d+3} +M(jLA^r)^{A^r}+M \bar{e}^{1/2} \delta_{r+1}^{1/2}\right)\\
&\hspace{.5cm}+\sum_{r=0}^q\left(M+\frac{M^2}{|\Lambda|}\right) (5jLA^r)^{5A^r}\lambda_{r+1}^{-2\alpha}\\
&\lesssim  \sum_{r=0}^q\left(\left(M+\frac{M^2}{|\Lambda|}\right)\left[ \lambda_{r+1}^{-\alpha}+\frac{1}{5^r}\right] (5jLA^{r})^{5A^{r}} +M \bar{e}^{1/2} \delta_{r+1}^{1/2}\right)\\
&\lesssim  \left(M+\frac{M^2}{|\Lambda|}\right) \sum_{r=0}^q\left[\delta_{r+1}^{1/2}+\frac{1}{5^r}\right](5jLA^{q})^{5A^{q}}  +M \bar{e}^{1/2}\sum_{r=0}^q \delta_{r+1}^{1/2}.
\end{align*}
In view of \eqref{delta series bound}, we might increase $M_0$ by multiplying the right hand side of \eqref{M_0} by some constant to confirm the second key bound.
\pagebreak
\subsubsection*{Third Key Bound (\ref{key bound c})}
Taking \eqref{wpC}, \eqref{wcC} and \eqref{wtC} into account, we infer 
\begin{align*}
\|w_{q+1}\|_{L_\Omega^jC_{I,x}^N}&\lesssim \left( M+\frac{M^2}{|\Lambda|}\right) \ell^{-(d+3)N-(4d+8)} \left(1+\|\oRq\|_{L_\Omega^{(N+4)j}C_\Il L_x^1}^{N+4} \right) \lambda^N \mu^N\\
&\hspace{1.2cm}\cdot\left( r_\perp^{-\frac{d-(2N+1)}{2}}r_\parallel^{-\frac{2N+1}{2}}+ r_\perp^{-\frac{d-(2N+3)}{2}} r_\parallel^{-\frac{2N+3}{2}}+r_\perp^{-d+N+1}r_\parallel^{-(N+1)}\lambda\mu^{-1}\right)
\end{align*}
for $N=1,2,3,4$. Bearing $\tau \leq 2\varsigma$ and \eqref{key bound f} in mind permits to furnish 
\begin{align}
\|w_{q+1}\|_{L_\Omega^jC_{I,x}^N}&\lesssim \left( M+\frac{M^2}{|\Lambda|}\right)\Big[2(N+4)jLA^q\Big]^{2(N+4)A^q}\lambda_{q+1}^{\big((4d+12)N+16d+32\big)\frac{\alpha}{\sigma-2\varepsilon}}\lambda_{q+1}^{\frac{N+1}{2}d+N+1}. \label{wC}
\end{align}
So, by possibly increasing $a,\, b$ and $A$ and \eqref{key bound c}, this results in
\begin{align*}
\|v_{q+1}\|_{L_\Omega^jC_{I,x}^N} &\lesssim \|v_q\|_{L_\Omega^jC_{\Il,x}^N}+ \|w_{q+1}\|_{L_\Omega^j C_{I,x}^N}\\
&\lesssim \left(\left( M+\frac{M^2}{|\Lambda|}\right)\lambda_{q+1}^{-\alpha}+A^{-2(N+4)}\right)\Big[2(N+4)jLA^{q}\Big]^{2(N+4)A^{q}} \lambda_{q+1}^{\frac{N+1}{2}d+N+2}\\
&\leq\Big[2(N+4)jLA^{q}\Big]^{2(N+4)A^{q}} \lambda_{q+1}^{\frac{N+1}{2}d+N+2}.
\end{align*}

\subsubsection*{Fourth Key Bound (\ref{key bound d})}   
Combining \eqref{wpwcW}, \eqref{wtW} and \eqref{essential bound b} with \eqref{key bound f} furnishes
\begin{align} \label{wWiota} 
\|w_{q+1}\|_{L_\Omega^{2J} C_I W_x^{1,\iota}} &\lesssim \left(M+\frac{M^2}{|\Lambda|}\right) \ell^{-\frac{7d+17}{2}} \left(1+\|\oRq\|_{L_\Omega^{7J} C_\Il L_x^1}^{7/2}\right)r_\perp^{\frac{d-1}{\iota}-\frac{d-1}{2}} r_\parallel^{\frac{1}{\iota}-\frac{1}{2}} \lambda \\
&\lesssim \left(M+\frac{M^2}{|\Lambda|}\right) (14JLA^q)^{7A^q}\lambda_{q+1}^{-2\alpha} \notag. 
\end{align}
\eqref{property va} and \eqref{key bound c} permit finally to conclude 
\begin{align}\label{vWiota}
\|v_{q+1}\|_{L_\Omega^{2J}C_IW_x^{1,\iota}}
&\leq  \sum_{r=0}^q\Big(\|v_\ell-v_r\|_{L_\Omega^{2J}C_I W_x^{1,\iota}} +\|w_{r+1}\|_{L_\Omega^{2J}C_I W_x^{1,\iota}}\Big)  \\
&\lesssim  \sum_{r=0}^q\Big(\ell_r(24JLA^{r-1})^{12A^{r-1}} \lambda_r^{\frac{3}{2}d+4} +\|w_{r+1}\|_{L_\Omega^{2J}C_IW_x^{1,\iota}}\Big)\notag \\
&\lesssim  \sum_{r=0}^q\bigg(\lambda_{r+1}^{-2\alpha}\lambda_{r-1}^{12} \lambda_r^{\frac{3}{2}d+4} +\left(M+\frac{M^2}{|\Lambda|}\right)\lambda_r^{7}\lambda_{r+1}^{-2\alpha}\bigg)\notag \\
&\lesssim  \sum_{r=0}^q \left(M+\frac{M^2}{|\Lambda|}\right)\Big( \lambda_{r}^{\frac{3}{2}d+5-\alpha b}+\lambda_{r}^{7-\alpha b} \Big) \lambda_{r+1}^{-\beta } \notag\\
&\leq  \sum_{r=1}^{q+1}  \delta_r^{1/2} \notag, 
\end{align}
where the absorption of the implicit constant in the last step relies again on the choice of $a$ and~$b$.

\subsection{Control of the Energy} \label{Control of the Energy}
As a next step, we will turn our attention to the verification of \eqref{approximation energy} at level $q+1$. For this purpose, let us fix some $t \in [o_{q+1}, \infty)$ that belongs to some interval $I\subseteq [o_{q+1}, \infty)$ of length $1$. Thus, we derive
\begin{align} \label{Control Energy Equation}
&\Big|e(t)\big(1-4\delta_{q+2}\big)-\mathscr{H}\{v_{q+1}+z_{q+1}\}(t) \Big|  \\
&\hspace{.2cm} \leq \underbrace{\E \Big| 3(2\pi)^d\eta_q(t)-\|\wpr(t)\|_{L_x^2}^2 \Big|}_{:=\text{I}}+\underbrace{2\E\|\wpr \cdot (\wc+\wt)\|_{ C_I L_x^1}}_{:=\text{II}}\notag\\
&\hspace{.5cm}+\underbrace{\E \|\wc+\wt\|_{C_I L_x^2}^2}_{:=\text{III}} +\underbrace{2\E\|\wpr\cdot (v_\ell+z_{q+1})\|_{C_I L_x^1}}_{:=\text{IV}}  \notag \\
&\hspace{.5cm}+\underbrace{2\E \|(\wc+\wt) \cdot (v_\ell+z_{q+1})\|_{C_I L_x^1}}_{:=\text{V}}+\underbrace{\E\|v_q+z_q-v_\ell-z_{q+1}\|_{C_I L_x^2}^2}_{\text{:=VI}}\notag\\
&\hspace{.5cm}+\underbrace{2\E\|(v_q+z_q-v_\ell-z_{q+1})(v_\ell+z_{q+1})\|_{ C_I L_x^1}}_{\text{:=VII}} \notag
\\
&\hspace{.5cm}\underbrace{
\begin{aligned}
&+\frac{2H}{t+H}\E\bigg| \int_0^{t} \int_{\mathbb{T}^d} \big|\mathcal{A}D(v_{q}+z_{q})(s,x)\colon D(v_{q}+z_{q})(s,x) \notag \\
&\hspace{4.5cm} - \mathcal{A}D(v_{q+1}+z_{q+1})(s,x) \colon D(v_{q+1}+z_{q+1})(s,x)\big| \,dx \, ds \bigg| \notag.
\end{aligned}
}_{:=\text{VIII}}
\end{align}
As in the previous section, the required conditions
\begin{align*}
\alpha b>2d+4N^\ast, \qquad  (\iota-1)\sigma \alpha>2\beta b^3 \qquad \text{and} \qquad a \geq \Big((8N^\ast+32) JLA\Big)^c 
\end{align*} 
turn out to be indispensable in the upcoming computations. We will also successively increase $a$, $b$ and $A$ again, as needed, to absorb all implicit constants that appear therein.
 
\noindent \begin{tikzpicture}[baseline=(char.base)]
\node(char)[draw,fill=white,
  shape=rounded rectangle,
  drop shadow={opacity=.5,shadow xshift=0pt},
  minimum width=.8cm]
  {\Large I};
\end{tikzpicture} 
The following statement is used to bound the first term and will also be crucial for finding a representation of the oscillation error, which becomes easier to handle (cf. \eqref{divRosc}). 
\begin{flalign} 
\text{\underline{1.Claim:}}\ \  & \wpr \otimes \wpr=\sum_{\xi \in \Lambda} \axi^2 \mathbb{P}_{\neq 0}\big(\Wxi \otimes \Wxi \big)+\rho \Id-\mathring{R}_\ell .& \label{wpxwp}
\end{flalign}
\begin{proof}
\thref{Lemma Building Blocks}, \eqref{building block c} and the Geometric \thref{Geometric Lemma} teach us 
\begin{align*}
\wpr \otimes  \wpr&= \sum_{\xi \in \Lambda} \axi^2 \big(\Wxi \otimes \Wxi\big)\\
&= \sum_{\xi \in \Lambda} \axi^2 \mathbb{P}_{\neq 0}\big(\Wxi \otimes \Wxi\big)+\sum_{\xi \in \Lambda} \axi^2  (2\pi)^{-d}\|\psi_{(\xi)}^2\phi_{(\xi)}^2\|_{L_x^1}\xi \otimes \xi\\
&= \sum_{\xi \in \Lambda} \axi^2 \mathbb{P}_{\neq 0}\big(\Wxi \otimes \Wxi\big)+\sum_{\xi \in \Lambda} \rho  \gamma^2_\xi\left(\Id-\frac{\mathring{R_\ell}}{\rho}\right)\xi \otimes \xi \\
&=\sum_{\xi \in \Lambda} \axi^2 \mathbb{P}_{\neq 0}\big(\Wxi \otimes \Wxi\big)+  \rho  \Id-\mathring{R_\ell}.
\end{align*}
\end{proof} 
\noindent Thus, the first term in \eqref{Control Energy Equation} amounts to  
\begin{align} \label{specific value of the key bound}
\text{I}&=\E\Big\|3\eta_q(t)-\tr\Big(\wpr \otimes \wpr\Big)(t)\Big\|_{L_x^1}  \\
&=\E\Big\|3 \eta_q(t)- \sum_{\xi \in \Lambda} \axi^2(t) \mathbb{P}_{\neq 0} \tr\Big(\Wxi \otimes \Wxi\Big)(t) - 3\rho(t)\Big\|_{L_x^1}\notag \\
&\leq 3(2\pi)^d |\eta_q(t)-\eta_\ell(t)|+\sum_{\xi \in \Lambda}\E\Big\|\axi^2 \mathbb{P}_{\neq 0} |\Wxi|^2\Big\|_{C_IL_x^1}+ 6(2\pi)^d \ell+6\,\E\|\mathring{R}_\ell\|_{C_I L_x^1}\notag.
\end{align}

We additionally need
\begin{adjustwidth}{20pt}{20pt}
\begin{flalign*} 
\text{\underline{2.Claim:}}\ \  &3(2\pi)^d |\eta_q(t)-\eta_\ell(t)| \leq \frac{1}{24} \underline{e} \delta_{q+2}.& 
\end{flalign*}
\begin{proof}
Each term of 
\begin{align} \label{etaq-etal}
\begin{split}
&3(2\pi)^d |\eta_q(t)-\eta_\ell(t)|\\ 
&\hspace{.2cm}\leq  |e(t)-(e\ast_t \varphi_\ell)(t)|+ \Big| \E\|(v_q+z_q)(t)\|^2_{L_x^2}-\big(\E\|v_q+z_q\|_{L_x^2}^2\ast_t \varphi_\ell \big)(t)\Big| \\
&\hspace{.5cm}+2H \bigg|\frac{1}{t+H}\E\left[\int_0^t \int_{\mathbb{T}^d} \mathcal{A}D(v_{q}+z_{q})(s,x)\colon D(v_{q}+z_{q})(s,x)\, dx \, ds \right] \\
&\hspace{.7cm}-\left\{\left(\frac{1}{\cdot+H}\E\left[\int_0^{\cdot } \int_{\mathbb{T}^d} \mathcal{A}D(v_{q}+z_{q})(s,x)\colon D(v_{q}+z_{q})(s,x)\, dx \, ds \right] \right) \ast_t \varphi_\ell \right\} (t) \bigg|
\end{split}
\end{align}
will be estimated separately. The first term is controlled as\footnote{At this point, it becomes obvious that it is insufficient to simply require the prescribed energy to be bounded; we additionally need the boundedness of its first derivative as well.} 
\begin{align*}
|e(t)-(e\ast \varphi_\ell)(t)|\leq \ell \widetilde{e}\leq \widetilde{e} \lambda_{q+1}^{-\alpha} \lambda_{q+1}^{-2\beta b}\leq \frac{1}{72 } \underline{e} \delta_{q+2},
\end{align*}
where we increase $a$ and $b$, so that $\widetilde{e}\lambda_{q+1}^{-\alpha}16 \lambda_2^{-2\beta}\leq \frac{1}{72 } \underline{e}$ holds.\\
Estimating the second term of \eqref{etaq-etal} follows by standard mollification estimates, \eqref{key bound a}, \eqref{delta series bound}, \eqref{OPa}, \thref{Proposition SHE} and \eqref{key bound c} 
\begin{align*}
&\Big|\E\|(v_q+z_q)(t)\|_{L_x^2}^2-\Big(\E\|v_q+z_q\|_{L_x^2}^2\ast_t \varphi_\ell\Big)(t) \Big|\\
&\lesssim \ell^{\sigma/2-\varepsilon} \|v_q+z_q\|_{L_\Omega^2C_\Il L_x^2} \|v_q+z_q\|_{L_\Omega^2 C_\Il^{0,\sigma/2-\varepsilon}L_x^2} \\
& \lesssim \ell^{\sigma/2-\varepsilon} \Big(\|v_q\|_{L_\Omega^2 C_\Il L_x^2}+\|z_q\|_{L_\Omega^2 C_\Il L_x^2} \Big)\Big(\|v_q\|_{L_\Omega^2 C_{\Il,x}^1}+\|z_q\|_{L_\Omega^2 C_{\Il}^{0,\sigma/2-\varepsilon}L_x^2}\Big) \\
&\lesssim \lambda_{q+1}^{-\alpha} \lambda_{q}^{-d-5} (M_0\bar{e}^{1/2}+K_G)\Big[(20LA^{q-1})^{10A^{q-1}} \lambda_q^{d+3}+K_G\Big]\\
&\lesssim K_G^2 M_0\bar{e}^{1/2} \lambda_{q-1}^{10} \lambda_q^{d+3} \lambda_{q+1}^{-\alpha} \lambda_{q}^{-d-5}\\
&\lesssim K_G^2 M_0\bar{e}^{1/2} \lambda_q^{-1} \lambda_{q+1}^{-2\beta b}\\
&\leq \frac{1}{72} \underline{e} \delta_{q+2}.
\end{align*}
Involving \thref{Lemma 2.4}, \eqref{key bound d}, \eqref{delta series bound}, \eqref{OPb}, \eqref{OPa} and \thref{Proposition SHE} permits to control the last term as  \pagebreak
\begin{align} \label{Reason of H}
&\bigg|\frac{1}{t+H}\E\left[\int_0^t \int_{\mathbb{T}^d} \mathcal{A}D(v_{q}+z_{q})(s,x)\colon D(v_{q}+z_{q})(s,x)\, dx \, ds \right]\\
&\hspace{.5cm}-\bigg\{\left(\frac{1}{\cdot+H}\E\left[\int_0^\cdot \int_{\mathbb{T}^d} \mathcal{A}D(v_{q}+z_{q})(s,x)\colon D(v_{q}+z_{q})(s,x)\, dx \, ds\right]\right) \ast_t \varphi_\ell \bigg\} (t) \bigg|\notag\\
&\lesssim \frac{1}{t+H}  \E\left[\int_0^{\ell }\int_{t-u}^t \int_{\mathbb{T}^d} \max_{k=1,\iota} \|D(v_{q}+z_{q})(s,x)\|_F^k\, dx \, ds  \, \varphi_\ell(u) \, du\right]\notag\\
&\hspace{.1cm}+ \int_0^{\ell }\Big| \frac{u}{(t+H)(t-u+H)}\Big|\E\left|\int_0^{t-u} \int_{\mathbb{T}^d} \max_{k=1,\iota}\|D(v_{q}+z_{q})(s,x)\|^k_F\, dx \, ds\right|\, \varphi_\ell(u)\,du  \notag\\
&\lesssim \frac{1}{t+H}  \int_0^{\ell }u \max_{k=1,\iota}\|v_{q}+z_{q}\|_{L_\Omega^\iota C_\Il W_x^{1,\iota}}^{k } \,\varphi_\ell(u) \, du \notag\\
&\hspace{.1cm}+ \int_0^{\ell }\Big| \frac{u}{(t+H)(t-u+H)}\Big|\E\left|\int_0^{t-u} \max_{k=1,\iota}\|(v_{q}+z_{q})(s)\|_{W_x^{1,\iota}}^{k } \, ds\right|\, \varphi_\ell(u)\,du  \notag \\
 &\lesssim \ell\left\{1+ \int_0^{\ell }\Big| \frac{ |t-u| }{(t+H)(t-u+H)}\Big|\, \varphi_\ell(u)\,du \right\}\notag\\
 &\hspace{4.5cm}\cdot \sup_{n\geq o_q}\max_{k=1,\iota} \Big( \|v_{q}\|_{L_\Omega^\iota C_{[n,n+1]} W_x^{1,\iota}}+\|z_{q}\|_{L_\Omega^\iota C_{[n,n+1]} W_x^{1,\iota}} \Big)^k \notag \\
&\lesssim \lambda_{q+1}^{-\alpha}\lambda_{q+1}^{-2\beta b} (1+ K_G^\iota) \notag \\
&\leq \frac{1}{144 H} \underline{e} \delta_{q+2},\notag
\end{align}
which finally establishes the assertion.
\end{proof}

\end{adjustwidth}
\vspace{0.5cm}

\noindent 
\begin{adjustwidth}{20pt}{20pt}
\begin{flalign*} 
\text{\underline{3.Claim:}}\ \  &\sum_{\xi \in \Lambda}\E\Big\|\axi^2 \mathbb{P}_{\neq 0} |\Wxi|^2\Big\|_{C_IL_x^1} \leq \frac{1}{24} \underline{e} \delta_{q+2}.& 
\end{flalign*}
\begin{proof}
We intend to apply Green's formula $N^\ast \in 2\mathbb{N}$ times with 
\begin{align} \label{Nast}
 N^\ast\geq  \frac{d(\sigma-2\varepsilon)-(30d^2+84d+48) \alpha}{(104d+224)\alpha} .
\end{align}
Together with Leibniz's formula, \eqref{OPg}, \eqref{amplitude function b}, \eqref{OPd}, \eqref{key bound f} and \eqref{fundamental bounds c} this actually entails   
\begin{align*}
\sum_{\xi \in \Lambda}&\E\Big\|\axi^2 \mathbb{P}_{\neq 0} |\Wxi|^2\Big\|_{C_IL_x^1}\\
 &\leq \sum_{\xi \in \Lambda} \E\|\Delta^{N^\ast} \axi^2\|_{C_{I,x}^0} \|(-\Delta)^{-N^\ast} \mathbb{P}_{\neq 0}|\Wxi|^2\|_{C_I L_x^1}\\
& \lesssim  \sum_{\xi \in \Lambda} \sum_{k=0}^{2N^\ast} \|\axi\|_{L_\Omega^2 C_{I,x}^{2N^\ast-k}} \|\axi\|_{L_\Omega^2 C^k_{I,x}}\|(-\Delta)^{-N^\ast} \mathbb{P}_{\geq \frac{r_\perp \lambda_{q+1}}{2}}|\Wxi|^2\|_{C_I L_x^2}\\
&\lesssim \frac{M^2}{|\Lambda|^2}  \ell^{-(2d+6)N^\ast- (2d+3)}\left(1+\|\oRq\|^{2N^\ast+2}_{L_\Omega^{4N^\ast+3}C_\Il L_x^1}\right)(r_\perp \lambda_{q+1})^{-2N^\ast} \sum_{\xi \in \Lambda}\|\Wxi\|^2_{C_I L_x^4}\\
&\hspace{.2cm} + \frac{M^2}{|\Lambda|^2}  \ell^{-(2d+6)N^\ast-(3d+5)}\left(1+\|\oRq\|^{2N^\ast+3}_{L_\Omega^{4N^\ast+1}C_\Il L_x^1}\right) (r_\perp \lambda_{q+1})^{-2N^\ast} \sum_{\xi \in \Lambda}\|\Wxi\|^2_{C_I L_x^4}\\
& \lesssim  \frac{M^2}{|\Lambda|}  \lambda_{q+1}^{\big((8d+24)N^\ast+12d+20\big)\frac{\alpha}{\sigma-2\varepsilon}} \big( (8N^\ast+6)LA^q\big)^{(4N^\ast+6)A^q}r_\perp^{-\frac{d-1}{2}-2N^\ast}r_\parallel^{-\frac{1}{2}}\lambda_{q+1}^{-2N^\ast}\\
& \lesssim  \frac{M^2}{|\Lambda|}   \lambda_q^{4N^\ast+6-\alpha b} \lambda_{q+1}^{-2\beta b}\\
&\leq \frac{1}{24}\underline{e} \delta_{q+2}.
\end{align*}

\end{proof}
\end{adjustwidth}
\vspace{0.5cm}
Together with \eqref{key bound e} the above considerations result in
\begin{align*}
\text{I}\leq \frac{1}{6}\underline{e} \delta_{q+2}.
\end{align*}

\noindent As a next step we aim to control the second and fifth term.\\
\noindent 
\begin{tikzpicture}[baseline=(char.base)]
\node(char)[draw,fill=white,
  shape=rounded rectangle,
  drop shadow={opacity=.5,shadow xshift=0pt},
  minimum width=.8cm]
  {\Large II \& V};
\end{tikzpicture} 
For this reason we appeal to \eqref{wpL2}, combine \thref{Proposition SHE} with \eqref{OPa} and invoke \eqref{property vb}, \eqref{key bound a}, \eqref{delta series bound}, \eqref{wcL}, \eqref{wtL}, \eqref{key bound e}, \eqref{key bound f}, \eqref{essential bound a}, \eqref{essential bound g}, \eqref{essential bound h} and \eqref{M_0} to obtain 

\begin{align*}
\text{II \& V} &\lesssim \|\wpr\|_{L_\Omega^2C_I L_x^2} \left(\|\wc\|_{L_\Omega^2C_I L_x^2}+\|\wt\|_{L_\Omega^2C_I L_x^2}\right)\\
&\hspace{.2cm}+ \left(\|z_{q+1}\|_{L_\Omega^2C_I L_x^2}+\|v_\ell\|_{L_\Omega^2C_I L_x^2} \right) \left(\|\wc\|_{L_\Omega^2C_I L_x^2}+\|\wt\|_{L_\Omega^2C_I L_x^2}\right)\\
&\lesssim \left\{ M\|\oRq\|_{L_\Omega^1C_\Il L_x^1}^{1/2}+M\bar{e}^{1/2} \delta_{q+1}^{1/2}+M \ell^{-\frac{5d+11}{2}} (1+\|\oRq\|_{L_\Omega^5C_\Il L_x^1}^{5/2})r_\perp^{-1/2}\lambda^{-1/2} \right\}\\
&\hspace{.5cm}\cdot \left(M+\frac{M^2}{|\Lambda|}\right)\ell^{-\frac{5d+11}{2}} \big(1+\|\oRq\|^{5/2}_{L_\Omega^5C_\Il L_x^1}\big)\big(r_\perp r_\parallel^{-1}+r_\perp^{-\frac{d-1}{2}}r_\parallel^{-\frac{1}{2}}\mu^{-1}\big)\\
&\hspace{.2cm}+ \left\{K_G+M_0\bar{e}^{1/2} \right\}\\
&\hspace{.5cm}\cdot\left(M+\frac{M^2}{|\Lambda|}\right)\ell^{-\frac{5d+11}{2}} \big(1+\|\oRq\|^{5/2}_{L_\Omega^5C_\Il L_x^1}\big)\big(r_\perp r_\parallel^{-1}+r_\perp^{-\frac{d-1}{2}}r_\parallel^{-\frac{1}{2}}\mu^{-1}\big)\\
&\lesssim  M_0^2K_G \bar{e}^{1/2}  (10LA^q)^{10A^q}\lambda_{q+1}^{-2\alpha}\\
&\lesssim  M_0^2K_G \bar{e}^{1/2}  \lambda_q^{10-\alpha b}\lambda_{q+1}^{-2\beta b}\\
&\leq \frac{1}{6} \underline{e}\delta_{q+2}.
\end{align*}

\noindent 
\begin{tikzpicture}[baseline=(char.base)]
\node(char)[draw,fill=white,
  shape=rounded rectangle,
  drop shadow={opacity=.5,shadow xshift=0pt},
  minimum width=.8cm]
  {\Large III};
\end{tikzpicture} 
We proceed with the estimation of III, which is done in a very similar fashion. As a result of \eqref{wcL}, \eqref{wtL}, \eqref{key bound f}, \eqref{essential bound g} and \eqref{essential bound h} we get 
\begin{align*}
\text{III}& \lesssim \left(\|\wc\|_{L_\Omega^2 C_I L_x^2}+\|\wt\|_{L_\Omega^2 C_I L_x^2} \right)^2\\
&\lesssim \left(M+\frac{M^2}{|\Lambda|}\right)^2\ell^{-(5d+11)} \big(1+\|\oRq\|^{5}_{L_\Omega^5C_\Il L_x^1}\big)\big(r_\perp r_\parallel^{-1}+r_\perp^{-\frac{d-1}{2}}r_\parallel^{-\frac{1}{2}}\mu^{-1}\big)^2\\
&\lesssim \left(M+\frac{M^2}{|\Lambda|}\right)^2 (10LA^q)^{10A^q} \lambda_{q+1}^{-2 \alpha}\\
&\lesssim \left(M+\frac{M^2}{|\Lambda|}\right)^2 \lambda_q^{10-\alpha b}\lambda_{q+1}^{-2\beta b}\\
&\leq \frac{1}{6} \underline{e} \delta_{q+2}.
\end{align*}

\noindent 
\begin{tikzpicture}[baseline=(char.base)]
\node(char)[draw,fill=white,
  shape=rounded rectangle,
  drop shadow={opacity=.5,shadow xshift=0pt},
  minimum width=.8cm]
  {\Large IV};
\end{tikzpicture}
Bearing \eqref{property vc}, \eqref{property za}, \eqref{wpL}, \eqref{key bound c}, \thref{Proposition SHE}, \eqref{key bound e} in mind and remembering that $p$ belongs to 
\begin{align*}
\left(1, \frac{2(d-1)\varsigma+\tau}{2d\varsigma+\tau-1+(20d+46)\frac{\alpha}{\sigma-2\varepsilon}} \wedge \frac{4(d-1)\varsigma +2\tau}{2(d-1)\varsigma+\tau+(d-2)\frac{\ln(f(q+1))}{\ln(\lambda_{q+1})}+(20d+48)  \frac{\alpha}{\sigma-2\varepsilon}}\right]
\end{align*}
 we appeal to \eqref{essential bound d} to control the fourth term as 
\begin{align*}
\text{IV}&\lesssim\|\wpr\|_{L_\Omega^2 C_I L_x^p} \left(\|v_q\|_{L_\Omega^2 C_{\Il,x}^1}+ f^{\frac{d-2}{2}}(q+1)\|z\|_{L_\Omega^2 C_I H_x^{1+\sigma}}\right)\\
&\lesssim M \ell^{-\frac{d+1}{2}} \big(1+\|\oRq\|_{L_\Omega^1 C_\Il L_x^1}^{1/2}\big)r_\perp^{\frac{d-1}{p}-\frac{d-1}{2}} r_\parallel^{\frac{1}{p}-\frac{1}{2}}\\
&\hspace{.5cm} \cdot \Big( (20LA^{q-1})^{10A^{q-1}}\ \lambda_q^{d+3}+f^{\frac{d-2}{2}}(q+1)K_G \Big)\\
&\lesssim M  \bar{e}^{1/2}  K_G \lambda_{q+1}^{-2 \alpha}\lambda_{q-1}^{10}\lambda_{q}^{d+3}\\
&\lesssim M  \bar{e}^{1/2} K_G\lambda_{q}^{d+4-\alpha b}\lambda_{q+1}^{-2\beta b}\\
&\leq \frac{1}{6} \underline{e} \delta_{q+2}.
\end{align*}

\noindent 
\begin{tikzpicture}[baseline=(char.base)]
\node(char)[draw,fill=white,
  shape=rounded rectangle,
  drop shadow={opacity=.5,shadow xshift=0pt},
  minimum width=.8cm]
  {\Large VI \& VII};
\end{tikzpicture} 
Combining \eqref{property va}, \eqref{property zb}, \eqref{OPa}, \eqref{key bound c}, \thref{Proposition SHE}, \eqref{key bound a} and \eqref{delta series bound} with \eqref{frequency z} results in 
\begin{align*}
\text{VI \& VII} &\lesssim \Big(\|v_q-v_\ell\|_{L_\Omega^2 C_I L_x^2}+\|z_q-z_{q+1}\|_{L_\Omega^2 C_I L_x^2} \Big)^2\\
&\hspace{.5cm}+\Big(\|v_q-v_\ell\|_{L_\Omega^2 C_I L_x^2}+\|z_q-z_{q+1}\|_{L_\Omega^2 C_I L_x^2} \Big) \Big(\|v_q\|_{L_\Omega^2 C_\Il L_x^2}+\|z_{q+1}\|_{L_\Omega^2 C_I L_x^2} \Big) \\
&\lesssim \Big(\ell(20LA^{q-1})^{10A^{q-1}}\lambda_q^{d+3}+f^{-(1+\sigma)}(q)K_G\Big)^2\\ 
&\hspace{.5cm}+\Big(\ell(20LA^{q-1})^{10A^{q-1}}\lambda_q^{d+3}+f^{-(1+\sigma)}(q)K_G\Big)\Big(M_0 \bar{e}^{1/2}+K_G\Big)\\
&\lesssim M_0\bar{e}^{1/2}K_G^2 \max_{k=1,2}\Big(\lambda_{q+1}^{-2\alpha}\lambda_{q-1}^{10}\lambda_q^{d+3}+f^{-(1+\sigma)}(q)\Big)^k\\
&\lesssim M_0\bar{e}^{1/2}K_G^2 \Big(\lambda_q^{d+4-\alpha b}\lambda_{q+1}^{-2\beta b}+\lambda_q^{-2\beta b^3\frac{1+\sigma}{(\iota-1)\sigma}}\Big)\\
&\leq \frac{1}{6}  \underline{e} \delta_{q+2}.
\end{align*}

\noindent 
\begin{tikzpicture}[baseline=(char.base)]
\node(char)[draw,fill=white,
  shape=rounded rectangle,
  drop shadow={opacity=.5,shadow xshift=0pt},
  minimum width=.8cm]
  {\Large VIII};
\end{tikzpicture} 
To estimate VIII we employ \thref{Lemma 2.4} and Hölder's inequality with $\iota, \frac{\iota}{\iota-1}$ twice to achieve  
\begin{align*}
\text{VIII}
&\lesssim \frac{1}{t+H} \E\bigg|\int_0^{t} \|D(v_{q+1}+z_{q+1}-v_q-z_q)(s)\|_{L_x^{\iota}} \Big(1+ \|D(v_{q+1}+z_{q+1})(s)\|^{\iota-1}_{L_x^{\iota}}\Big)  \, ds\bigg|\\
&\hspace{.5cm}+ \frac{1}{t+H} \E\bigg|\int_0^{t} \|D(v_{q+1}+z_{q+1}-v_q-z_q)(s)\|_{L_x^{\iota}} \|D(v_q+z_q)(s)\|^{\iota-1}_{L_x^{\iota}} \, ds\bigg|\\
&\lesssim  \frac{|t|}{t+H}  \sup_{s \in [o_{q+1},\infty)}  \|v_{q+1}+z_{q+1}-v_q-z_q\|_{L_\Omega^{\iota}C_{[s,s+1]}W_x^{1,\iota}} \\
&\hspace{.5cm}\cdot \frac{|t|}{t+H}    \sup_{s \in [o_{q+1},\infty)} \Big(1+ \|v_{q+1}+z_{q+1}\|^{\iota-1}_{L_\Omega^\iota C_{[s,s+1]}W_x^{1,\iota}}+ \|v_q+z_q\|^{\iota-1}_{L_\Omega^\iota C_{[s,s+1]}W_x^{1,\iota}}\Big)  . 
\end{align*}
Since $t \mapsto \frac{|t|}{t+H}$ is for every $t\in [o_{q+1},\infty)$ bounded from above by  $1$, the additional consideration of \eqref{property va}, \eqref{key bound c}, \eqref{wWiota}, \eqref{property zb}, \eqref{key bound d}, \eqref{delta series bound}, \eqref{OPb}, \eqref{OPa}, \thref{Proposition SHE} and \eqref{frequency z} yield \pagebreak
\begin{align*}
\text{VIII}&\lesssim \sup_{s \in [o_{q+1},\infty)} \bigg\{  \Big(\|v_q-v_\ell\|_{L_\Omega^{\iota}C_{[s,s+1]}W_x^{1,\iota}}+\| w_{q+1}\|_{L_\Omega^{\iota}C_{[s,s+1]}W_x^{1,\iota}} \Big)K_G^{\iota-1} \bigg\}\\
&\hspace{.5cm}+\sup_{s \in [o_{q+1},\infty)}   \|z_{q+1}-z_q\|_{L_\Omega^{\iota}C_{[s,s+1]}W_x^{1,\iota}} K_G^{\iota-1} \\
&\lesssim  K_G^{\iota-1}  \Big[ \ell (12\iota L A^{q-1})^{12A^{q-1}} \lambda_q^{\frac{3}{2}d+4}+ \left(M+\frac{M^2}{|\Lambda|}\right) (14J LA^q)^{7A^q} \lambda_{q+1}^{-2\alpha}+f^{-\sigma}(q)K_G \Big]\\
&\lesssim  K_G^{\iota-1}  \Big[  \lambda_{q-1}^{12} \lambda_q^{\frac{3}{2}d+4-\alpha b}\lambda_{q+1}^{-2\beta b }+ \left(M+\frac{M^2}{|\Lambda|}\right) \lambda_q^{7-\alpha b} \lambda_{q+1}^{-2\beta b }+\lambda_{q}^{\frac{-2\beta b^3}{\iota-1}} K_G \Big] \\
&\leq \frac{1}{6} \underline{e} \delta_{q+2}.
\end{align*}
\medskip~\\
Altogether, \eqref{approximation energy} remains valid on the level $q+1$.
 
\subsection{Convergence of the Sequence} \label{Convergence of the Sequence v}
Moreover, we have all the necessary components to establish the convergence of $\left(v_q\right)_{q \geq 0}$ to some limit function $v$ in
\begin{align*}
\left( L^{2J}\big(\Omega;C\big([0,\infty);L^2\big( \mathbb{T}^d;\mathbb{R}^d\big)\cap W^{1,\iota}\big( \mathbb{T}^d;\mathbb{R}^d\big)\big)\big),\sup_{|I|=1}\|\cdot\|_{L_\Omega^{2J}C_I L_x^2}+\sup_{|I|=1}\|\cdot\|_{L_\Omega^{2J}C_I W_x^{1,\iota}} \right).
\end{align*}
To be more precise we follow the approach in \eqref{vL2a} and \eqref{vWiota} and use \eqref{delta series bound} to achieve
\begin{align*}
\sum_{q \geq 0}\Big(\|v_{q+1}-v_q\|_{L_\Omega^{2J}C_{I}L_x^2} +\|v_{q+1}-v_q\|_{L_\Omega^{2J}C_{I}W_x^{1,\iota}} \Big)&\lesssim M_0 \bar{e}^{1/2}.
\end{align*}
That means the sequence $\left(\sum\limits_{q=0}^n \Big( \|v_{q+1}-v_q\|_{L_\Omega^{2J}C_{I}L_x^2} +\|v_{q+1}-v_q\|_{L_\Omega^{2J}C_{I}W_x^{1,\iota}} \Big)\right)_{n \in \mathbb{N}_0}$ \hspace{-0.4cm} converges along a subsequence in $\mathbb{R}$ and is therefore particularly Cauchy. Strictly speaking, for every $\varepsilon>0$, we find some $N>0$, such that 
\begin{align*}
&\|v_{n}-v_k\|_{L_\Omega^{2J}C_{I}  L_x^2} +\|v_{n}-v_k\|_{L_\Omega^{2J}C_{I}  W_x^{1,\iota}}\\
 &\hspace{.2cm}\leq \sum_{q=k}^{n-1}\Big( \|v_{q+1}-v_q\|_{L_\Omega^{2J}C_{I}  L_x^2}+\|v_{q+1}-v_q\|_{L_\Omega^{2J}C_{I}W_x^{1,\iota}} \Big)\\
 &\hspace{.2cm}\leq \varepsilon
\end{align*}
holds for any $n>k\geq N$, which shows that even $(v_q)_{q \geq 0}$ is Cauchy in 
\begin{align*}
\left(L^{2J}\big(\Omega;C\big([0,\infty);L^2\big( \mathbb{T}^d;\mathbb{R}^d\big)\cap W^{1,\iota}\big( \mathbb{T}^d;\mathbb{R}^d\big)\big)\big),\sup_{|I|=1}\|\cdot\|_{L_\Omega^{2J}C_I L_x^2}+\sup_{|I|=1}\|\cdot\|_{L_\Omega^{2J}C_I W_x^{1,\iota}} \right).
\end{align*}

\pagebreak
\boldmath
\section{Decomposition of the Reynolds Stress $\mathring{R}_{q+1}$}\label{Decomposition of the Reynolds Stress}
\unboldmath
To verify the key bounds \eqref{key bound e} and \eqref{key bound f}, we need to determine the precise form of the Reynolds stress at level $q+1$. Upon plugging the velocity $v_{q+1}$ and $z_{q+1}$ into \eqref{Iterative Equation}, we obtain the divergence of the new Reynolds stress and the gradient of the new pressure term as  
\begin{align*}
&\divs(\mathring{R}_{q+1})-\nabla p_{q+1}\\
&:=\underbrace{\left\{\partial_tv_q+\divs\big((v_q+z_q)\otimes (v_q+z_q)\big)-\divs \big(\mathcal{A}\left(Dv_q+Dz_q \right)\big)+\Delta z_q-z_q \right\}\ast_t \varphi_\ell \ast_x \phi_\ell}_{=: \divs(\mathring{R}_\ell)-\nabla p_\ell}\\
&\hspace{.25cm}-\divs\big((v_q+z_q)\otimes (v_q+z_q)\big)\ast_t \varphi_\ell \ast_x \phi_\ell+\divs \big(\mathcal{A}\left(Dv_q+Dz_q \right)\big)\ast_t \varphi_\ell \ast_x \phi_\ell-\Delta z_\ell +z_\ell\\
&\hspace{.25cm}+\partial_t w_{q+1}+\divs\big((v_{q+1}+z_{q+1})\otimes (v_{q+1}+z_{q+1})\big)-\divs \big(\mathcal{A}\left(Dv_{q+1}+Dz_{q+1} \right)\big)+\Delta z_{q+1}\\
&\hspace{.25cm}-z_{q+1} .
\end{align*}
For better handling of the new stress term, we will split it into seven parts: the linear error, the nonlinear errors $1$ and $2$, the oscillation error, the corrector error and the commutator errors $1$ and $2$. Each of them remains traceless, which is achieved using the formulas $\divs \big(A \big)= \divs\big(\mathring{A} \big)+\frac{1}{d} \nabla \tr\big(A\big)$ and $\tr(a \otimes b)=a \cdot b $ for arbitrary quadratic matrices $A$ and vectors $a$ and $b$. Thus, the above expression amounts to
\begin{align*}
&\divs(\mathring{R}_{q+1})-\nabla p_{q+1}\\
&=\underbrace{ \partial_t(\wpr+\wc)+\divs\big((v_\ell+z_\ell) \ootimes w_{q+1}+w_{q+1}\ootimes (v_\ell+z_\ell)\big)+\Delta (z_{q+1}-z_\ell)-(z_{q+1}-z_\ell)}_{=:\divs(\oRlin)}\\
&\hspace{.5cm}+\nabla \underbrace{\frac{2}{d} \big((v_\ell+z_\ell) \cdot w_{q+1}\big)}_{=: p_\text{lin}}+\divs\big(\underbrace{\mathcal{A}(Dv_q+Dz_q)\ast_t \varphi_\ell \ast_x \phi_\ell-\mathcal{A}(Dv_\ell+Dz_\ell)}_{=:\oRnonlin}\big)\\
&\hspace{.5cm}+\divs\big(\underbrace{\mathcal{A}(Dv_\ell+Dz_{\ell}) - \mathcal{A}(Dv_{q+1}+Dz_{q+1}) }_{=:\oRnonlinn}\big)\\
&\hspace{.5cm}+\underbrace{\divs(\wpr \otimes \wpr+\mathring{R}_\ell)+\partial_t \wt-\nabla p_\ell}_{=:\divs(\oRosc)+\nabla p_\text{osc} }\\
&\hspace{.5cm}+\divs(\underbrace{\wpr \ootimes (\wc +\wt)+ (\wc + \wt)\ootimes w_{q+1} }_{=: \oRcor})\\
&\hspace{.5cm}+\nabla \underbrace{\frac{1}{d}\Big((\wpr+w_{q+1})\cdot(\wc+\wt)\Big)}_{=:p_\text{cor}}\\
&\hspace{.5cm}+\divs\big(\underbrace{(v_\ell+z_\ell) \ootimes (v_\ell+z_\ell)- \big[(v_q+z_q) \ootimes (v_q+z_q)\big]\ast_t \varphi_\ell \ast_x \phi_\ell}_{=:\oRcom}\big)\\
 &\hspace{.5cm}+ \nabla \underbrace{\frac{1}{d} \big(|v_\ell+z_\ell|^2-|v_q+z_q|^2 \ast_t \varphi_\ell \ast_x \phi_\ell\big)}_{=:p_\text{com1}}\\
&\hspace{.5cm}+\divs\big(\underbrace{v_{q+1}\ootimes (z_{q+1}-z_\ell) + (z_{q+1}-z_\ell)\ootimes v_{q+1}+ z_{q+1}\ootimes z_{q+1}-z_\ell \ootimes z_\ell}_{=:\oRcomm}\big)\\
&\hspace{.5cm}+ \nabla \underbrace{\frac{1}{d} \big(2 v_{q+1}\cdot(z_{q+1}-z_\ell)+|z_{q+1}|^2 -|z_\ell|^2\big)}_{=:p_\text{com2}}.
\end{align*}

\noindent
In view of \thref{Lemma 2.1} and \thref{Lemma 2.2}, this gives rise to the definition of
\begin{align*}
\oRlin&:=\invdiv \partial_t(\wpr+\wc)+ (v_\ell+z_\ell) \ootimes w_{q+1}+w_{q+1}\ootimes (v_\ell+z_\ell)+\invdiv\Delta (z_{q+1}-z_\ell)\\
&\hspace{.5cm}-\invdiv(z_{q+1}-z_\ell), \\
\oRnonlin&:=\mathcal{A}(Dv_q+Dz_q)\ast_t \varphi_\ell \ast_x \phi_\ell-\mathcal{A}(Dv_\ell+Dz_\ell),\\
\oRnonlinn&:= \mathcal{A}(Dv_\ell+Dz_{\ell}) - \mathcal{A}(Dv_{q+1}+Dz_{q+1}), \\
\oRosc&:= \sum_{\xi \in \Lambda} \mathcal{B}\left(  \mathbb{P}_{\geq \frac{r_\perp \lambda_{q+1}}{2} }\big(\Wxi \otimes \Wxi\big),\nabla \axi^2 \right)-\mu^{-1} \sum_{\xi \in \Lambda} \invdiv\mathbb{P}_{\neq 0} \Big[ \big(\partial_t \axi^2\big)\psi_{(\xi)}^2\phi_{(\xi)}^2 \xi \Big],\\
\oRcor&:=\wpr \ootimes (\wc +\wt)+ (\wc + \wt)\ootimes w_{q+1}, \\
\oRcom&:=(v_\ell+z_\ell) \ootimes (v_\ell+z_\ell)- \big[(v_q+z_q) \ootimes (v_q+z_q)\big]\ast_t \varphi_\ell \ast_x \phi_\ell,\\
\oRcomm&:=v_{q+1}\ootimes (z_{q+1}-z_\ell) + (z_{q+1}-z_\ell)\ootimes v_{q+1}+ z_{q+1}\ootimes z_{q+1}-z_\ell \ootimes z_\ell
\end{align*}
and the corresponding pressures as 
\begin{align*}
p_\text{lin}&:=\frac{2}{d} \Big((v_\ell+z_\ell) \cdot   w_{q+1}\Big),\\
p_\text{osc}&:=\mu^{-1}\sum_{\xi \in \Lambda}\Delta^{-1} \divs\Big( \mathbb{P}_{\neq 0} \big[\partial_t (\axi^2\psi_{(\xi)}^2\phi_{(\xi)}^2 \xi ) \big]\Big)+\rho-p_\ell,\\
p_\text{cor}&:=\frac{1}{d}\Big((\wpr+w_{q+1})\cdot(\wc+\wt)\Big),\\
p_\text{com1}&:= \frac{1}{d} \Big(|v_\ell+z_\ell|^2-|v_q+z_q|^2 \ast_t \varphi_\ell \ast_x \phi_\ell\Big),\\
p_\text{com2}&:=\frac{1}{d} \Big(2 v_{q+1}\cdot(z_{q+1}-z_\ell)+|z_{q+1}|^2 -|z_\ell|^2\Big).
\end{align*}
The representation of the oscillation error and oscillation pressure rely on the subsequent computation:\\
Since $\axi^2 \mathbb{P}_{\neq 0} (\Wxi \otimes \Wxi)$ is a smooth function with periodic boundary conditions, \linebreak $\divs\big(\axi^2 \mathbb{P}_{\neq 0} (\Wxi \otimes \Wxi)\big)$ has zero mean, which allows us, together with \eqref{wpxwp}, to rewrite the oscillation error as 
\begin{align}  \label{divRosc}
\begin{split}
\divs(\mathring{R}_{\text{osc}})+\nabla p_{\text{osc}} &=\sum_{\xi \in \Lambda} \mathbb{P}_{\neq 0}\Big[  \mathbb{P}_{\neq 0}(\Wxi \otimes \Wxi)\nabla\axi^2\Big]+\partial_t \wt\\
&\hspace{.5cm}+\sum_{\xi \in \Lambda} \mathbb{P}_{\neq 0}\bigg[ \axi^2 \divs\Big(\mathbb{P}_{\neq 0}(\Wxi \otimes \Wxi)\Big) \bigg]+\nabla \big(\rho-p_\ell\big).
\end{split}
\end{align}
The fact that $\{\xi,A^1_\xi,A^2_\xi,\ldots,  A^{d-1}_\xi\}$ forms an orthonormal basis of $\mathbb{R}^d$ entails
\begin{align*}
&\big(\divs( \Wxi \otimes \Wxi)\big)_i
\\&\hspace{.5cm}=\sum_{j=1}^d  2n_\ast r_\perp \lambda \xi_j \psi_{r_\parallel}^{\prime} \psixi\phi_{(\xi)}^2\xi_i\xi_j+\sum_{j=1}^d 2 \, \psi_{(\xi)}^2 n_\ast r_\perp \lambda \sum_{k=1}^{d-1}(A^k_\xi)_j \partial_{y_k}\phi_{r_\perp} \phixi \xi_i\xi_j\\
&\hspace{.5cm}=\mu^{-1} \partial_t(\psi_{(\xi)}^2\phi_{(\xi)}^2 \xi_i),
\end{align*}
with $y_k:=n_\ast r_\perp \lambda (x-\alpha_\xi)\cdot A^k_\xi,\, k=1,\ldots, d-1$, leading to 
\begin{align*}
&\divs(\mathring{R}_{\text{osc}})+\nabla p_{\text{osc}}\\
&\hspace{.5cm}= \sum_{\xi \in \Lambda} \mathbb{P}_{\neq 0} \Big[ \mathbb{P}_{\neq 0}\big(\Wxi \otimes \Wxi\big)\nabla \axi^2 \Big]+ \mu^{-1} \sum_{\xi \in \Lambda} \mathbb{P}_{\neq 0} \Big[\partial_t \big(\axi^2\psi_{(\xi)}^2\phi_{(\xi)}^2 \xi\big) \Big]\\
&\hspace{1cm}-\mu^{-1} \sum_{\xi \in \Lambda} \mathbb{P}_{\neq 0} \Big[\big(\partial_t \axi^2\big)\psi_{(\xi)}^2\phi_{(\xi)}^2 \xi \Big]-\mu^{-1} \sum_{\xi \in \Lambda} \mathbb{P}\mathbb{P}_{\neq 0} \Big[ \partial_t \big(\axi^2\psi_{(\xi)}^2\phi_{(\xi)}^2 \xi \big) \Big]\\
&\hspace{1cm}+\nabla \big(\rho-p_\ell\big).
\end{align*}
Recalling that the Leray projection is given by $\Id- \mathbb{P}= \nabla \Delta^{-1} \divs$, and that the projection $\mathbb{P}_{\neq 0}$ onto functions with zero mean is, thanks to \eqref{OPg}, equal to $\mathbb{P}_{\geq \frac{r_\perp \lambda_{q+1}}{2}}$, we deduce
\begin{align*}
\divs(\mathring{R}_{\text{osc}})+\nabla p_{\text{osc}}
&=\nabla \bigg[\Delta^{-1} \divs\Big(\mu^{-1}\sum_{\xi \in \Lambda} \mathbb{P}_{\neq 0} \big[\partial_t (\axi^2\psi_{(\xi)}^2\phi_{(\xi)}^2 \xi ) \big]\Big)+\rho-p_\ell \bigg]\\
&\hspace{0.5cm}+ \sum_{\xi \in \Lambda} \mathbb{P}_{\neq 0} \Big[ \mathbb{P}_{\geq \frac{r_\perp \lambda_{q+1}}{2} }\big(\Wxi \otimes \Wxi\big)\nabla \axi^2 \Big]\\
&\hspace{.5cm}-\mu^{-1} \sum_{\xi \in \Lambda} \mathbb{P}_{\neq 0} \Big[ \big(\partial_t \axi^2\big)\psi_{(\xi)}^2\phi_{(\xi)}^2 \xi \Big].
\end{align*}
The desired representation for the oscillation error is finally obtained by using the antidivergence operator $\invdiv$ once again, along with its bilinear version $\mathcal{B}$.\\
So, in short, the Reynolds stress at level $q+1$ can be expressed as
\begin{align*}
\oRqq=\oRlin+\oRnonlin+\oRnonlinn+\oRosc+\oRcor+\oRcom+\oRcomm.
\end{align*}
This formulation clearly reveals that it retains symmetry and tracelessness. The terms $\oRlin,\,  \oRcor$, $\oRcom,\, \oRcomm$ are constructed in such a way that the corresponding pressures absorb the previous trace parts, whereas $\oRnonlin$ and $\oRnonlinn$ have zero trace, because $v_{q+1}$ and $z_{q+1}$ are divergence free. $\oRosc$ is by definition of the antidivergence operators $\invdiv$ and $\mathcal{B}$ traceless.\\
  Also the symmetry of each part, and consequently of $\oRqq$, follows from the construction of the tensor product as well of the symmetry of $\invdiv$, $\mathcal{B}$ and $D$.

\section{Inductive Estimates of the Reynolds Stress $\mathring{R}_{q+1}$} \label{Inductive Estimates for the Reynolds Stress}
To affirm the remaining key bounds \eqref{key bound e} and \eqref{key bound f} we will proceed in a very similar fashion as in Section~\ref{Verifying the Key Bounds v} and Section \ref{Control of the Energy}. That means we will possibly enlarge the parameters $a$ and $b$ at each step to absorb several implicit constants. We will also frequently use the relations
\begin{align*}
  \alpha b>2d+4N^\ast , \qquad (\iota-1)\sigma \alpha> 2\beta b^3 \qquad \text{and} \qquad a \geq \Big((8N^\ast+32) JLA\Big)^c
\end{align*} 
again. It is, however, necessary to involve the additional constraints \eqref{essential bound c} -- \eqref{essential bound f}, which require choosing $p$ within
\begin{gather*}
  \left(1, \frac{2(d-1)\varsigma+\tau}{2d\varsigma+\tau-1+(20d+46)\frac{\alpha}{\sigma-2\varepsilon}} \wedge \frac{4(d-1)\varsigma +2\tau}{2(d-1)\varsigma+\tau+(d-2)\frac{\ln(f(q+1))}{\ln(\lambda_{q+1})}+(20d+48) \frac{\alpha}{\sigma-2\varepsilon}}\right].
\end{gather*}

\subsection{Verifying the Key Bounds on the Level $q+1$} \label{Verifying the Key Bounds R}
We begin by verifying \eqref{key bound e} through a thorough investigation of each component of the Reynolds stress separately.
\subsubsection*{Fifth Key Bound (\ref{key bound e})} 
\textbf{Linear Error} ~\\
 All terms of 
 
 \begin{align*}
 \|\oRlin\|_{L_\Omega^J C_{I}L_x^1} &\lesssim \underbrace{\|\invdiv\partial_t (\wpr+\wc)\|_{L_\Omega^J C_IL_x^p}}_{=:\text{I}}+\underbrace{\|v_\ell+z_\ell\|_{L_\Omega^{2J} C_I L_x^\infty} \|w_{q+1}\|_{L_\Omega^{2J}C_I L_x^p}}_{=:\text{II}}\\
 & \hspace{.5cm}+\underbrace{\|\invdiv \Delta\big(z_{q+1}-z_\ell\big)\|_{L_\Omega^J C_I L_x^1}+\|\invdiv \big(z_{q+1}-z_\ell\big)\|_{L_\Omega^J C_I L_x^1}}_{=:\text{III}}
 \end{align*}
 
will be estimated in order of their appearance.
 
\noindent \begin{tikzpicture}[baseline=(char.base)]
\node(char)[draw,fill=white,
  shape=rounded rectangle,
  drop shadow={opacity=.5,shadow xshift=0pt},
  minimum width=.8cm]
  {\Large I};
\end{tikzpicture}
Thanks to \eqref{wp+wc}, \thref{Lemma 2.3}, \eqref{amplitude function b}, \eqref{fundamental bounds c}, \eqref{key bound f} and \eqref{essential bound c}, the first term is estimated as 
\begin{align*}
\text{I} &\lesssim \sum_{\xi \in \Lambda} \left( \|a_\xi\|_{L_\Omega^J C_{I,x}^1}\|\Vxi\|_{ C_IL_x^p}+\|a_\xi\|_{L_\Omega^J C^0_{I,x}}\|\partial_t\Vxi\|_{ C_I L_x^p}\right) \\
&\lesssim M \ell^{-\frac{5d+11}{2}}\left(1+\|\oRq\|^{\frac{5}{2}}_{L_\Omega^{\frac{5}{2}J} C_\Il L_x^1}\right)r_\perp^{\frac{d-1}{p}-\frac{d-3}{2}}r_\parallel^{\frac{1}{p}-\frac{3}{2}} \mu\\
&\lesssim M \left( 5JLA^q\right)^{5A^q} \lambda_{q+1}^{-2\alpha}\\
&\lesssim M \lambda_q^{5-\alpha b} \delta_{q+3}\\
&\leq \frac{1}{3024} \underline{e} \delta_{q+3}.
\end{align*}

 \noindent \begin{tikzpicture}[baseline=(char.base)]
\node(char)[draw,fill=white,
  shape=rounded rectangle,
  drop shadow={opacity=.5,shadow xshift=0pt},
  minimum width=.8cm]
  {\Large II};
\end{tikzpicture}
To bound the second term we combine \eqref{property vc}, \eqref{key bound c}, \eqref{property za}, \thref{Proposition SHE}, \eqref{wpL}, \eqref{wcL}, \eqref{wtL}, \eqref{key bound f}, \eqref{essential bound d} and \eqref{essential bound e} to deduce
\begin{align*}
\text{II}&\lesssim \left( \|v_q\|_{L_\Omega^{2J}C^1_{\Il,x}}+ \|z_q\|_{L_\Omega^{2J}C_\Il L_x^\infty} \right)\|w_{q+1}\|_{L_\Omega^{2J}C_I L_x^p}\\
&\lesssim \left( (20JLA^{q-1})^{10A^{q-1}} \lambda_q^{d+3}+ f^\frac{d-2}{2}(q)K_G J^{1/2}\right)   \left(M+\frac{M^2}{|\Lambda|}\right)\ell^{-\frac{5d+11}{2}}\\
&\hspace{.5cm}\cdot \left(1+\|\oRq\|^{\frac{5}{2}}_{L_\Omega^{5J}C_\Il L_x^1} \right)\left(r_\perp^{\frac{d-1}{p}-\frac{d-1}{2}}r_\parallel^{\frac{1}{p}-\frac{1}{2}}+r_\perp^{\frac{d-1}{p}-\frac{d-3}{2}}r_\parallel^{\frac{1}{p}-\frac{3}{2}}+r_\perp^{\frac{d-1}{p}-(d-1)}r_\parallel^{\frac{1}{p}-1}\mu^{-1}\right)\\
&\lesssim\left(M+\frac{M^2}{|\Lambda|}\right) \left( \lambda_{q-1}^{10} \lambda_q^{d+3}+ f^{\frac{d-2}{2}}(q) K_G J^{1/2}\right) \lambda_{q+1}^{(10d+22)\frac{\alpha}{\sigma-2\varepsilon}}   (10JLA^q)^{5A^q}\\
&\hspace{.5cm}\cdot \left(r_\perp^{\frac{d-1}{p}-\frac{d-1}{2}}r_\parallel^{\frac{1}{p}-\frac{1}{2}}+r_\perp^{\frac{d-1}{p}-(d-1)}r_\parallel^{\frac{1}{p}-1}\mu^{-1}\right)\\
&\lesssim\left(M+\frac{M^2}{|\Lambda|}\right) \left(  \lambda_q^{d+9-\alpha b}+ K_GJ^{1/2} \lambda_q^{5-\alpha b}\right) \lambda_{q+1}^{-2\beta b^2}\\
&\leq \frac{1}{3024} \underline{e} \delta_{q+3}.
\end{align*}

 \noindent \begin{tikzpicture}[baseline=(char.base)]
\node(char)[draw,fill=white,
  shape=rounded rectangle,
  drop shadow={opacity=.5,shadow xshift=0pt},
  minimum width=.8cm]
  {\Large III};
\end{tikzpicture}
Furthermore, \thref{Lemma 2.3}, \thref{Lemma 2.1}, \eqref{property zb}, \eqref{property zc}, \thref{Proposition SHE} and \eqref{frequency z} imply 
\begin{align*}
\text{III} & \lesssim \|z_{q+1}-z_q\|_{L_\Omega^J C_IH_x^1}+\|z_q-z_\ell\|_{L_\Omega^J C_IH_x^1}\\
&\lesssim \left( f^{-\sigma}(q) +\ell^{\sigma/2-\varepsilon}+ \ell f^{1-\sigma}(q)\right)K_G J^{1/2}\\
&\lesssim \big(\lambda_{q}^{-\frac{2\beta b^3}{\iota-1}}+\lambda_q^{-d-5}\lambda_{q+1}^{-\alpha}+\lambda_{q}^{-2d-10} \lambda_q^{-\frac{2\beta b^3}{\iota-1}}\big)K_G J^{1/2}\\
&\leq \frac{1}{3024} \underline{e} \delta_{q+3}.
\end{align*}

\noindent 
\textbf{Nonlinear Error 1}~\\
Now, we turn our attention to a meticulous analysis of $\oRnonlin$. It holds 
 \begin{align*}
\|\oRnonlin\|_{L_\Omega^J C_{I}L_x^1}&\lesssim \underbrace{ \|\mathcal{A}(Dv_q+Dz_q)\ast_t\varphi_\ell \ast_x \phi_\ell -\mathcal{A}(Dv_q+Dz_q)\|_{L_\Omega^J C_IL_x^1}}_{=:\text{I}}\\
&\hspace{.5cm} +\underbrace{\|\mathcal{A}(Dv_q+Dz_q)-\mathcal{A}(Dv_\ell+Dz_\ell)\|_{L_\Omega^J C_IL_x^1}}_{\text{:= II}}.
\end{align*}

\noindent \begin{tikzpicture}[baseline=(char.base)]
\node(char)[draw,fill=white,
  shape=rounded rectangle,
  drop shadow={opacity=.5,shadow xshift=0pt},
  minimum width=.8cm]
  {\Large I};
\end{tikzpicture}
Owing to \thref{Lemma 2.4},  \eqref{key bound c}, \eqref{OPb}, \eqref{OPa}, \eqref{OPc}, \thref{Proposition SHE} and \eqref{frequency z}, the first term admits
\begin{align*}
\text{I} &\lesssim  \max_{k=1,\iota-1}\left\{ \ell^k\left( \|v_q\|^k_{L_\Omega^J C_{\Il,x}^2}+ \|z_q\|^k_{L_\Omega^J C_\Il W_x^{2,1}} \right)+\ell^{(\sigma/2-\varepsilon)k}\|z_q\|^k_{L_\Omega^J C_\Il^{0,\sigma/2-\varepsilon}W_x^{1,1}}\right\}\\
&\lesssim   \max_{k=1,\iota-1}\Big\{ \ell \left((12JLA^{q-1})^{12A^{q-1}}\lambda_q^{\frac{3}{2}d+4}+\|(\Id-\Delta)^{\frac{1-\sigma}{2}}\mathbb{P}_{\leq f(q)} (\Id-\Delta)^{\frac{1+\sigma}{2}}z\|_{L_\Omega^J C_\Il L_x^2}\right)\Big\}^k\\
&\hspace{.5cm}+  \max_{k=1,\iota-1}\ell^{(\sigma/2-\varepsilon)k} \|z\|^k_{L_\Omega^J C_\Il^{0,\sigma/2-\varepsilon}H_x^{1+\sigma}}  \\
&\lesssim  \max_{k=1,\iota-1}\left\{ \ell \lambda_{q-1}^{12}\lambda_q^{\frac{3}{2}d+4}+\ell f^{1-\sigma}(q)K_G J^{1/2}+\ell^{\sigma/2-\varepsilon} K_G J^{1/2} \right\}^k\\
&\lesssim \left\{\lambda_{q+1}^{-2\alpha} \lambda_q^{-\frac{1}{2}d-5} +\lambda_q^{-2d-10}\lambda_{q}^{-\frac{2\beta b^3}{\iota-1}}K_G J^{1/2}+\lambda_{q+1}^{-\alpha}\lambda_q^{-d-5}K_G J^{1/2}\right\}^{\iota-1}\\
&\lesssim \left\{\lambda_{q+1}^{-\alpha} \lambda_q^{-\frac{1}{2}d-5} +\lambda_q^{-2d-10}K_G J^{1/2}+\lambda_q^{-d-5}K_G J^{1/2}\right\}^{\iota-1} \lambda_{q+3}^{-2\beta }\\
&\leq \frac{1}{2016}  \underline{e} \delta_{q+3}.
\end{align*}

 \noindent \begin{tikzpicture}[baseline=(char.base)]
\node(char)[draw,fill=white,
  shape=rounded rectangle,
  drop shadow={opacity=.5,shadow xshift=0pt},
  minimum width=.8cm]
  {\Large II};
\end{tikzpicture}

\noindent Similarly to the previous calculation, we deduce 
\begin{align*}
\text{II} \leq \frac{1}{2016} \underline{e}\delta_{q+3}
\end{align*}
 as well.

\noindent  \textbf{Nonlinear Error 2}~\\  
\thref{Lemma 2.4}, \eqref{wWiota}, \eqref{property zc}, \eqref{property zb}, \eqref{frequency z} and \thref{Proposition SHE} additionally yield 
\begin{align*}
 &\|\oRnonlinn\|_{L_\Omega^J C_{I}L_x^1}\\
 &\hspace{.3cm}\lesssim \max_{k=1,\iota-1} \|D(v_\ell -v_{q+1}+z_\ell- z_{q+1})\|^k_{L_\Omega^J C_I L_x^1}\\
 &\hspace{.3cm}\lesssim \max_{k=1,\iota-1} \left\{ \|w_{q+1}\|_{L_\Omega^JC_I W_x^{1,1}}+ \|z_\ell-z_q\|_{L_\Omega^J C_I W_x^{1,1}}+\|z_q-z_{q+1}\|_{L_\Omega^J C_I W_x^{1,1}}\right\}^k \\
 &\hspace{.3cm}\lesssim \max_{k=1,\iota-1} \left\{\left(M+\frac{M^2}{|\Lambda|}\right) (14JLA^q)^{7A^q}\lambda_{q+1}^{-2\alpha}\right\}^k\\&\hspace{.5cm}+\max_{k=1,\iota-1} \left\{\left( \ell^{\sigma/2-\varepsilon}+\ell f^{1-\sigma}(q) +f^{-\sigma}(q)\right)\|z\|_{L_\Omega^J C^{0,\sigma/2-\varepsilon}_\Il H_x^{1+\sigma}}\right\}^k\\
 &\hspace{.3cm}\lesssim \max_{k=1,\iota-1} \left\{\left( M+\frac{M^2}{|\Lambda|}\right) \lambda_q^{7-\alpha b}\lambda_{q+1}^{-\alpha}\right\}^k\\
 &\hspace{.5cm}+\max_{k=1,\iota-1} \left\{\left( \lambda_{q+1}^{-\alpha}\lambda_q^{-d-5}+ \lambda_q^{-2d-10} \lambda_{q+1}^{-\frac{2\beta b^3}{\iota-1}}+f^{-\sigma}(q) \right)K_G J^{1/2}\right\}^k\\
  &\hspace{.3cm}\leq \frac{1}{32256} \lambda_2^{2\beta} \underline{e} \left\{\lambda_{q+1}^{-\alpha}+\lambda_{q}^{-\frac{2\beta b^3}{\iota-1}}\right\}^{\iota-1}\\
  & \hspace{.3cm}\leq \frac{1}{1008} \underline{e} \delta_{q+3}.
 \end{align*}

\noindent \textbf{Oscillation Error} ~\\
To control the oscillation error we involve \thref{Lemma 2.2}, \thref{Lemma 2.1}, \eqref{OPa}, \thref{Lemma Building Blocks}, \eqref{amplitude function b}, \eqref{fundamental bounds c}, \eqref{fundamental bounds a}, \eqref{fundamental bounds b}, \eqref{key bound f}, \eqref{essential bound f} and \eqref{essential bound e} to infer
\begin{align*}
&\|\oRosc\|_{L_\Omega^J C_{I}L_x^1}\\
&\hspace{.5cm}\lesssim \sum_{\xi \in \Lambda} (r_\perp \lambda_{q+1})^{-1} \|\nabla \axi^2\|_{L_\Omega^J C_{I,x}^1}\|\Wxi \ootimes \Wxi \|_{C_I L_x^p}\\
&\hspace{1cm}+ \sum_{\xi \in \Lambda}  \mu^{-1} \|\partial_t \axi^2 \|_{L_\Omega^J C_{I,x}^0}\|\psixi^2 \phixi^2\|_{C_I L_x^p} \\
&\hspace{.5cm}\lesssim \sum_{\xi \in \Lambda}(r_\perp \lambda_{q+1})^{-1} \left(\|\axi\|_{L_\Omega^{2J}C_{I,x}^2}\|\axi\|_{L_\Omega^{2J}C_{I,x}^0}+\|\axi\|_{L_\Omega^{2J}C_{I,x}^1}^2 \right) \|\Wxi\|_{C_IL_x^{2p}}^2\\
&\hspace{1cm}+ \sum_{\xi \in \Lambda}\mu^{-1} \|\axi\|_{L_\Omega^{2J}C_{I,x}^1}\|\axi\|_{L_\Omega^{2J}C_{I,x}^0} \|\psixi\|_{C_I L_x^{2p}}^2 \|\phixi\|_{L_x^{2p}}^2  \\
&\hspace{.5cm}\lesssim \frac{M^2}{|\Lambda|} \ell^{-(5d+11)}\left(1+\|\oRq\|_{L_\Omega^{7J}C_\Il L_x^1}^5\right) r_\perp^{\frac{d-1}{p}-d}r_\parallel^{\frac{1}{p}-1} \lambda_{q+1}^{-1}\\
&\hspace{1cm}+ \frac{M^2}{|\Lambda|} \ell^{-(3d+6)}(1+\|\oRq\|_{L_\Omega^{5J}C_\Il L_x^1}^3)r_\perp^{\frac{d-1}{p}-(d-1)}r_\parallel^{\frac{1}{p}-1} \mu^{-1}\\
&\hspace{.5cm}\lesssim \frac{M^2}{|\Lambda|} (14JLA^q)^{10A^q}\lambda_{q+1}^{-2\alpha}\\
&\hspace{.5cm}\lesssim \frac{M^2}{|\Lambda|}\lambda_q^{10-\alpha b} \lambda_{q+1}^{-\alpha}\\
&\hspace{.5cm}\leq \frac{1}{1008}  \underline{e} \delta_{q+3}.
\end{align*}

\noindent \textbf{Corrector Error} ~\\
 We continue by estimating the corrector error as
  \begin{align*}
 &\|\oRcor\|_{L_\Omega^J C_{I}L_x^1}\\
 &\hspace{.5cm} \lesssim \left(\|\wpr\|_{L_\Omega^{2J}C_I L_x^2}+ \|w_{q+1}\|_{L_\Omega^{2J}C_I L_x^2} \right)\|\wc+\wt\|_{L_\Omega^{2J}C_I L_x^2}\\
&\hspace{.5cm}\lesssim \|\wpr\|_{L_\Omega^{2J}C_I L_x^2} \|\wc+\wt\|_{L_\Omega^{2J}C_IL_x^2}+ \|\wc+\wt\|^2_{L_\Omega^{2J}C_I L_x^2}\\
 &\hspace{.5cm}\lesssim M\left\{ \|\oRq\|_{L_\Omega^J C_\Il L_x^1}^{1/2}+\bar{e}^{1/2} \delta^{1/2}_{q+1}+\ell^{-\frac{5d+11}{2}}\left(1+\|\oRq\|_{L_\Omega^{5J}C_\Il L_x^1}^{5/2}\right) r_\perp^{-1/2} \lambda_{q+1}^{-1/2}\right\}\\
 &\hspace{1cm}\cdot\left(M+\frac{M^2}{|\Lambda|}\right)\ell^{-\frac{5d+11}{2}} \left(1+\|\oRq\|_{L_\Omega^{5J}C_\Il L_x^1}^{5/2}\right)  \left( r_\perp r_\parallel^{-1}+r_\perp^{-\frac{d-1}{2}}r_\parallel^{-1/2}\mu^{-1}\right)\\
  &\hspace{1cm}+\left(M+\frac{M^2}{|\Lambda|}\right)^2\ell^{-(5d+11)} \left(1+\|\oRq\|_{L_\Omega^{5J}C_\Il L_x^1}^{5}\right)  \left( r_\perp r_\parallel^{-1}+r_\perp^{-\frac{d-1}{2}}r_\parallel^{-1/2}\mu^{-1}\right)^2 \\
  &\hspace{.5cm}\lesssim \left(M+\frac{M^2}{|\Lambda|}\right)^2\left\{ \bar{e}^{1/2} \delta_{q+1}^{1/2}+ (10JLA^q)^{5A^q}\lambda_{q+1}^{-2\alpha}\right\} (10JLA^q)^{5A^q}\lambda_{q+1}^{-2\alpha}\\
  &\hspace{1cm}+\left(M+\frac{M^2}{|\Lambda|}\right)^2(10JLA^q)^{10A^q}\lambda_{q+1}^{-2\alpha}\\
  &\hspace{.5cm}\lesssim \left(M+\frac{M^2}{|\Lambda|}\right)^2 \left(\bar{e}^{1/2}\delta_{q+1}^{1/2}+ \lambda_q^{5-\alpha b}\lambda_{q+1}^{-\alpha}\right)\lambda_{q}^{5-\alpha b} \lambda_{q+1}^{-\alpha}+\left(M+\frac{M^2}{|\Lambda|}\right)^2\lambda_q^{10-\alpha b}\lambda_{q+1}^{-\alpha}\\
  &\hspace{.5cm}\leq \frac{1}{1008} \underline{e} \delta_{q+3},
 \end{align*}
 relying basically on \eqref{wpL2}, \eqref{wcL}, \eqref{wtL}, \eqref{key bound e}, \eqref{key bound f}, \eqref{essential bound a}, \eqref{essential bound g} and \eqref{essential bound h}.

 \noindent
 \textbf{Commutator Error 1}~\\
 Next, we investigate the behavior of $\oRcom$. It holds
 \begin{align*}
\|\oRcom\|_{L_\Omega^J C_I L_x^1}&\lesssim \underbrace{\|(v_\ell+z_\ell)\ootimes (v_\ell+z_\ell)- (v_\ell+z_\ell) \ootimes (v_q+z_q)\|_{L_\Omega^J C_I L_x^1}}_{=:\text{I}}\\
&\hspace{.5cm}+\underbrace{\|(v_\ell+z_\ell)\ootimes (v_q+z_q)- (v_q+z_q) \ootimes (v_q+z_q)\|_{L_\Omega^J  C_I L_x^1}}_{=:\text{II}}\\
&\hspace{.5cm}+\underbrace{\|(v_q+z_q)\ootimes (v_q+z_q)- \Big[(v_q+z_q) \ootimes (v_q+z_q)\Big]\ast_t\varphi_\ell \ast_x \phi_\ell\|_{L_\Omega^J C_I L_x^1}.}_{=:\text{III}}
\end{align*}

\noindent \begin{tikzpicture}[baseline=(char.base)]
\node(char)[draw,fill=white,
  shape=rounded rectangle,
  drop shadow={opacity=.5,shadow xshift=0pt},
  minimum width=.8cm]
  {\Large I \& II };
\end{tikzpicture} 
~\\ 
Owing to \eqref{property vd} we obtain 
\begin{align*}
\text{I}+ \text{II} &\lesssim\|v_q+z_q\|_{L_\Omega^{2J}C_\Il L_x^2}\|v_\ell+z_\ell-v_q-z_q\|_{L_\Omega^{2J}C_I L_x^2}\\
&\lesssim \left( \|v_q\|_{L_\Omega^{2J}C_\Il L_x^2}+\|z_q\|_{L_\Omega^{2J}C_\Il L_x^2}\right) \ell \Big[\|v_q\|_{L_\Omega^{2J}C_{\Il,x}^1}+\| z_q\|_{L_\Omega^{2J }C_\Il H_x^{1+\sigma}}\Big]\\
&\hspace{.5cm} +\left( \|v_q\|_{L_\Omega^{2J}C_\Il L_x^2}+\|z_q\|_{L_\Omega^{2J}C_\Il L_x^2}\right)\ell^{\sigma/2-\varepsilon} \|z_q\|_{L_\Omega^{2J}C_\Il^{0,\sigma/2-\varepsilon}H_x^{1+\sigma}}.
\end{align*}

 \noindent \begin{tikzpicture}[baseline=(char.base)]
\node(char)[draw,fill=white,
  shape=rounded rectangle,
  drop shadow={opacity=.5,shadow xshift=0pt},
  minimum width=.8cm]
  {\Large III};
\end{tikzpicture}
~\\
Moreover, one has
\begin{align*}
\text{III} &\lesssim \|v_q+z_q\|_{L_\Omega^{2J} C_\Il L_x^2}\bigg\{ \ell \|\partial_t v_q\|_{L_\Omega^{2J}C_\Il L_x^2}+\ell \|\nabla v_q\|_{L_\Omega^{2J}C_\Il L_x^2}\bigg\}\\
&\hspace{.2cm}+\|v_q+z_q\|_{L_\Omega^{2J} C_\Il L_x^2}\bigg\{  \ell^{\sigma/2-\varepsilon} \|z_q\|_{L_\Omega^{2J}C^{0,\sigma/2-\varepsilon}_\Il L_x^2}+\ell \|\nabla z_q\|_{L_\Omega^{2J}C_\Il L_x^2}\bigg\}\\
&\lesssim \left( \|v_q\|_{L_\Omega^{2J}C_\Il L_x^2}+\|z_q\|_{L_\Omega^{2J}C_\Il L_x^2}\right)\\
&\hspace{.5cm} \cdot \left\{\ell \Big[\|v_q\|_{L_\Omega^{2J}C_{\Il,x}^1}+\| z_q\|_{L_\Omega^{2J }C_\Il H_x^{1+\sigma}}\Big] +\ell^{\sigma/2-\varepsilon} \|z_q\|_{L_\Omega^{2J}C_\Il^{0,\sigma/2-\varepsilon}H_x^{1+\sigma}}\right\}.
\end{align*}
Thanks to \eqref{key bound a}, \eqref{delta series bound}, \eqref{OPb}, \eqref{OPa}, \thref{Proposition SHE} and \eqref{key bound c} we conclude 
\begin{align*}
\|\oRcom\|_{L_\Omega^J C_I L_x^1}
&\lesssim  (M_0 \bar{e}^{1/2}+K_G J^{1/2})\\&\hspace{.5cm}\cdot\left\{ \ell \Big[(20JLA^{q-1})^{10A^{q-1}}\lambda_q^{d+3}+K_G J^{1/2} \Big]+ \ell^{\sigma/2-\varepsilon}K_G J^{1/2}\right\}\\
&\lesssim (M_0 \bar{e}^{1/2}+K_G J^{1/2})K_G J^{1/2} \left\{ \ell \lambda_{q-1}^{10}\lambda_q^{d+3} + \ell^{\sigma/2-\varepsilon}\right\}\\
&\lesssim (M_0 \bar{e}^{1/2}+K_G J^{1/2})K_G J^{1/2} \left\{\lambda_q^{d+4-\alpha b}+ \lambda_{q}^{-d-5}\right\}\lambda_{q+1}^{-\alpha}\\
&\leq \frac{1}{1008}  \underline{e} \delta_{q+3}. 
\end{align*}
 \textbf{Commutator Error 2}~\\
To estimate the last term of the Reynolds stress, we appeal to \eqref{key bound a}, \eqref{delta series bound}, \eqref{property zb}, \eqref{property zc}, \eqref{OPa}, \thref{Proposition SHE} and \eqref{frequency z} to deduce 
\begin{align*}
\|\oRcomm\|_{L_\Omega^J C_I L_x^1}& \lesssim \|v_{q+1} \ootimes (z_{q+1}-z_\ell)\|_{L_\Omega^J C_I L_x^1} + \|  (z_{q+1}-z_\ell) \ootimes v_{q+1}\|_{L_\Omega^J C_I L_x^1}\\
&\hspace{.5cm}+ \|z_{q+1}\ootimes (z_{q+1}-  z_q)\|_{L_\Omega^J C_I L_x^1}  + \| (z_{q+1}-z_q)\ootimes z_q\|_{L_\Omega^J C_I L_x^1} \\
&\hspace{.5cm}+ \| z_q \ootimes (z_q- z_\ell) \|_{L_\Omega^J C_I L_x^1}+\|(z_q  -z_\ell) \ootimes z_\ell \|_{L_\Omega^J C_I L_x^1}  \\
& \lesssim \|v_{q+1}\|_{L_\Omega^{2J} C_I L_x^2} \Big( \|z_{q+1}-z_q\|_{L_\Omega^{2J} C_I L_x^2}+ \|z_q-z_\ell \|_{L_\Omega^{2J} C_I L_x^2} \Big)\\
& \hspace{.5cm} +\Big(\|z_{q+1}\|_{L_\Omega^{2J} C_I L_x^2}+\|z_{q}\|_{L_\Omega^{2J} C_I L_x^2} \Big) \|z_{q+1}-z_q\|_{L_\Omega^{2J} C_I L_x^2} \\
&\hspace{.5cm} +\|z_q\|_{L_\Omega^{2J} C_\Il L_x^2}  \|z_q-z_\ell\|_{L_\Omega^{2J} C_I L_x^2} \\
& \lesssim M_0\bar{e}^{1/2} K_G^2 J\left(f^{-(1+\sigma)}(q)+ \ell^{\sigma/2-\varepsilon}+ \ell f^{1-\sigma}(q) \right)\\
& \lesssim M_0\bar{e}^{1/2} K_G^2 J\left(\lambda_{q}^{-2\beta b^3 \frac{1+\sigma}{(\iota-1)\sigma}}+\lambda_{q+1}^{-\alpha} \lambda_q^{-d-5}+ \lambda_q^{-2d-10} \lambda_{q+1}^{-\frac{2\beta b^3}{\iota-1}} \right)\\
& \leq\frac{1}{1008} \underline{e} \delta_{q+3}.
\end{align*} 
\noindent 
Summing up all these bounds verifies \eqref{key bound e} at the level $q+1$. 
\subsubsection*{Sixth Key Bound (\ref{key bound f})} 
The last key bound is estimated in a very similar fashion. \\
\textbf{Linear Error} ~\\
We compute
 \begin{align*}
 \|\oRlin\|_{L_\Omega^j C_{I}L_x^1}& \lesssim M \left(  5j LA^{q} \right)^{5A^{q}}\lambda_{q+1}^{-2\alpha}+\left(M+\frac{M^2}{|\Lambda|}\right) K_G (20jLA^q)^{15A^q} \lambda_q^{d+3} \lambda_{q+1}^{-2 \alpha}\\
 &\hspace{.5cm}+ \left( f^{-\sigma}(q) +\ell^{\sigma/2-\varepsilon}+ \ell f^{1-\sigma}(q)\right)K_G j^{1/2}\\
 &\lesssim  \left(M+\frac{M^2}{|\Lambda|}\right) K_G \left( \lambda_q^{d+3- \alpha b} \lambda_{q+1}^{- \alpha}+\lambda_{q}^{-\frac{2\beta b^3}{\iota-1}}+ \lambda_{q+1}^{-\alpha}\right)\left(  2j LA^{q+1} \right)^{2A^{q+1}}\\
 &\leq \frac{1}{7} \left(  2j LA^{q+1} \right)^{2A^{q+1}}.
  \end{align*}
  \noindent
\textbf{Nonlinear Error 1} \\
We also get 
 \begin{align*}
&\|\oRnonlin\|_{L_\Omega^j C_{I}L_x^1}\\
&\hspace{.2cm}\lesssim  \max_{k=1,\iota-1}\left\{ \ell \left((12jLA^{q-1})^{12A^{q-1}}\lambda_q^{\frac{3}{2}d+4}+f^{1-\sigma}(q)K_G j^{1/2}\right)+\ell^{\sigma/2-\varepsilon}K_G j^{1/2} \right\}^k\\
&\hspace{.2cm}\lesssim \max_{k=1,\iota-1} \left\{  \lambda_q^{\frac{3}{2}d+4-\alpha b}\lambda_{q+1}^{-\alpha}+ \lambda_{q}^{-\frac{2\beta b^3}{\iota-1}} + \lambda_{q+1}^{-\alpha}  \right\}^k  K_G (2jLA^{q+1})^{2A^{q+1}}\\
&\hspace{.2cm}\leq \frac{1}{7} (2jLA^{q+1})^{2A^{q+1}}.
\end{align*}
\noindent 
\textbf{Nonlinear Error 2}~\\  
Additionally we find \pagebreak
\begin{align*}
&\|\oRnonlinn\|_{L_\Omega^j C_{I}L_x^1}\\
&\lesssim \max_{k=1,\iota-1} \left\{\left(M+\frac{M^2}{|\Lambda|}\right) (7jLA^q)^{7A^q}\lambda_{q+1}^{-2\alpha}+\left( \ell^{\sigma/2-\varepsilon}+\ell f^{1-\sigma}(q) +f^{-\sigma}(q)\right)K_Gj^{1/2}\right\}^k\\
 &\lesssim  \max_{k=1,\iota-1} \left\{   \lambda_{q+1}^{-\alpha}+\lambda_q^{-\frac{2\beta b^3}{\iota-1}}\right\}^k  \left(M+\frac{M^2}{|\Lambda|}\right)K_G (2jLA^{q+1})^{2A^{q+1}}\\
&\leq \frac{1}{7}(2jLA^{q+1})^{2A^{q+1}}.
\end{align*}

 \noindent
\textbf{Oscillation Error} \\
Moreover, it holds
\begin{align*}
\|\oRosc\|_{L_\Omega^j C_{I}L_x^1} &\lesssim \frac{M^2}{|\Lambda|} (14jLA^q)^{10A^q}\lambda_{q+1}^{-2\alpha} \leq \frac{1}{7}(2jLA^{q+1})^{2A^{q+1}}.
\end{align*}

\noindent 
\textbf{Corrector Error} ~\\
The corrector error is estimated as 
 \begin{align*}
 &\|\oRcor\|_{L_\Omega^j C_{I}L_x^1}\\
 & \lesssim M\Big\{ (2jLA^q)^{A^q}+\bar{e}^{1/2}\delta_{q+1}^{1/2}+ (10jLA^q)^{5A^q}\lambda_{q+1}^{-2 \alpha}\Big\} \left(M+\frac{M^2}{|\Lambda|}\right) (10jLA^q)^{5A^q}\lambda_{q+1}^{-2 \alpha}\\
 &\hspace{.5cm}+ \left(M+\frac{M}{|\Lambda|}\right)^2  (10jLA^q)^{10A^q} \lambda_{q+1}^{-2 \alpha}\\
 &\lesssim \left(M+\frac{M^2}{|\Lambda|} \right)^2 \bar{e}^{1/2} \lambda_{q+1}^{-2 \alpha}(10jLA^q)^{10A^q} \\
 &\leq \frac{1}{7} (2jLA^{q+1})^{2A^{q+1}}
 \end{align*}
either.

\noindent 
 \textbf{Commutator Error 1}~\\
 The same bound is derived for the Commutator Error $1$ as
\begin{align*}
\|\oRcom\|_{L_\Omega^j C_I L_x^1}
&\lesssim  \big(M_0(10jLA^{q-1})^{5A^{q-1}}+M_0 \bar{e}^{1/2}+K_G j^{1/2}\big)\\
&\hspace{.5cm} \cdot\left\{ \ell \Big[(20jLA^{q-1})^{10A^{q-1}}\lambda_q^{d+3}+K_G j^{1/2} \Big]+ \ell^{\sigma/2-\varepsilon}K_G j^{1/2}\right\}\\
&\lesssim  M_0K^2_G \bar{e}^{1/2}\Big( \lambda_q^{d+3-\alpha b} \lambda_{q+1}^{-\alpha} +\lambda_{q+1}^{-\alpha}\Big) (20jLA^{q-1})^{15A^{q-1}} \\
&\leq \frac{1}{7} (2jLA^{q+1})^{2A^{q+1}}.
\end{align*}
\noindent 
 \textbf{Commutator Error 2}\\
 Also the Commutator Error $2$ admits this bound;
\begin{align*}
\|\oRcomm\|_{L_\Omega^j C_I L_x^1}&\lesssim \Big(M_0(10jLA^{q-1})^{5A^{q-1}}+M_0\bar{e}^{1/2}+ K_G j^{1/2}\Big)\\
&\hspace{.5cm}\cdot\left(f^{-(1+\sigma)}(q)+\ell^{\sigma/2-\varepsilon}+ \ell f^{1-\sigma}(q)\right)K_Gj^{1/2}\\
&\lesssim M_0K_G^2\bar{e}^{1/2}\left(\lambda_{q}^{-2\beta b^3\frac{1+\sigma}{(\iota-1)\sigma}}+ \lambda_{q+1}^{-\alpha}+ \lambda_{q}^{-\frac{2\beta b^3}{\iota-1}}\right)(10jLA^{q-1})^{6A^{q-1}}\\
&\leq \frac{1}{7}(2jLA^{q+1})^{2A^{q+1}}.
\end{align*}
~\\
Altogether this confirms the sixth key bound \eqref{key bound f} at the level $q+1$.

\subsection{Convergence of the Sequence} \label{Convergence of the Sequence R}
Moreover, it follows immediately from \eqref{key bound e} that $\big(\oRq\big)_{q \geq 0}$ converges in
 \begin{align*}
\left( L^{J}\big(\Omega;C\big([0,\infty);L^1\big( \mathbb{T}^d;\mathbb{R}^{d\times d}\big)\big)\big),\sup_{|I|=1}\|\cdot\|_{L_\Omega^JC_I L_x^1} \right)
\end{align*}
 to $0$.

\section{Adaptedness of the Iterates \& Deterministic Initial Values} \label{Conclusion of the Convex Integration Scheme}
To complete the proof of \thref{Proposition Main Iteration}, it remains to establish the adaptedness of the iterates $\big(v_q(t),\oRq(t)\big)_{t\geq 0}$ for any $q\in \mathbb{N}_0$, as well as the deterministic nature of their initial values $\big(v_q(0),\oRq(0)\big)$. The verification of these properties relies on the same arguments as those in \cite{Be23}. More precisely, the starting point (cf. Section \ref{Start of the Iteration}) and hence, by induction, each component of the velocity and Reynolds stress at level $q+1$ exhibits both of these characteristics. Therefore, this must also be true for the entire process at level $q+1$. The relationships among the components of the velocity $v_{q+1}$ are illustrated in Figure \ref{Implication scheme}.\\
It is important to emphasize that the time mollifier must be compactly supported in the interval $(0,1)$ (see Section \ref{Mollification}). This choice is essential to ensure the validity of adaptedness when applied to mollified fields.
\end{proof}

\chapter{ Existence \& Non-Uniqueness of Ergodic Leray--Hopf Solutions} \label{Existence and Non-Uniqueness of ergodic Leray--Hopf Solutions}
 In what follows we will frequently denote the Borel $\sigma$-algebra on $\mathcal{T}$ by  $\mathcal{B\left(\mathcal{T}\right)}$.
\section{Existence \& Non-Uniqueness of Stationary Leray--Hopf Solutions} \label{Existence and Non-Uniqueness of stationary Leray--Hopf Solutions}
Before ultimately proving the existence and non-uniqueness of ergodic Leray--Hopf solutions, we will first focus on establishing the existence of stationary Leray--Hopf solutions as a preliminary step.\\
The proof is based on a Krylov--Bogoliubov-type argument, following the methodology in \cite{HZZ25}. Specifically, we aim to identify a convergent subsequence of the sequence of ergodic averages, whose limit is, according to Skorokhod's representation theorem, the law of the desired solution.
\begin{samepage}
\begin{theorem}\thlabel{Theorem Stationary}
There exist infinitely many stationary Leray--Hopf solutions to the power-law system of shear-thinning fluids \eqref{PLF} with power index $\iota \in \left(1,\frac{2d}{d+2} \right)$.\\ 
More precisely, there exists a stationary Leray--Hopf solution $\widetilde{u}$ on $\Big( (\widetilde{\Omega},\widetilde{\mathcal{F}},\widetilde{\mathcal{P}}), (\widetilde{\mathcal{F}_t})_{t\geq 0}, \widetilde{B}\Big)$ to \eqref{PLF}, belonging $\mathcal{P}$-a.s. to
\begin{align*} 
C\big([0,\infty);H^\gamma\big( \mathbb{T}^d;\mathbb{R}^d\big) \big)\cap C_{\rm{loc}}^{0,\gamma}\big([0,\infty); L^{2}\big( \mathbb{T}^d;\mathbb{R}^d\big)\cap W^{1+\gamma,\iota}\big( \mathbb{T}^d;\mathbb{R}^d\big)\big)
\end{align*}
for some $\gamma \in (0,1)$. For fixed $J\geq 1,\, N\in \mathbb{N}_0$, arbitrary small $\epsilon>0$ and some constant $C>0$ this solution particularly admits
\begin{align} \label{Stationary Property 1}
\mathbf{\widetilde{E}}\|\widetilde{u}\|_{C_{[0,N]}H_x^\gamma }^{2J}+\mathbf{\widetilde{E}}\|\widetilde{u}\|_{C^{0,\gamma}_{[0,N]}L_x^2 }^{2J} +\mathbf{\widetilde{E}}\|\widetilde{u}\|_{C^{0,\gamma}_{[0,N]}W_x^{1+\gamma,\iota} }^{2J} \leq C N
\end{align}
as well as
\begin{align} \label{Stationary Property 2}
\sup_{|I|=1}\mathbf{\widetilde{E}}\|\widetilde{u}-\widetilde{z}\|_{C_I W_x^{1,\iota} }^{2J} \leq \epsilon,
\end{align}
where $\widetilde{z}$ is the mild solution \eqref{SHE Solution} to the modified stochastic heat equation \eqref{Stochastic Heat Equation} with $B$ replaced by $\widetilde{B}$.
 \end{theorem}
\end{samepage}

\begin{remark}
Let $u$ be a Leray--Hopf solution to \eqref{PLF} on $\big((\Omega, \mathcal{F}, \mathcal{P}), B\big)$, as given in \thref{Theorem Weak Solution}, and let $\widetilde{u}$ be a stationary Leray--Hopf solution on $\big((\widetilde{\Omega}, \widetilde{\mathcal{F}}, \widetilde{\mathcal{P}}), \widetilde{B}\big)$, obtained in \thref{Theorem Stationary}. Then, as a consequence of a Krylov--Bogoliubov argument, we even  find a subsequence of the sequence of ergodic averages
\begin{align*}
\left(\frac{1}{T} \int_0^{T} \mathcal{P}\Big(S_s(u,B) \in \cdot\Big) \,ds \right)_{T> 0}
\end{align*}
on $\mathcal{B}\left(\mathcal{T}\right)$, which converges weakly to the law of $\big(\widetilde{u}, \widetilde{B} \big)$.
\end{remark}
\medskip
 \begin{proof}[Proof of \thref{Theorem Stationary}]~\\
 Since we intend to apply the Krylov--Bogoliubov procedure, we initiate the proof by introducing the standard sequence of ergodic averages.
 \subsection{Existence \& Limit of the Subsequence} \label{Tightness}
Let $u:=v+z$ be a Leray--Hopf solution to \eqref{PLF} on $\big((\Omega, \mathcal{F}, \mathcal{P}), B\big)$, constructed as in \thref{Theorem Weak Solution}.\\
For $T>0$ let us denote the ergodic averages by $\nu_T$, that is let us define
\begin{gather*}
\nu_T:= \frac{1}{T}\int_0^T \mathcal{P}\Big(S_s (u,B)\in \cdot \Big) \,ds 
\end{gather*}
on $\mathcal{B}\left(\mathcal{T} \right)$.\\
  We will use Prokhorov's theorem in order to show that there exists a subsequence of $\left(\nu_T\right)_{T>0}$, which converges weakly to some probability measure $\nu$ on $\mathcal{B}\left(\mathcal{T} \right)$. It turns out that this measure serves, by Skorokhod's representation theorem, the law of a stationary solution $(\widetilde{u},\widetilde{B})\in \mathcal{T}$.\\
Prokhorov's theorem requires:

\subsubsection*{Tightness of $\left(S^1_s u\right)_{s\geq 0}$ on $ \mathcal{T}_1$} \label{Tightness u}
Taking Hölder's inequality and Remark \ref{Remark Weak Solution} into account gives
\begin{align} 
\E\|S_s^1u&\|_{ C_{[0,N]} H_x^\gamma}+\E\|S_s^1u\|_{ C^{0,\gamma}_{[0,N]} L_x^2}+\E\|S_s^1u\|_{C^{0,\gamma}_{[0,N]} W_x^{1+\gamma,\iota}}\label{Tightness Bound}\\
&\hspace{.5cm}\leq  \sum_{i=s}^{N+s-1}\left(\E\|u\|_{ C_{[i,i+1]} H_x^\gamma}+\E\|u\|_{ C^{0,\gamma}_{[i,i+1]} L_x^2}+\E\|u\|_{C^{0,\gamma}_{[i,i+1]} W_x^{1+\gamma,\iota}}\right)\notag \\
&\hspace{.5cm}\leq RN \notag
\end{align}
for an arbitrary but fixed $  N \in \mathbb{N}$, all $s\geq 0$ and some $R>0$. In the light of this, let us consider the set
\begin{align*}
K_1^N:=\left\{f \in \mathcal{T}_1^N\; \middle| \; 
 \|f\|_{C_{[0,N]}H_x^\gamma}+\|f\|_{C^{0,\gamma}_{[0,N]}L_x^2}+\|f\|_{C^{0,\gamma}_{[0,N]} W_x^{1+\gamma,\iota}}\leq  \frac{\pi^2RN^{3}}{6\epsilon_1}
\right\}
\end{align*}
for some arbitrary small $\epsilon_1>0$ and with
\begin{align*}
\mathcal{T}_1^N:=C\big([0,N];L_\text{free}^2\big( \mathbb{T}^d;\mathbb{R}^d\big)\cap W^{1,\iota}\big( \mathbb{T}^d;\mathbb{R}^d\big)\big).
\end{align*}
 Due to Arzelà--Ascoli's theorem, $K_1^N$ is compact in 
 \begin{align*}
 \left(\mathcal{T}_1^N,\|\cdot\|_{ C_{[0,N]}L_x^2}+\|\cdot\|_{ C_{[0,N]}W_x^{1,\iota}}\right),
 \end{align*}
 provided $K_1^N$ is equicontinuous and $K_1^N(t):=\{f(t): f\in K_1^N\}$ is for all $t\in [0,N]$ relatively compact in $\left(L^2\left(\mathbb{T}^d; \mathbb{R}^d\right)\cap W^{1,\iota}\big( \mathbb{T}^d;\mathbb{R}^d\big),\|\cdot\|_{ L_x^2}+\|\cdot\|_{ W_x^{1,\iota}} \right)$. 
This, however, follows immediately from the Hölder condition and Rellich--Kondrachov's theorem. At this point, we want to emphasize that it is of great importance that the solution belongs to $W^{1+\gamma,\iota}$ in space and is Hölder continuous in time, necessitating the use of the $C_{I,x}^4$-norm during the convex integration scheme (cf. \eqref{key bound c}, \eqref{target series}).\\
A standard diagonal argument (details are postponed to Appendix \ref{Appendix Tightness}) implies that  
\begin{align*} 
K_1:=\bigcap\limits_{N=1}^\infty K_1^N
\end{align*}
is compact in $\mathcal{T}_1$.\\
Finally, making use of Markov's inequality and \eqref{Tightness Bound} we obtain
\begin{align*} 
& \mathcal{P}\left(S^1_s u \in K_1 \right)\\
&\hspace{.5cm}\geq 1-\sum_{N=1}^\infty \mathcal{P}\left( \|S_s^1 u\|_{C_{[0,N]} H_x^\gamma}+\|S_s^1 u\|_{C^{0,\gamma}_{[0,N]}L_x^2}+\|S_s^1 u\|_{C^{0,\gamma}_{[0,N]}W_x^{1+\gamma,\iota}} >  \frac{\pi^2RN^{3}}{6\epsilon_1} \right)\\
 &\hspace{.5cm}\geq 1-\sum_{N=1}^\infty \frac{6\epsilon_1}{\pi^2 RN^{3}} \E \left[  \|S_s^1 u\|_{C_{[0,N]} H_x^\gamma}+\|S_s^1 u\|_{C^{0,\gamma}_{[0,N]}L_x^2}+\|S_s^1 u\|_{C^{0,\gamma}_{[0,N]} W_x^{1+\gamma,\iota}}\right]\\
 &\hspace{.5cm}\geq  1-\epsilon_1
\end{align*}
for any $s\geq 0$, which establishes the tightness of $\left(S^1_s u\right)_{s\geq 0}$ on $\mathcal{T}_1$.

\subsubsection*{Tightness of $\left(S_s^2 B\right)_{s\geq 0}$ on $\mathcal{T}_2$}\label{Subsection 6.1.2}
The tightness of $\left(S_s^2 B\right)_{s\geq 0}$ on $\mathcal{T}_2$ follows from the fact, that the law of $t\mapsto S_s^2 B_t$ remains the same for all $s\geq 0$. That means, for every $\epsilon_2>0$ we find some compact set $K_2$, satisfying
\begin{align} \label{Tightness B}
\inf_{s \geq 0} \mathcal{P}\left(S^2_s B \in K_2 \right)\geq 1-\epsilon_2. 
\end{align}

\noindent Keeping in mind that $K:= K_1 \times K_2$ retains the compactness of $K_1$ and $K_2$, we conclude, as a consequence of the previous tightness results, that the ergodic average is tight as well: 

\begin{align*}
&\inf_{T>0}\nu_T(K)\\
&\hspace{.5cm}=\inf_{T>0}\frac{1}{T}\int_0^T \mathcal{P}\left(S_s (u, B)\in K \right)\, ds\\
&\hspace{.5cm}\geq \inf_{T>0}\frac{1}{T}\int_0^T \bigg( 1-\mathcal{P}\big(S_s (u, B)\in (K_1\times \mathcal{T}_2)^c \big)-\mathcal{P}\big(S_s (u, B)\in (\mathcal{T}_1\times K_2)^c \big)\bigg)\, ds\\
&\hspace{.5cm}= \inf_{T>0}\frac{1}{T}\int_0^T  \bigg(\mathcal{P}\big(S_s (u, B)\in K_1\times \mathcal{T}_2 \big)+\mathcal{P}\big(S_s (u, B)\in \mathcal{T}_1\times K_2 \big)-1 \bigg) \,ds\\
&\hspace{.5cm}\geq 1-(\epsilon_1+\epsilon_2).
\end{align*}
 This allows to apply Prokhorov's theorem, to obtain a subsequence $\left(\nu_{T_n} \right)_{n>0}$ of $\left(\nu_T \right)_{T>0}$, that converges weakly to some limit $\nu$ on $\mathcal{B}\left( \mathcal{T}\right)$.  By Skorokhod's representation theorem, there exists a sequence of random variables $(\widetilde{u}_n,  \widetilde{B}_n)_{n \in \mathbb{N}}$, possibly defined on a different probability space $(\widetilde{\Omega}, \widetilde{\mathcal{F}}, \widetilde{\mathcal{P}})$ with values in $ \mathcal{T}$, which converge $\widetilde{\mathcal{P}}$-a.s. to some random variable $(\widetilde{u},\widetilde{B})$ in $\mathcal{T}$ and whose laws correspond to $ \left(\nu_{T_n}\right)_{ n\in \mathbb{N}}$ and $\nu$, respectively.\\
Let us introduce the set
\begin{align*} 
M:=\left\{(u,B)\in \mathcal{T}\;  \middle|  \; \begin{array}{l}
   \int_{\mathbb{T}^d}  \varphi(x) B_t(x) \,dx 
 = \int_{\mathbb{T}^d}   \Big(u(t,x)-u(0,x)\Big)\cdot\varphi(x)\,dx\\
\hspace{0.2cm} -\int_0^t \int_{\mathbb{T}^d} (u \otimes u)( s,x): \nabla \varphi^T(x)\,dx \,ds\\
\hspace{0.2cm}+ \int_0^t \int_{\mathbb{T}^d}  \mathcal{A}(Du(s,x)): \nabla \varphi^T(x)\,dx\,ds \\
\mathcal{P}-\text{a.s.},\,  \forall \varphi \in C^\infty (\mathbb{T}^d; \mathbb{R}^d); \divs (\varphi)=0, \forall \,t \geq 0
\end{array} \right\},
\end{align*}
facilitating the derivation that $\widetilde{u}$ is, in fact, an analytically weak solution to \eqref{PLF} on $\big((\widetilde{\Omega}, \widetilde{\mathcal{F}}, \widetilde{\mathcal{P}}), \widetilde{B}\big)$.
Since $(u,B)$ is an analytically weak solution to \eqref{PLF}, the same applies to $(\widetilde{u}_n,\widetilde{B}_n)$, as we have
\begin{align*}
\widetilde{\mathcal{P}}\Big( (\widetilde{u}_n,\widetilde{B}_n)\in M \Big)=\nu_{T_n}(M)=\frac{1}{T_n} \int_0^{T_n} \mathcal{P}(S_s(u,B)\in M)\,ds=1.
\end{align*}
Passing to the limit in $\mathcal{T}$ (cf. Section \ref{Existence}), confirms that $(\widetilde{u}, \widetilde{B})$ also solves \eqref{PLF} in an analytically weak sense, as desired. 
 \subsection{Stationarity} \label{Stationarity}
This solution is even stationary, a property that naturally arises from the weak convergence of $(\nu_{T_n})_{n>0}$ to $\nu$. Indeed, for every $\mathscr{F}\in C_b(\mathcal{T};\mathbb{R})$ and $r>0$ we have
\begin{align*}
 \int_{\mathcal{T} } \mathscr{F} (\tau) \,d\widetilde{\mathcal{P}}_{S_r( \widetilde{u}, \widetilde{B})}(\tau)&= \int_{\widetilde{\Omega} }\big[(\mathscr{F}\circ S_r)(\widetilde{u},\widetilde{B})\big](\widetilde{\omega})\, d\widetilde{\mathcal{P}}(\widetilde{\omega})\\
  &= \int_{\mathcal{T} } (\mathscr{F} \circ S_r)(\tau) \,d\nu(\tau) \\
&=\lim_{n\to \infty}\frac{1}{T_n} \int_0^{T_n}\int_{\mathcal{T} } (\mathscr{F} \circ S_r)(\tau)\,d\mathcal{P}_{S_s(u,B)}(\tau)\, ds\\
&=\lim_{n\to \infty}\frac{1}{T_n} \int_0^{T_n}\int_\Omega \big[(\mathscr{F} \circ S_{r+s})(u,B)\big](\omega)\,d\mathcal{P}(\omega)\, ds\\
&=\lim_{n\to \infty}\frac{1}{T_n} \int_r^{T_n+r}\int_\Omega \big[(\mathscr{F} \circ S_{s})(u,B)\big](\omega)\,d\mathcal{P}(\omega)\, ds \\
&= \lim_{n\to \infty}\frac{1}{T_n} \int_0^{T_n}\int_\Omega \big[(\mathscr{F} \circ S_{s})(u,B)\big](\omega)\,d\mathcal{P}(\omega)\, ds\\
&\hspace{.5cm}-\lim_{n\to \infty}\frac{1}{T_n} \int_0^r\int_\Omega \big[(\mathscr{F} \circ S_{s})(u,B)\big](\omega)\,d\mathcal{P}(\omega)\, ds\\
&\hspace{.5cm}+\lim_{n\to \infty}\frac{1}{T_n} \int_{T_n}^{T_n+r}\int_\Omega \big[(\mathscr{F} \circ S_{s})(u,B)\big](\omega)\,d\mathcal{P}(\omega)\, ds\\
&=\int_{\mathcal{T} }\mathscr{F}(\tau)\,d\nu(\tau).
\end{align*}
Remembering that the weak limit of Borel measures is unique establishes the assertion. 

\subsection{Regularity \& Bounds} \label{Stationary Regularity and Bound}
In order to prove \eqref{Stationary Property 1} and \eqref{Stationary Property 2}, let us first define $\widetilde{z}_n$ as in \thref{Proposition SHE} with $B$ replaced by $\widetilde{B}_n$ for all $n\in \mathbb{N}$. Thanks to Itô's formula applied to the functional $(r,x)\mapsto e^{-(\Id-\Delta)(t-r)}x$ and $\widetilde{B}_n$, we know that $\widetilde{z}_n$ admits the expression 
\begin{align*}
\widetilde{z}_n(\omega,t,x)=\widetilde{B}_n(\omega,t,x)-\int_0^t (\Id-\Delta)e^{-(\Id-\Delta) (t-r)}\widetilde{B}_n(\omega,r,x)\,dr,
\end{align*}
which compels us to introduce the functional
\begin{align*}
\mathcal{Z}\colon \mathcal{T} \to \mathbb{R},\qquad \mathcal{Z}(u,B)=\Big\|u-B+\int_0^\cdot (\Id-\Delta)e^{-(\Id-\Delta)(\cdot-r)}B(r) \, dr\Big\|_{C_I W_x^{1,\iota}}^{2J}
\end{align*}
for $J\geq 1$ and some interval $I \subseteq [0,\infty)$ of length $1$. It facilitates the derivation of 
\begin{align} 
\widetilde{E}\|\widetilde{u}_n-\widetilde{z}_n\|^{2J}_{C_I W_x^{1,\iota}}&= \int_{\mathcal{T}} \mathcal{Z}(\tau)\, d\widetilde{\mathcal{P}}_{(\widetilde{u}_n,\widetilde{B}_n)}(\tau)\label{from u to statioanry}\\
&= \frac{1}{T_n} \int_0^{T_n}\int_{\mathcal{T}} \mathcal{Z}(\tau)\, d\mathcal{P}_{S_s(u,B)}(\tau) \, ds\notag\\
&=\frac{1}{T_n} \int_0^{T_n}\int_{\Omega} \mathcal{Z}\big(S_s(u,B)(\omega)\big)\, d\mathcal{P}(\omega) \, ds\notag\\
&=\frac{1}{T_n}\int_0^{T_n}\E\big[(\mathcal{Z}\circ S_s) (u,B)\big]\, ds \notag.
\end{align}
As a consequence of \thref{Remark Weak Solution} we conclude
\begin{align*}
\widetilde{E}\|\widetilde{u}_n-\widetilde{z}_n\|^{2J}_{C_I W_x^{1,\iota}}&\leq \epsilon^{2J},
\end{align*}
with $\epsilon$ as defined therein.\\
Proceeding in a similar way with $\mathcal{Z}$ replaced by
\begin{align*}
(u,B) \mapsto \|u\|^{2J}_{C_{[0,N]}H_x^{\gamma}}+\|u\|^{2J}_{C_{[0,N]}^{0,\gamma}L_x^2}+\|u\|^{2J}_{C_{[0,N]}^{0,\gamma}W_x^{1+\gamma,\iota}}
\end{align*}
 additionally yields 
\begin{align} \label{Banach--Alaoglu}
&\widetilde{\E}\|\widetilde{u}_n\|^{2J}_{C_{[0,N]}H_x^{\gamma}}+\widetilde{\E}\|\widetilde{u}_n\|^{2J}_{C_{[0,N]}^{0,\gamma}L_x^2}+\widetilde{\E}\|\widetilde{u}_n\|^{2J}_{C_{[0,N]}^{0,\gamma}W_x^{1+\gamma,\iota}}\\
&\hspace{.5cm}= \frac{1}{T_n} \int_0^{T_n}\left(\E\|S^1_su\|^{2J}_{C_{[0,N]}H_x^{\gamma}}+\E\|S^1_su\|^{2J}_{C_{[0,N]}^{0,\gamma}L_x^2}+\E\|S_s^1u\|^{2J}_{C_{[0,N]}^{0,\gamma}W_x^{1+\gamma,\iota}}\right)\, ds\notag\\
&\hspace{.5cm}\leq \frac{1}{T_n} \int_0^{T_n} \sum_{i=s}^{N+s-1}\left(\|u\|^{2J}_{L_\Omega^{2J} C_{[i,i+1]}H_x^{\gamma}}+\|u\|^{2J}_{L_\Omega^{2J} C_{[i,i+1]}^{0,\gamma}L_x^2}+\|u\|^{2J}_{L_\Omega^{2J}C_{[i,i+1]}^{0,\gamma}W_x^{1+\gamma,\iota}}\right)\, ds\notag \\
&\hspace{.5cm}\lesssim N \notag.
\end{align}
While \eqref{Stationary Property 1} follows finally from a weak$^\ast$ lower semicontinuity argument (details are postponed to Appendix~\ref{Weak Star Lower Semicontinuity}), \eqref{Stationary Property 2} relies on Fatou's lemma. Strictly speaking, Fatou's lemma would directly imply \eqref{Stationary Property 2} if $(\widetilde{z}_n)_{n\in \mathbb{N}}$ converges $\widetilde{\mathcal{P}}$-almost surely to $\widetilde{z}$ in $C\big([0,\infty); W^{1,\iota}\big( \mathbb{T}^d;\mathbb{R}^d\big)\big)$. In fact, their corresponding functionals, fixed in time, converge, thanks to the embeddings \linebreak $L^2 \cong \left(L^2\right)^\prime \subseteq \left(H^2\right)^\prime\cong H^{-2}$ and (5.23) on p. 135 in \cite{DPZ14}, in  $H^{-2}\left(\mathbb{T}^d;\mathbb{R}^d\right) $ $ \widetilde{\mathcal{P}}$-almost surely. Since we have $H^{-2} \cong \left(H^2\right)^\prime \cong H^2 \subseteq W^{1,\iota}$, the convergence of $(\widetilde{z}_n)_{n\in \mathbb{N}}$ takes actually place in the desired space. 
\subsection{Energy Inequality} 
The final step in confirming that $\widetilde{u}$ qualifies as a Leray--Hopf solution is to ensure that the energy inequality holds. For this purpose, we intend to use the stationarity of $\widetilde{u}$. The energy-related functional
\begin{align*}
u\mapsto \|u(t)\|_{L_x^2}^2+2\int_0^{t} \int_{\mathbb{T}^d} \mathcal{A}(Du(r,x)) : Du(r,x)\,dx \,dr-\tr(G)t
\end{align*}
is for fixed times $t\geq 0$, however, not bounded in $\mathbb{R}$, so that we shall obliged stop the process when it takes larges values. Strictly speaking, for some arbitrary large $K>0$ and $t\geq 0$, let us introduce the functional 
\begin{gather*}
\mathcal{E}^K_t\colon \mathcal{T}\to \mathbb{R},\\  \mathcal{E}^K_t\{(u,B)\}:=\|u(t)\|_{L_x^2}^2\wedge K+2\int_0^{t} \int_{\mathbb{T}^d} \mathcal{A}(Du(r,x)) : Du(r,x)\,dx \,dr \wedge K-\tr(G)t,
\end{gather*}
which is continuous and bounded in $\mathbb{R}$. If we extend all functions in $\mathcal{T}$ to negative times by taking their values at time $t=0$, the same arguments as in Section \ref{Stationarity} yield
\begin{align*}
\widetilde{\E}[\mathcal{E}^K_t\{(\widetilde{u},\widetilde{B})\}]=\widetilde{\E}\Big[\|u(0)\|_{L_x^2}^2\wedge K\Big]+2\widetilde{\E}\bigg[\int_{\mathbb{T}^d} \mathcal{A}(Du(0,x)) : Du(0,x)\,dx\, t \wedge K\bigg]-\tr(G)t,
\end{align*} 
which amounts thanks to the theorem of monotone convergence to
\begin{align*}
\widetilde{\mathscr{E}}\{\widetilde{u}\}(t)&:=\widetilde{\E}\|\widetilde{u}(t)\|_{L_x^2}^2+2\widetilde{\E}\bigg[\int_0^{t} \int_{\mathbb{T}^d} \mathcal{A}(D\widetilde{u}(r,x)) \colon D\widetilde{u} (r,x)\,dx \,dr\bigg]-\tr(G)t\\
&=\widetilde{\E}\|\widetilde{u}(0)\|_{L_x^2}^2+2\widetilde{\E}\bigg[ \int_{\mathbb{T}^d} \mathcal{A}(D\widetilde{u}(0,x)) \colon D\widetilde{u} (0,x)\,dx \bigg]t-\tr(G)t.
\end{align*}
Furthermore, Fatou's lemma teaches us
\begin{align*}
\widetilde{\mathscr{E}}\{\widetilde{u}\}(t)
\leq \widetilde{\E}\|\widetilde{u}(0)\|_{L_x^2}^2+\left(\liminf_{n \to \infty}2\widetilde{\E}\bigg[ \int_{\mathbb{T}^d} \mathcal{A}(D\widetilde{u}_n(0,x)) \colon D\widetilde{u}_n (0,x)\,dx \bigg]-\tr(G)\right)t.
\end{align*}
In order to control the nonlinear term, we use similar arguments as in \eqref{from u to statioanry} to obtain
\begin{align*}
2\widetilde{\E}\left[\int_{\mathbb{T}^d} \mathcal{A}(D\widetilde{u}_n(0,x)) \colon D\widetilde{u}_n (0,x)\,dx \right]&= \frac{2}{T_n}\E\left[\int_0^{T_n}\int_{\mathbb{T}^d} \mathcal{A}(Du(s,x)) \colon Du(s,x)\,dx \, ds \right].
\end{align*}
Remembering that the energy inequality at the level of $u$ is satisfied, allows us to derive 
\begin{align*}
&2\widetilde{\E}\left[\int_{\mathbb{T}^d} \mathcal{A}(D\widetilde{u}_n(0,x)) \colon D\widetilde{u}_n (0,x)\,dx \right] \\
&\hspace{.5cm}\leq \frac{1}{T_n} \E\|u(0)\|_{L_x^2}^2- \frac{1}{T_n} \E\|u(T_n)\|_{L_x^2}^2 +\frac{2}{T_n}\E\left[\int_{\mathbb{T}^d} \mathcal{A}(Du(0,x)) \colon Du(0,x)\,dx \right]+ \tr(G),
\end{align*}
resulting in
\begin{align*}
\widetilde{\mathscr{E}}\{\widetilde{u}\}(t)
\leq \widetilde{\E}\|\widetilde{u}(0)\|_{L_x^2}^2\leq \widetilde{\mathscr{E}}\{\widetilde{u}\}(0). 
\end{align*}

\subsection{Non-Uniqueness } \label{Non-Uniqueness via Help Energy}
Non-uniqueness of solutions relies on different choices of the underlying energy $e$. To be more precise, let us first follow the approach as in \eqref{from u to statioanry} to derive 
\begin{align*} 
\widetilde{\E}\|\widetilde{u}_n(t)\|_{L_x^2}^2&=\frac{1}{T_n} \int_0^{T_n}\E\|S_s^1u(t)\|_{L_x^2}^2\,ds\\
&=  \frac{1}{T_n} \int_0^{T_n} \mathscr{H}\{u\}(t+s)\,ds\\
&\hspace{.5cm} - \frac{1}{T_n} \int_0^{T_n}\frac{2H}{t+s+H}\E\bigg[\int_0^{t+s} \int_{\mathbb{T}^d} \mathcal{A}(Du(r,x)) \colon Du (r,x)\,dx \,dr\bigg]\,ds 
\end{align*}
for any fixed $t\geq 0$ and $n \in \mathbb{N}$. In view of \thref{Lemma 2.4}, Hölder's inequality, \thref{Remark Weak Solution} and \thref{Proposition SHE} the nonlinear term is controlled as
\begin{align*} 
\begin{split}
&\frac{2H}{t+s+H}\E\bigg[\int_0^{t+s} \int_{\mathbb{T}^d} \mathcal{A}(Du(r,x)) \colon Du(r,x)\,dx \,dr\bigg] \\
&\lesssim \frac{2H}{t+s+H}\E\bigg[\int_0^{t+s} \int_{\mathbb{T}^d} \|Du(r,x)\|_F \left(1+\|Du(r,x)\|^{\iota-1}_F\right)\,dx \,dr\bigg]\\
&\lesssim  \sup_{|I|=1}\, \left(\E\|u\|_{C_I W_x^{1,\iota}}+\E\|u\|_{C_I W_x^{1,\iota}}^{\iota}\right)\\
& \leq N_o,
\end{split}
\end{align*}
for some $N_o>0$, which is independent of the underlying energy $e$.\\
Moreover, we know from \thref{Theorem Weak Solution} that the auxiliary energy $\mathscr{H}\{u\}$ coincides with the prescribed energy $e$ therein. Let us stipulate that the energy is bounded from below by $N_o$. Strictly speaking, let us require $e(t)\geq \underline{e}>N_0$, which allows to infer
\begin{align}\label{Stationary Non-uniqueness}
\widetilde{\E}\|\widetilde{u}_n(t)\|_{L_x^2}^2\in [\underline{e}-N_o,\bar{e}].
\end{align}
Suppose for the moment that \eqref{Stationary Non-uniqueness} is stable under limitation, i.e. that
\begin{align} \label{Help Energy in Interval}
\widetilde{\E}\|\widetilde{u}(t)\|_{L_x^2}^2\in [\underline{e}-N_o,\bar{e}]
\end{align}
holds as well, and take another stationary solution $\widetilde{u}^\prime$ to the power-law system with underlying energy $e^\prime $, complying with $e^\prime(t)\in [\underline{e}^\prime,\bar{e}^\prime] $ for some $0<\underline{e}^\prime \leq \bar{e}^\prime$ with $\bar{e}<\underline{e}^\prime-N_o$ and all $t\geq 0$.\footnote{In particular, in the case of stationary solutions, where the energy takes the form $e(t)= H \mathrm{tr}(G)$, one must choose $H^\prime>0$ large enough to ensure that $e^{\prime}(t)=H^{\prime}\mathrm{tr}(G)$ satisfies $0 < \bar{e} < \underline{e}^{\prime}-N_0.$} In other words, we chose the auxiliary energies of the corresponding solutions in a way that their mean kinetic energies lie within disjoint intervals. Consequently, the velocities $\widetilde{u}$ and $\widetilde{u}^\prime$ distinguish, meaning that we get infinitely many stationary solutions to~\eqref{PLF}. \\
It remains to show that it is indeed possible to pass to the limit in \eqref{Stationary Non-uniqueness}. Towards this goal, we proceed again as in \eqref{from u to statioanry} and involve \thref{Remark Weak Solution} to get
\begin{align*} 
\widetilde{\mathbf{E}}\|\widetilde{u}_n(t)\|^{2J}_{L_x^2}= \frac{1}{T_n} \int_0^{T_n} \E\|S_s^1u(t)\|_{L_x^2}^{2J} \, ds \leq \sup_{|I|=1} \E\|u\|_{C_IL_x^2}^{2J}<\infty
\end{align*}
for $J\geq 1$. Accordingly, Vitali's theorem of convergence permits to establish \eqref{Help Energy in Interval}. The proof of \thref{Theorem Stationary} is therefore complete.

\end{proof}

\section{End of the Proof of \thref{Main Result}} \label{Ergodicity of Solutions}
To conclude the existence and non-uniqueness of ergodic Leray--Hopf solutions, stated in \thref{Main Result}, we follow a very standard approach via Krein--Milman's theorem (cf. \cite{HZZ25} for instance).
\begin{proof}[Proof of \thref{Main Result}]
 We already know that the set of all stationary Leray--Hopf solutions $(u,B)$, that satisfy the conditions \eqref{Stationary Property 1}, \eqref{Stationary Property 2} as well as \eqref{Help Energy in Interval}, and hence the set, consisting of corresponding laws $\mathscr{L}$, are non-empty.  Particularly, $\mathscr{L}$ is convex and owing to similar arguments as in Section \ref{Tightness}, it is also compact with respect to the weak topology.
Therefore, there exist, thanks to the Krein--Milman theorem, an extreme point $\mu \in \mathscr{L}$, which is by a standard contradiction argument (cf. Proposition 12.4 in \cite{Ph01}) even ergodic. In other words, $\mu$ is the law of an ergodic Leray--Hopf solution.\\
Non-uniqueness of these solutions follows for reasons similar to those presented in Section \ref{Non-Uniqueness via Help Energy}, which concludes the proof of \thref{Main Result}.
\end{proof}

%%%%%%%%%%%%%%%%%%%%%%%%%%%%%%%%%%%%%%%%%%%%%%
%% Single Appendix:                         %%
%%%%%%%%%%%%%%%%%%%%%%%%%%%%%%%%%%%%%%%%%%%%%%
%\begin{appendix}
%\section*{???}%% if no title is needed, leave empty \section*{}.
%\end{appendix}
%%%%%%%%%%%%%%%%%%%%%%%%%%%%%%%%%%%%%%%%%%%%%%
%% Multiple Appendixes:                     %%
%%%%%%%%%%%%%%%%%%%%%%%%%%%%%%%%%%%%%%%%%%%%%%

\appendix
\chapter{Appendix} \label{Appendix A.1} 
\noindent The appendix is organized as follows: While Appendix \ref{Appendix A.1} provides additional information about the choice of parameters discussed in Section \ref{Choice of Parameters}, Appendix \ref{Appendix A.2} recalls Proposition 3 and Lemma 3 from \cite{BMS21} and \cite{MS18}, respectively, which were frequently applied throughout the paper. Finally, Appendix \ref{Appendix A.3} presents selected proofs.\\
\section{Choice of the Parameters}
We begin by collecting the strongest conditions that have to be satisfied, along with noting their respective appearances in the relevant sections.
\subsection{Sufficient Conditions}

\renewcommand*{\arraystretch}{1.5}
\begin{table}[H]
\begin{tabularx}{\textwidth}{|l|l|X|}
\hline
\hspace{0.8cm}  \textbf{1. Key Bound}  & \hspace{0.4cm}  \textbf{3. Key Bound}  &  \hspace{1.6cm}  \textbf{4. Key Bound}  \\
\hhline{===}
$r_\perp^{-\frac{1}{2}}\lambda_{q+1}^{-\frac{1}{2}}\leq \lambda_{q+1}^{-(10d+24)\frac{\alpha}{\sigma-2\varepsilon} } $  & $(32d+80)\alpha<\sigma-2\varepsilon$ &$r_\perp^{\frac{d-1}{\iota}-\frac{d-1}{2}} r_\parallel^{\frac{1}{\iota}-\frac{1}{2}}\lambda_{q+1} \leq \lambda_{q+1}^{-(14d+36)\frac{\alpha}{\sigma-2\varepsilon} } $ \\
 & &$\alpha b > \frac{3}{2} d+5$ \\
\hline
\end{tabularx}
\caption{Essential conditions for verifying the first four key bounds.}
\end{table}
\renewcommand*{\arraystretch}{1.3}

\begin{table}[H]
\begin{tabularx}{\textwidth}{|l|l|X|}
\hline
&\hspace{2.3cm}\textbf{5. Key Bound} & \hspace{-0.1cm} \textbf{6. Key Bound}  \\
\hhline{===}
\multirow{5}{*}{Linear Error}& $r_\perp^{\frac{d-1}{p}-\frac{d-3}{2}}r_\parallel^{\frac{1}{p}-\frac{3}{2}}\mu \leq \lambda_{q+1}^{-(10d+24)\frac{\alpha}{\sigma-2\varepsilon} } $ & $A\geq 10$ \\
&$f^{\frac{d-2}{2}}(q)r_\perp^{\frac{d-1}{p}-\frac{d-1}{2}}r_\parallel^{\frac{1}{p}-\frac{1}{2}} \leq \lambda_{q+1}^{-(10d+24)\frac{\alpha}{\sigma-2\varepsilon} } $  & \\
& $f^{\frac{d-2}{2}}(q)r_\perp^{\frac{d-1}{p}-(d-1)}r_\parallel^{\frac{1}{p}-1}\mu^{-1} \leq \lambda_{q+1}^{-(10d+24)\frac{\alpha}{\sigma-2\varepsilon} } $& \\
&$\alpha b> d+ 9 $ & \\
&$\ell^{\sigma/2-\varepsilon}< \lambda_{q+1}^{-\alpha}$ & \\
\hline
\multirow{2}{*}{Nonlinear Error $1$}
& $\lambda_{q+1}^{-2\alpha} f^{1-\sigma}(q)<\lambda_{q+1}^{-\frac{2\beta b^2}{\iota-1}}$ & \\
&$(\iota-1) \alpha>2\beta b^2 $ & \\
\hline
Nonlinear Error $2$ & $f^{-\sigma}(q) < \lambda_{q+1}^{-\frac{2\beta b^2}{\iota-1}}$ & \\
\hline
\multirow{2}{*}{Oscillation Error } &$r_\perp^{\frac{d-1}{p}-d} r_\parallel^{\frac{1}{p}-1} \lambda_{q+1}^{-1} \leq \lambda_{q+1}^{-(20d+46)\frac{\alpha}{\sigma-2\varepsilon}  } $ &   \\
& $r_\perp^{\frac{d-1}{p}-(d-1)}r_\parallel^{\frac{1}{p}-1}\mu^{-1} \leq \lambda_{q+1}^{-(12d+26)\frac{\alpha}{\sigma-2\varepsilon} } $& \\
\hline
\end{tabularx}
\caption{Essential conditions for verifying the fifth and sixth key bound.}
\end{table}
\vspace{1cm}
\renewcommand*{\arraystretch}{1.3}
\begin{table}[H]
\begin{tabularx}{\textwidth}{|l|l|X|}
\hline 
& \hspace{4cm}\textbf{Control of Energy}\\
\hhline{==}
\multirow{4}{*}{I }& \eqref{Reason of H} gives rise to add the factor $\frac{H}{t+H}$ in the definition of the auxiliary energy $\mathscr{H}$  \\
& $\ell^{\sigma/2-\varepsilon} \lambda_q^{d+4}<\lambda_{q+1}^{-\alpha}$ \\
& $a\geq\big( (8N^\ast+6)LA\big)^c$ \\
& $4N^\ast+6<\alpha b$\\
\hline
\multirow{2}{*}{II \& V }& $r_\perp r_\parallel^{-1}\leq \lambda_{q+1}^{-(20d+46)\frac{\alpha}{\sigma-2\varepsilon} }$\\
&$r_\perp^{-\frac{d-1}{2}}r_\parallel^{-\frac{1}{2}} \mu^{-1} \leq \lambda_{q+1}^{-(20d+46)\frac{\alpha}{\sigma-2\varepsilon} } $ \\
\hline
IV& $f^{\frac{d-2}{2}}(q+1)r_\perp^{\frac{d-1}{p}-\frac{d-1}{2}}r_\parallel^{\frac{1}{p}-\frac{1}{2}} \leq \lambda_{q+1}^{-(2d+4)\frac{\alpha}{\sigma-2\varepsilon} } $\\
\hline
\end{tabularx}
\caption{Essential conditions for verifying the energy estimate \eqref{approximation energy}.}
\end{table}
\vspace{1cm}

\renewcommand*{\arraystretch}{1.3}
\begin{table}[H]
\begin{tabularx}{\textwidth}{|l|X|}
\hline 
&\hspace{0cm} \textbf{ Regularity of Solutions }  \\
\hhline{==}
\eqref{target space}&$\beta \leq \frac{(2d+6)(\sigma -2\varepsilon)(4-\sigma+2\varepsilon)}{(2-\sigma+2\varepsilon)^2}$\\
\hline
\eqref{target series}& $a\geq (32JLA)^c$\\
\hline
\end{tabularx}
\caption{Essential conditions for increasing the regularity of the constructed solution in Section \ref{Regularity and Bound}.}
\end{table}

\subsection{Condition Cascade}
The above considerations result in requiring the constraint \eqref{essential bound a} -- \eqref{essential bound h}. To determine the precise parameters $\varsigma, \, \tau$ and $p> 1$ let us first set 
\begin{align*}
r_\perp:= \lambda_{q+1}^{-s_1}, \qquad r_\parallel:=\lambda_{q+1}^{-s_2}, \qquad \mu:=\lambda_{q+1}^{s_3}, \qquad f(q+1):= \lambda_{q+1}^{s_4}
\end{align*}
for some $s_1,s_2\in(0,1),\, s_3>0$ and $s_4\in\Big( \frac{2\beta b^3}{(\iota-1)\sigma},\frac{2\alpha b (\iota -1 )-2\beta b^3}{(\iota -1) (1-\sigma)} \Big)$. The choice of $s_4$ additionally requires $(\iota-1)\sigma\alpha  > \beta b^2$.\\
These parameters satisfy \eqref{essential bound a} -- \eqref{essential bound h} if and only if they adhere to the bounds listed in the first generation of the table below. To determine the precise values, we proceed by merging conditions incrementally.\\
In the second generation, we merge \eqref{essential bound c} with \eqref{essential bound f} and \eqref{essential bound d} with \eqref{essential bound e}, which necessitates imposing additional constraints on $s_3$, as detailed in the corresponding row. The remaining conditions are preserved.\\
In the third generation, we retain the conditions \eqref{essential bound a}, \eqref{essential bound b} and \eqref{essential bound g}, while merging \eqref{essential bound c}, \eqref{essential bound d}, \eqref{essential bound e} and \eqref{essential bound f} with \eqref{essential bound h}. This leads to further restrictions on the parameters. To simplify the process, we impose the specific choice $s_3:= \frac{d}{2}s_1$, which affects the remaining constraints, as shown in the corresponding row. Continuing in this manner, we reach the final set of parameter values in the fifth generation.
\medskip
\renewcommand*{\arraystretch}{2}
\begingroup
\setlength{\LTleft}{-20cm plus -1fill}
\setlength{\LTright}{\LTleft}
\begin{longtable}{|l|l|l|}
\hline
\large \textbf{Generation} &\large \textbf{Equation} & \large\textbf{Condition}\\
 \hhline{===}
\multirow{8}{*}{\large\textbf{1. Generation}}&\eqref{essential bound b}&$\iota\leq \frac{2\big[(d-1)s_1+s_2 \big]}{(d-1)s_1+s_2+2+(28d+72)\frac{\alpha}{\sigma-2\varepsilon} }$ \\ 
\cline{2-3}
\noalign{\vskip-3.5\tabcolsep  }
\\ \cline{2-3}
&\eqref{essential bound a} &  $s_1\leq 1-(20d+48)\frac{\alpha}{\sigma-2\varepsilon} $ \\
\cline{2-3}
\noalign{\vskip-3.5\tabcolsep  }
\\ \cline{2-3}
&\eqref{essential bound g} &  $s_2\leq s_1-(20d+46)\frac{\alpha}{\sigma-2\varepsilon} $ \\
\cline{2-3}
\noalign{\vskip-3.5\tabcolsep  }
\\ \cline{2-3}
&\eqref{essential bound h} &  $s_3\geq \frac{d-1}{2}s_1+\frac{1}{2} s_2+(20d+46)\frac{\alpha}{\sigma-2\varepsilon} $ \\
\cline{2-3}
\noalign{\vskip-3.5\tabcolsep  }
\\ \cline{2-3}
&\eqref{essential bound f} & $p \leq \frac{(d-1)s_1+s_2}{ds_1+s_2-1+(20d+46)\frac{\alpha}{\sigma-2\varepsilon} }$  \\
\cline{2-3}
&\eqref{essential bound c} & $p \leq \frac{2\big[(d-1)s_1+s_2 \big]}{(d-3)s_1+3s_2+2s_3+(20d+48)\frac{\alpha}{\sigma-2\varepsilon} }$  \\
\cline{2-3}
&\eqref{essential bound d} & $p \leq \frac{2 \big[(d-1)s_1 +s_2\big]}{(d-1)s_1+s_2+(d-2)s_4+(20d+48) \frac{\alpha}{\sigma-2\varepsilon} }$  \\
\cline{2-3}
&\eqref{essential bound e} & $p \leq \frac{2 \big[(d-1)s_1 +s_2\big]}{2\big[(d-1)s_1+s_2\big]-2s_3+(d-2)s_4+(24d+52) \frac{\alpha}{\sigma-2\varepsilon} }$  \\
 \cline{1-3}

\cline{1-3}
\multicolumn{3}{c}{\includegraphics[scale=1.3]{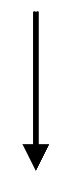}} \\
\hline
\multirow{6}{*}{\large\textbf{2. Generation}}
&\eqref{essential bound b}&$\iota\leq \frac{2\big[(d-1)s_1+s_2 \big]}{(d-1)s_1+s_2+2+(28d+72)\frac{\alpha}{\sigma-2\varepsilon} }$ \\ 
\cline{2-3}
\noalign{\vskip-3.5\tabcolsep  }
\\ \cline{2-3}
&\eqref{essential bound a} &  $s_1\leq 1-(20d+48)\frac{\alpha}{\sigma-2\varepsilon} $ \\
\cline{2-3}
\noalign{\vskip-3.5\tabcolsep  }
\\ \cline{2-3}
&\eqref{essential bound g} &  $s_2\leq s_1-(20d+46) \frac{\alpha}{\sigma-2\varepsilon} $ \\
\cline{2-3}
\noalign{\vskip-3.5\tabcolsep  }
\\ \cline{2-3}
&\eqref{essential bound h} &  $s_3\geq \frac{d-1}{2}s_1+\frac{1}{2} s_2+(20d+46) \frac{\alpha}{\sigma-2\varepsilon} $ \\
\cline{2-3}
\noalign{\vskip-3.5\tabcolsep  }
\\ \cline{2-3}
&\eqref{essential bound c}, \eqref{essential bound f}& $p \leq \frac{(d-1)s_1+s_2}{ds_1+s_2-1+(20d+46)\frac{\alpha}{\sigma-2\varepsilon} }, \quad s_3\leq \frac{d+3}{2}s_1-\frac{1}{2}s_2-1+(10d+22)\frac{\alpha}{\sigma-2\varepsilon} $  \\
\cline{2-3}
&\eqref{essential bound d}, \eqref{essential bound e} & $p \leq \frac{2 \big[(d-1)s_1 +s_2\big]}{(d-1)s_1+s_2+(d-2)s_4+(20d+48) \frac{\alpha}{\sigma-2\varepsilon} }, \quad s_3\geq \frac{d-1}{2}s_1+\frac{1}{2}s_2+(2d+2)\frac{\alpha}{\sigma-2\varepsilon}$  \\
 \cline{1-3}
\multicolumn{3}{c}{\raisebox{-.1cm}[3.2cm][.6cm]{\includegraphics[scale=1.3]{Pfeil.png}}} \\ 

\hline
\multirow{6}{*}{\large\textbf{3. Generation}}&\eqref{essential bound b}&$\iota\leq \frac{2\big[(d-1)s_1+s_2 \big]}{(d-1)s_1+s_2+2+(28d+72)\frac{\alpha}{\sigma-2\varepsilon} }$ \\ 
\cline{2-3}
\noalign{\vskip-3.5\tabcolsep  }
\\ \cline{2-3}
&\eqref{essential bound a} &  $s_1\leq 1-(20d+48)\frac{\alpha}{\sigma-2\varepsilon} $ \\
\cline{2-3}
\noalign{\vskip-3.5\tabcolsep  }
\\ \cline{2-3}
&\eqref{essential bound g} &  $s_2\leq s_1-(20d+46) \frac{\alpha}{\sigma-2\varepsilon} $ \\
\cline{2-3}
\noalign{\vskip-3.5\tabcolsep  }
\\ \cline{2-3}
&\eqref{essential bound c}, \eqref{essential bound d},& $  s_1\leq 1-(30d+68)\frac{\alpha}{\sigma-2\varepsilon}, \quad s_2\leq 3s_1-2+(20d+44) \frac{\alpha}{\sigma-2\varepsilon},$  \\
& \eqref{essential bound e}, \eqref{essential bound f},  &$ s_3:=\frac{d}{2}s_1,  $  \\
&\eqref{essential bound h}& $p \leq \min \left\{\frac{(d-1)s_1+s_2}{ds_1+s_2-1+(20d+46)\frac{\alpha}{\sigma-2\varepsilon} }, \frac{2 \big[(d-1)s_1 +s_2\big]}{(d-1)s_1+s_2+(d-2)s_4+(20d+48) \frac{\alpha}{\sigma-2\varepsilon} }\right\}$  \\
 \cline{1-3}
\multicolumn{3}{c}{\raisebox{-.1cm}[3.2cm][.6cm]{\includegraphics[scale=1.3]{Pfeil.png}}} \\ 
 
 \hline

  \multirow{5}{*}{\large\textbf{4. Generation}}&\eqref{essential bound b}&$\iota\leq \frac{2\big[(d-1)s_1+s_2 \big]}{(d-1)s_1+s_2+2+(28d+72)\frac{\alpha}{\sigma-2\varepsilon} }$ \\ 
\cline{2-3}
\noalign{\vskip-3.5\tabcolsep  }
\\ \cline{2-3}
&\eqref{essential bound a}, \eqref{essential bound c},  &  $\frac{2}{3}-\frac{20d+44}{3}\frac{\alpha}{\sigma-2\varepsilon} <s_1\leq 1-(30d+68)\frac{\alpha}{\sigma-2\varepsilon}$,\\
&\eqref{essential bound d}, \eqref{essential bound e}, & $ 0<s_2\leq 3s_1-2+(20d+44) \frac{\alpha}{\sigma-2\varepsilon}, $ \\
& \eqref{essential bound f}, \eqref{essential bound g},& $ s_3:=\frac{d}{2}s_1, \qquad (70d+160)\alpha<\sigma-2\varepsilon,$  \\
&\eqref{essential bound h}& $p \leq \min \left\{\frac{(d-1)s_1+s_2}{ds_1+s_2-1+(20d+46)\frac{\alpha}{\sigma-2\varepsilon} }, \frac{2 \big[(d-1)s_1 +s_2\big]}{(d-1)s_1+s_2+(d-2)s_4+(20d+48) \frac{\alpha}{\sigma-2\varepsilon} }\right\}$  \\
 \cline{1-3}
\multicolumn{3}{c}{\raisebox{-.1cm}[3.2cm][.6cm]{\includegraphics[scale=1.3]{Pfeil.png}}}  \\ 
\hline

  \multirow{5}{*}{\large\textbf{5. Generation}}&\eqref{essential bound b}&$\iota\leq \frac{2\big[(d-1)s_1+s_2 \big]}{(d-1)s_1+s_2+2+(28d+72)\frac{\alpha}{\sigma-2\varepsilon} }$ \\ 
\cline{2-3}
\noalign{\vskip-3.5\tabcolsep  }
\\ \cline{2-3}
&\eqref{essential bound a}, \eqref{essential bound c},  &  $\frac{2}{3}+\frac{4}{3}\frac{\alpha}{\sigma-2\varepsilon}  \leq  s_1\leq 1-(30d+68)\frac{\alpha}{\sigma-2\varepsilon} ,$ \\
&\eqref{essential bound d}, \eqref{essential bound e}, & $  (20d+48)\frac{\alpha}{\sigma-2\varepsilon}  \leq s_2\leq 3s_1-2+(20d+44)\frac{\alpha}{\sigma-2\varepsilon}, $ \\
&\eqref{essential bound f}, \eqref{essential bound g}, & $ s_3:=\frac{d}{2}s_1,\qquad s_4<s_1 ,\qquad (90d+208)\alpha  \leq \sigma-2\varepsilon,$  \\
&\eqref{essential bound h}& $1<p \leq \min\left\{ \frac{(d-1)s_1+s_2}{ds_1+s_2-1+(20d+46)\frac{\alpha}{\sigma-2\varepsilon} }, \frac{2 \big[(d-1)s_1 +s_2\big]}{(d-1)s_1+s_2+(d-2)s_4+(20d+48) \frac{\alpha}{\sigma-2\varepsilon} }\right\}$  \\
 \cline{1-3}
\end{longtable} 
\endgroup
\noindent Observe that the condition $s_4<s_1$ necessitates a final decrease in $\beta$, as $(\iota-1)\sigma \alpha>2\beta b^3$ (cf. \eqref{frequency z}).

\section{Essential Lemmata} \label{Appendix A.2}
In this section, we recall two lemmata from \cite{BMS21} and \cite{MS18}, which play a crucial role in the development of this paper.
\pagebreak
\begin{lemma}\thlabel{Lemma A.1} 
For $d\geq 3$ and a finite set of directions $\Lambda \subseteq \mathbb{R}^d $ there exists some $(\alpha_\xi)_{\xi \in \Lambda}$ and $\rho>0$ such that the periodic tubes 
\begin{align*}
\left(B_\rho (0)+\alpha_\xi+\{s \xi\}_{s\in \mathbb{R}}+2\pi\mathbb{Z}^d\right)_{\xi \in \Lambda}
\end{align*}
are mutually disjoint.
\end{lemma}
\begin{proof}
See \cite{BMS21}, p. 9, Lemma 3.
\end{proof}

\begin{lemma}[Improved Hölder inequality] \thlabel{Lemma A.2}
For any $f,g \in C^\infty\left( \mathbb{T}^d;\mathbb{R}\right)$, all $\kappa \in \mathbb{N}$ and $1\leq p \leq \infty$ it holds
\begin{align*}
\|fg\|_{L_x^p}\lesssim \left(\|f\|_{L_x^p}+ \kappa^{-1/p}\|f\|_{C_x^1} \right) \bigg\|g\Big( \frac{\cdot}{\kappa}\Big)\bigg\|_{L_x^p}.
\end{align*}
\end{lemma}
\begin{proof}
See \cite{MS18}, p. 11, Lemma 2.1.
\end{proof}

 \section{Proofs and Calculations} \label{Appendix A.3} 
 \subsection{Tightness Result} \label{Appendix Tightness}

\begin{proof}[Tightness Result in Section \ref{Tightness u}]~\\
As a consequence of the considerations taken in Section \ref{Tightness u}, we find for an arbitrary sequence $(f_n)_{n \in \mathbb{N}}$ in $K_1$ a subsequence $(f_n^{(1)})_{n \in \mathbb{N}}$, so that $(f_{n_{|_{[0,1]}}}^{(1)})_{n \in \mathbb{N}}\subseteq K_1^1$ converges in $\mathcal{T}_1^1$. From $(f_n^{(1)})_{n \in \mathbb{N}}$ we collect a further subsequence $(f_n^{(2)})_{n \in \mathbb{N}}$ so that also $(f_{n_{|_{[0,2]}}}^{(2)})_{n \in \mathbb{N}}\subseteq K_1^2$ converges in $\mathcal{T}_1^2$. Continuing in this way, we construct a diagonal sequence $(d_n)_{n \in \mathbb{N}}$, that is given by $d_n:=f_n^{(n)}$ for all $n \in \mathbb{N}$ (cf. Figure \ref{Diagonal}).\\
 \begin{figure}[H]
\begin{center}
\includegraphics[scale=0.13]{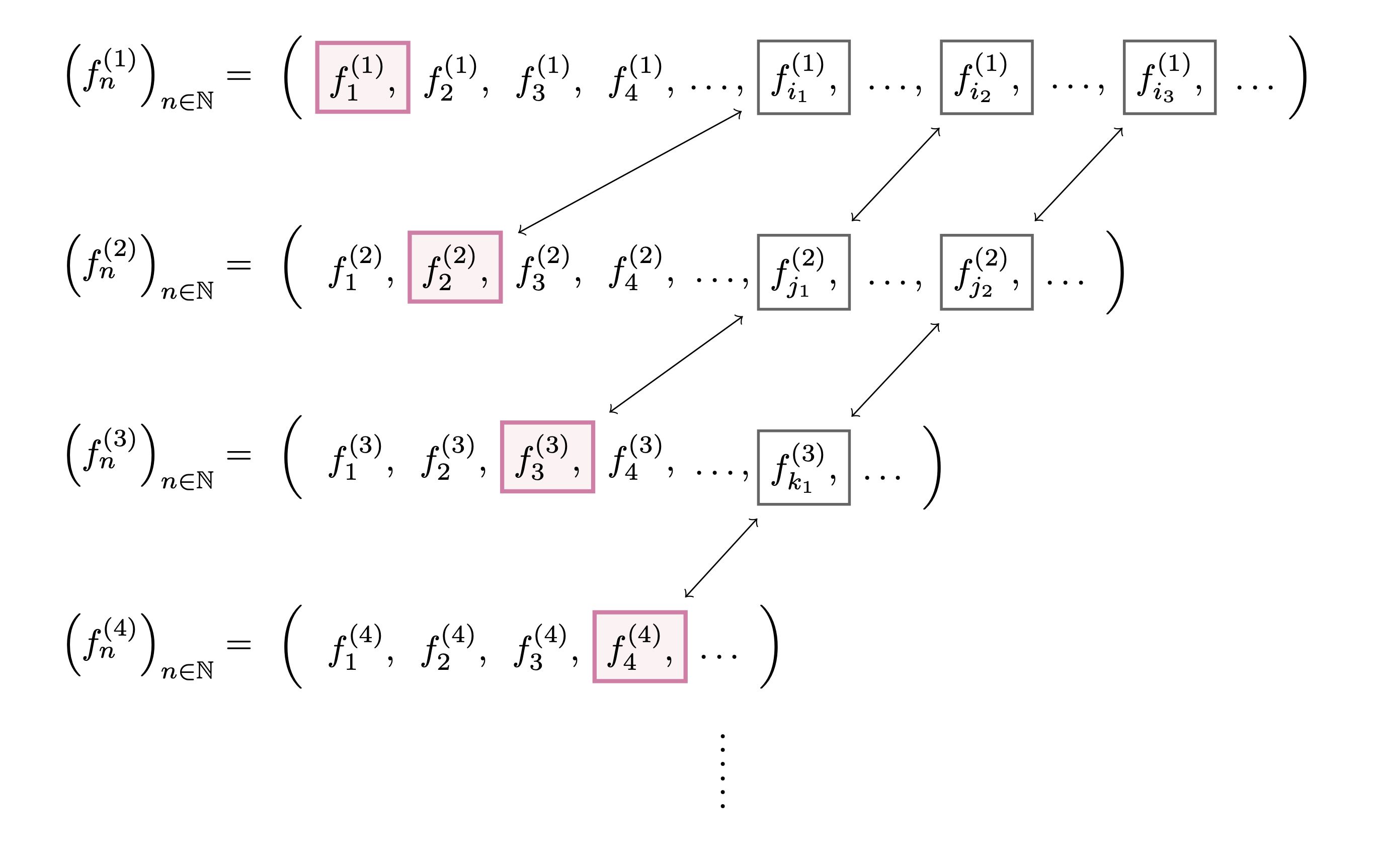}
\caption{Construction of the diagonal sequence $\left( d_n\right)_{n \in \mathbb{N}}$, which consists of the elements in the pink boxes.}
\label{Diagonal}
\end{center}
\end{figure}
\pagebreak
 \noindent That means for every $N \in \mathbb{N}$ and $\epsilon>0$ there exists some $\widetilde{N}\in \mathbb{N}$, so that 
 \begin{align*}
 \|f_{i_1}^{(N)}-f_{i_2}^{(N)}\|_{C_{[0,N]}L_x^2}+ \|f_{i_1}^{(N)}-f_{i_2}^{(N)}\|_{C_{[0,N]}W_x^{1,\iota}} < \epsilon
 \end{align*}
 holds true for all $i_2\geq i_1 \geq \widetilde{N}$. By possibly increasing $ i_1$ and $i_2$, $f_{i_1}^{(N)}$ and $f_{i_2}^{(N)}$ especially belong to $(d_n)_{n \in \mathbb{N}}$, say they are the $\ell$-th and $m$-th elements, i.e $f_{i_1}^{(N)}:=d_\ell$ and $f_{i_2}^{(N)}=d_m$ for some $\ell, m \in \mathbb{N}$. Hence we even have
 \begin{align*}
 \|d_\ell-d_m\|_{C_{[0,N]}L_x^2}+ \|d_\ell-d_m\|_{C_{[0,N]}W_x^{1,\iota}}  < \epsilon.
 \end{align*}
 Remembering that $N$ was chosen arbitrary large and that $\mathcal{T}_1$ is complete, the diagonal sequence $(d_n)_{n \in \mathbb{N}}$ converges to some limit in $\mathcal{T}_1$, which ascertains that $K_1$ is compact in $\mathcal{T}_1$.

\noindent 
\end{proof}

\subsection{Weak$^\ast$ Lower Semicontinuity} \label{Weak Star Lower Semicontinuity}

\begin{proof}[Weak$^\ast$ Lower Semicontinuity argument in Section \ref{Stationary Regularity and Bound}]
To establish \eqref{Stationary Property 1}, we invoked a weak$^\ast$ lower semicontinuity argument in Section \ref{Stationary Regularity and Bound}, with a more rigorous treatment provided in the subsequent discussion. \\
Let $\langle\widetilde{u}_n,\cdot\rangle$ be the corresponding functional to $\widetilde{u}_n$ on
\begin{align*} 
Y:=L^{2J}\big(\widetilde{\Omega};C\big([0,N];H^\gamma\big( \mathbb{T}^d;\mathbb{R}^d\big) \big)\big)L^{2J}\big(\widetilde{\Omega};C^{0,\gamma}\big([0,N];L^{2}\big( \mathbb{T}^d;\mathbb{R}^d\big)\cap W^{1+\gamma,\iota}\big( \mathbb{T}^d;\mathbb{R}^d\big)\big),
\end{align*}
in a sense that $\|\langle\widetilde{u}_n,\cdot \rangle\|_{Y^\prime}=\|\widetilde{u}_n\|_Y$ holds with $\|\cdot\|_Y$ being the natural norm associated to $Y$. In view of \eqref{Banach--Alaoglu}, Banach--Alaoglu's theorem provides therefore the existence of a subsequence $\langle\widetilde{u}_{n_k},\cdot\rangle$ that converges weakly$^\ast$ to some limit $\langle \widetilde{v}, \cdot \rangle \in Y^\prime$ on the one hand, whereas Vitali's theorem of convergence allows to conclude that $\widetilde{u}_{n_k}$ converges to $\widetilde{u}$ even in $L^{2J}\big(\widetilde{\Omega};C\big([0,N];L^2\big( \mathbb{T}^d;\mathbb{R}^d\big) \big)\big)$ on the other hand. That means their corresponding functional $\langle \widetilde{u}_{n_k}, \cdot \rangle$ converges in particular weakly$^\ast$ to $\langle \widetilde{u}, \cdot \rangle$, implying $\widetilde{v}=\widetilde{u}$.\\
 Thus taking \eqref{Banach--Alaoglu} into account once again permits to deduce
\begin{align*}
\|\widetilde{u}\|_{Y}&=\sup_{\|y\|_Y=1} |\langle\widetilde{u} ,y\rangle|=\sup_{\|y\|=1} \lim_{k \to \infty} |\langle\widetilde{u}_{n_k},y\rangle| \leq \liminf_{k\to \infty} \|\langle\widetilde{u}_{n_k},\cdot \rangle\|_{Y^\prime} =\liminf_{k\to \infty} \|\widetilde{u}_{n_k}\|_{Y}\lesssim N.
\end{align*}
\end{proof}

\end{document}